\documentclass[11pt]{amsart}

\usepackage{amsmath,amsthm,amssymb,graphics,color}
\usepackage{hyperref} 
\usepackage{xcolor}
\usepackage[english]{babel}
\usepackage[margin=1in]{geometry}
\usepackage{subcaption}
\usepackage{paralist}
\usepackage{graphicx}
\usepackage{esvect}
\usepackage{float}
\usepackage{array}
\usepackage{esvect}
\usepackage{float}
\usepackage{array}
\usepackage{tensor}
\usepackage{amsthm}
\usepackage{amsfonts}
\usepackage{bbm}

\newcommand{\Z}{\mathbb{Z}}
\newcommand{\N}{\mathbb{N}}

\newcommand{\R}{\mathbb{R}}

\newcommand{\T}{\mathbb{T}}
\newcommand{\cn}{\text{cn}}
\newcommand{\sn}{\text{sn}}
\newcommand{\am}{\text{am}}

\newcommand{\Id}{\text{Id}}

\newcommand{\supp}{\text{supp}}

\renewcommand{\phi}{\varphi}
\renewcommand{\epsilon}{\varepsilon}

\newtheorem{theo}{Theorem}[section]
\newtheorem{prop}[theo]{Proposition}
\newtheorem{coro}[theo]{Corollary}
\newtheorem{lemm}[theo]{Lemma}
\newtheorem{conj}[theo]{Conjecture}

\theoremstyle{definition}
\newtheorem{def1}[theo]{Definition}
\theoremstyle{remark}
\newtheorem{rema}[theo]{Remark}

\newcommand{\nwc}{\newcommand}
\nwc{\Oph}{\operatorname{Op}_\hbar}
\nwc{\la}{\langle}
\nwc{\ra}{\rangle}

\nwc{\mf}{\mathbf} %Latex (as in \bf not tilted math letters)
\nwc{\blds}{\boldsymbol} %Latex 
\nwc{\ml}{\mathcal} %Latex

\renewcommand{\Im}{\operatorname{Im}}
\renewcommand{\Re}{\operatorname{Re}}

\renewcommand{\d}{\partial}

\renewcommand{\phi}{\varphi}

\newcommand{\blue}[1]{{\color{blue}{#1}}}
\newcommand{\red}[1]{{\color{red}{#1}}}

\newcommand{\black}[1]{\color{black}}

\title{Robin Spectral Rigidity of the Ellipse}
\date{}

\author{Amir Vig}
\address{Department of Mathematics, UC Irvine, Irvine, CA 92697, USA} \email{bvig@uci.edu}

\begin{document}
\maketitle

\begin{abstract} In this paper, we investigate $C^1$ isospectral deformations of the ellipse with Robin boundary conditions, allowing both the Robin function and domain to deform simultaneously. We prove that if the deformations preserve the reflectional symmetries of the ellipse, then the first variation of both the domain and Robin function must vanish. If the deformation is in fact smooth, reparametrizing allows us to show that the first variation actually vanishes to infinite order. In particular, there exist no such analytic isospectral deformations. The key ingredients are a version of Hadamard's variational formula for variable Robin boundary conditions and an oscillatory integral representation of the wave trace variation which uses action angle coordinates for the billiard map. For the latter, we in fact construct an explicit parametrix for the wave propagator in the interior, microlocally near geodesic loops.
\end{abstract}

\section{Introduction}
In this paper, we prove infinitesimal spectral rigidity of the ellipse with $C^1$ deformations in both the domain and Robin boundary conditions which preserve the symmetries of the ellipse. This means that the first variations of both the domain and Robin function vanish. To make this precise, we consider the eigenvalue problem for the Laplacian. Let $\Omega_0$ be an ellipse and $\phi_\epsilon$ a $C^1$ family of diffeomorphisms defined in a neighborhood of $\Omega_0$ for $0 \leq  \epsilon < \epsilon_0$, such that $\phi_0 = \Id$.  Denote $\Omega_\epsilon = \phi_\epsilon(\Omega_0)$ and let $K_\epsilon$ be a $C^1$ family of smooth functions on $\partial \Omega_\epsilon = \phi_\epsilon(\partial \Omega_0)$.
The PDE we are interested in is
\begin{align}\label{PDE}
\begin{cases}
- \Delta u_\epsilon = \lambda^2(\epsilon) u_\epsilon, & x \in \Omega_\epsilon,\\
\frac{\partial u_\epsilon}{\partial \nu} = K_\epsilon u_\epsilon, & x \in \partial \Omega_\epsilon.
\end{cases}
\end{align}
The equation \eqref{PDE} is said to have Robin boundary condition with Robin function $K_\epsilon$. As $K_\epsilon \in C^\infty (\d \Omega_\epsilon)$,  it is shown in \cite{Taylor2} that $- \Delta$ is self adjoint on $L^2(\Omega_\epsilon)$ with densely defined domain $D = \{ u \in H^2(\Omega_\epsilon) : \d_\nu u_\epsilon = K_\epsilon u_\epsilon \}$. The analytic Fredholm theorem guarantees that \eqref{PDE} has nontrivial solutions only for a discrete collection of eigenvalues $\lambda_j^2 \in \R$, comprising the entire spectrum which we denote by $\text{Spec}(\Omega_\epsilon)$. We may assume $\phi_\epsilon(x) = x + \rho_\epsilon(x) \nu_x$ for $x \in \d \Omega_0$, where $\nu_x$ is the outward unit normal to $\d \Omega_0$ and $\rho_\epsilon$ is a smooth function on $\d \Omega_0$. We also impose the restriction that $\rho_\epsilon$ and $K_\epsilon$ should be invariant under the symmetry group of the ellipse, which is generated by reflections through the coordinate axes and is isomorphic to the Klein four-group. We denote the Laplace operator with boundary conditions above by $\Delta_{\epsilon}$ and prove the following result:
\begin{theo}\label{main}
Let $\phi_\epsilon(\Omega_0) = \Omega_\epsilon$ be a $C^1$ deformation of the ellipse through smooth domains with $\Z_2 \times \Z_2$ symmetries and $K_\epsilon$ a $C^1$ family of Robin functions with the same symmetries. If the deformation is isospectral, i.e. $\text{Spec}(\Delta_\epsilon) = \text{Spec}(\Delta_0)$ for $0 \leq \epsilon < \epsilon_0$, then $\dot{\rho} = \dot{K} = 0$.
\end{theo}
Here, we have written $\dot{\rho}(x) = \frac{d}{d \epsilon}\big|_{\epsilon = 0} \rho_{\epsilon}(x)$ and $\dot{K}(x) = \frac{d}{d \epsilon}\big|_{\epsilon = 0} K_\epsilon(x + \rho_{\epsilon}(x) \nu_x)$ to denote the first variations, which are studied in more detail in Section \ref{Variation of wave trace}. We also use the notation $\delta = \frac{d}{d \epsilon} \big|_{\epsilon = 0}$. We then say the ellipse is infinitesimally spectrally rigid through domains and Robin boundary conditions with the symmetries of an ellipse. The prefix ``infinitesimal'' means that only the first variation vanishes. If it could be shown that no nontrivial such isospectral deformations exist, we would say the ellipse is spectrally rigid amongst such domains. As a quick corollary to Theorem \ref{main}, we have:
\begin{coro}\label{analytic}
There are no nontrivial analytic isospectral deformations of the ellipse through $\Z_2 \times \Z_2$ symmetric domains and Robin functions.
\end{coro}
This follows from a simple reparametrization argument which can be found in Section 3.2 of \cite{HeZe12}. One of the main ingredients in our proof is a version of Hadamard's variational formula for variable Robin boundary conditions, which also appears to be new in the literature:
\begin{theo}\label{Hadamard}
Let $G_\epsilon$ be the Green's kernel for the eigenvalue problem with Robin boundary conditions on $\Omega_\epsilon$. Then, for $x, y \in \text{int}(\Omega_0)$, $\delta G_\epsilon(\lambda,x, y)$ is the distribution
\begin{align*}
\int_{\partial \Omega_0} -\langle \nabla_2^T G_0(\lambda, x, q), \nabla_1^T G_0 (\lambda, q, y)\rangle \dot{\rho} + (\lambda^2 \dot{\rho} & + K_0^2 \dot{\rho} + K_0 \kappa \dot{\rho} + \dot{K}) G_0(\lambda, x, q) G_0(\lambda, q, y) dq.
\end{align*}
\end{theo}
Here, $\nabla_i^T$ denotes the tangential derivative in the $i$th spatial variable and $dq$ is the natural line element on $\d \Omega_0$ inherited from the flat metric on $\R^2$. The Green's kernel or Green's function is the Schwartz kernel of the resolvent $(- \Delta - \lambda^2)^{-1}$ for $\Im \lambda^2 > 0$. In Section 3, this result is extended to variational formulas for both the even wave trace $\text{Tr} \cos(t\sqrt{-\Delta_\epsilon})$ and simple eigenvalues. In particular, we prove the following in Section \ref{Converting the singularity into an elliptic integral}:
\begin{theo}\label{wavetrace variation}
For the ellipse $\Omega_0 = \{(x,y): \frac{x^2}{a^2} + \frac{y^2}{b^2} \leq 1\}$, with boundary parametrized by $(x,y) = (a\cos \phi, b \sin \phi)$, the variation of the wave trace near a simple length $T_j \in \text{Lsp}(\Omega_0)$ corresponding to a caustic of rotation number $1/j$ is given by
\begin{align*}
\delta \text{Tr} \cos t \sqrt{-\Delta_\epsilon} &= c_j t \Re \left\{ e^{i \pi/4}(t - T_j - i0^+)^{-5/2}\right\} + L.O.T.,
\end{align*}
%where $\sigma$ is a Maslov index \red{(see Section \ref{Fourier Integral Operators})}
and the constants $c_j$ are given by
\begin{align*}
c_j = \int_0^{2\pi} P(\lambda_j) Q(\phi)  \frac{\dot{\rho}(\phi) d\phi}{\sqrt{(b^2 + (a^2 - b^2) \sin^2 \phi) - \lambda_j^2}}.
\end{align*}
Here, L.O.T. denotes lower order distributional terms and $\lambda_j$ is the parameter of the confocal ellipse
\begin{align*}
\frac{x^2}{a^2 - \lambda_j^2} + \frac{y^2}{b^2 -\lambda_j^2} = 1,
\end{align*}
to which periodic orbits of length $T_j$ are tangent. $P$ and $Q$ are certain nonzero analytic functions. Moreover, if $\dot{ \rho} = 0$, then the leading order term in the singularity expansion near $T_j$ becomes
\begin{align*}
\delta \text{Tr} \cos t \sqrt{-\Delta_\epsilon} = \widetilde{c_j} t \Re\left\{e^{\sigma i \pi /4} (t - T - i0^+)^{-1/2}\right\} + L.O.T.,
\end{align*}
where the constants $\widetilde{c_j}$ are given by
\begin{align*}
\widetilde{c_j} = \int_0^{2\pi}  \widetilde{P}(\lambda_j) \widetilde{Q}(\phi) \frac{\dot{K}(\phi) d\phi}{\sqrt{(b^2 + (a^2 - b^2)\sin^2\phi) - \lambda_j^2}}.
\end{align*}
$\widetilde{P}$ and $\widetilde{Q}$ are again certain nonzero analytic functions.
\end{theo}
\begin{rema}
In \cite{GuMe79a}, it is shown that there exists a $j_0$ such that for all $j \geq j_0$, the lengths $T_j$ of periodic orbits with rotation numbers $1/j$ are simple. In this context, simplicity of the length $T_j$ means that all periodic orbits of length $T_j$ are tangent to a single confocal conic section. These orbits have rotation number $1/j$, which means that they make $j$ reflections at the boundary and have winding number $1$. In Section \ref{Converting the singularity into an elliptic integral}, we will send $j \to \infty$ and use the constants $c_j$ and $\widetilde{c_j}$ to show that $\dot{ \rho}$ and $\dot{ K}$ vanish.
\end{rema}
%\red{We can in fact meromorphically continue the resolvent to an operator valued function on $\C$  with isolated poles at the eigenvalues via analytic Fredholm theory.}

\subsection{Schematic outline} In Section \ref{Background}, we review the relevant background on the inverse spectral problem for \eqref{PDE}. Section \ref{Variation of wave trace} is then devoted to proving a version of Hadamard's variational formula with Robin boundary conditions (Theorem \ref{Hadamard}) and extending this to variational formulas for both the localized wave trace near lengths of periodic orbits (Theorem \ref{wavetrace}) and simple eigenvalues (Theorem \ref{simple eigenvalue}). The billiard map and its special properties in the case of an ellipse are introduced in Section \ref{Billiards}. Relevant background information on dynamical systems and various notations are also introduced there which will be used throughout the paper. Section \ref{A parametrix for SR and singularity expansion} contains the most difficult material. It begins by reviewing Fourier integral operators and Chazarain's parametrix and then develops an explicit oscillatory integral parametrix for the wave propagator (Theorem \ref{parametrix with sum}) which is used to prove a variational trace formula (Theorem \ref{distribution with Aj factor}) in Section \ref{Computing the singularity in elliptical polar coordinates}. An important dynamical lemma (Lemma \ref{phase function}) on the structure of approximate geodesic loops is also stated in this section, but its proof is relegated to Section \ref{proof of lemma}. Section \ref{d omega d theta} analyzes the singularity coefficients of the variational trace formula in Theorem \ref{distribution with Aj factor} by using Jacobi elliptic function theory. Section \ref{Converting the singularity into an elliptic integral} converts the variational trace formula into an elliptic integral and then uses a method of Guillemin and Melrose to show that the first variation of both the Robin function and domain vanish, which completes the proof of Theorem \ref{main}. Section \ref{proof of lemma} contains the proof of Lemma \ref{phase function} and is broken up into several intermediate lemmas.

\section{Background}\label{Background}

The inverse spectral problem has a long history, dating back to Kac in 1967, who asked the famous question ``can one hear the shape of a drum?'' In \cite{Kac}, a positive answer to this question was obtained in the special case of a Euclidean ball by using heat invariants and the isoperimetric inequality. Almost immediately afterwards, John Milnor found an example of 16 dimensional isospectral tori which were not isometric. Sunada later generalized this example by using an algebraic method, but the question for planar domains remained open until 1992, when distinct polygonal domains in $\R^2$ were found to be isospectral in \cite{GordonWebbWolpert}. The question is still widely open for convex and/or smooth domains, although there has been significant progress. For example, it is proved in \cite{Zel09} that analytic domains with a single isometric involution are spectrally determined assuming some additional generic dynamical constraints on the length spectrum. Melrose also showed that the set of isospectral planar domains is precompact in the $C^\infty$ topology by a careful analysis of the heat invariants (see \cite{MelroseIsospectralCompactness}). This was improved in \cite{POS2}, \cite{POS1} and \cite{POS0}, where the authors proved genuine $C^\infty$ compactness of the isospectral set. This result is based on the Polyakov formula for $\zeta$-regularized determinants and applies to both bounded planar domains and closed surfaces. Recently, Hezari and Zelditch showed in \cite{HZ2} that ellipses of small eccentricity are spectrally determined amongst all smooth planar domains, which is the first positive result in such generality since Kac's original paper \cite{Kac}. Thorough surveys of the inverse spectral problem are contained in \cite{ZelditchSurvey2014}, \cite{ZelditchSurvey2}, \cite{HezariDatchevSurvey} and \cite{Melrosesurvey}.
%\\
%\\
%Weaker than spectral uniqueness is the question of spectral rigidity, i.e. the absence of isospectral deformations of smooth domains, which has also been investigated quite intensively. This began with the work of Guillemin and Kazhdan in \cite{GK1} and \cite{GK2}, where they proved spectral rigidity for closed manifolds with negative sectional curvature. This was recently generalized to Anosov surfaces in \cite{UPS}. In both cases, the authors smoothly varied the metric tensor and ultimately reduced the linearized problem to injectivity of geodesic X-ray transform. This can be seen to be at least formally analogous to our variational formula for the localized wave trace near periodic billiard orbits in Section \ref{the wave trace and eigenvalues}. For a survey of results related to integral geometry and tensor tomography, see \cite{Tensortomography}.
\\
\\
Dual to the Laplace spectrum is the so called \textit{length spectrum}, which is a discrete set of numbers containing the lengths of periodic orbits for the geodesic or billiard flow. The same inverse problem exists: can one determine a manifold up to isometry from its length spectrum? The answer is unfortunately negative, as was shown for the case of constant negative curvature in \cite{Vigneras}. However, it is conjectured by Katok and Burns that the \textit{marked length spectrum} does determine a smooth closed manifold up to isometry (\cite{BuKa85}). Here, the marked length spectrum also encodes the homotopy classes of periodic geodesics. Marked length spectral rigidity was recently shown in \cite{GuillarmouLefeuvre} for Anosov manifolds. The relationship with the Laplace spectrum is contained in the Poisson relation, which tells us that the singularities of the wave trace are a subset of the length spectrum. Assymptotic formulas near the singularities are given by the Selberg trace formula for hyperbolic surfaces (\cite{Sel56}), the Duistermaat-Guillemin trace theorem (\cite{DuGu75}) for general manifolds under a dynamical nondegeneracy condition, and a Poisson summation formula for strictly convex bounded planar domains due to Guillemin and Melrose (\cite{GuMe79b}). However, since these trace formulae involve sums over all periodic orbits of a given length, it is theoretically possible that the contributions of distinct orbits having the same length could cancel out and the wave trace is actually smooth near a point in the length spectrum. Hence, without length spectral simplicity, there is no way to deduce Laplace spectral information from the length spectrum alone.\\
\\
While some results on spectral rigidity are known in the chaotic regime (see for example \cite{GK1}, \cite{GK2} and \cite{Tensortomography}), very little is known about the completely integrable setting, in which the flow has a maximal number of conserved quantities (see Section \ref{Billiards}). In the theory of dynamical systems, a famous conjecture of Birkhoff is that the only strictly convex planar domains with completely integrable billiards are ellipses. While this remains an open conjecture, much progress has been made in the local setting. It is shown in \cite{KaSo16} that if a \textit{rationally integrable} billiard table is sufficiently close to an ellipse, then it must be an ellipse. Rational integrability means that for each integer $q \geq 1$, the billiard map has invariant curves of rotation number $1/q$, consisting entirely of periodic points. Using Aubry-Mather theory, the authors then show that ellipses are length spectrally rigid (Corollary 14 of \cite{KaSo16}).\\
\\
The ellipse is smooth, convex and has completely integrable dynamics, which makes it an interesting object to study in the context of spectral theory. In fact:
\begin{conj}[Melrose, \cite{Melrosesurvey}]
Ellipses are spectrally determined.
\end{conj}
In \cite{KaDSWe17}, the authors show that convex domains with $\Z_2$ axial symmetry which are sufficiently close to a circle in $C^k$ are length spectrally rigid. These domains include ellipses of small eccentricity. It is shown in \cite{PeSt92} that generically, convex domains have simple length spectrum and nondegenerate Poincar\'e map. Combining the results in \cite{KaDSWe17} and \cite{PeSt92}, it then follows that $\Z_2$ symmetric domains close to a circle are $\Delta$ spectrally rigid amongst a generic class of symmetric domains. In \cite{HamidRobin}, these results are extended to the Robin Laplacian, where the Robin function on the boundary is also allowed to deform through smooth functions with the same same $\Z_2$ symmetry. Our problem is similar in nature but considers ellipses of arbitrary eccentricity, which might not be close to a circle.\\
\\
The present article is inspired by \cite{HeZe12}, \cite{GuMe79a}, \cite{GuMe79b} and \cite{Pe80}. Guillemin and Melrose proved a version of the Poisson summation formula for bounded planar domains and then used this result in a subsequent article to show that for a fixed ellipse, a $\Z_2 \times \Z_2$ symmetric Robin function on the boundary is completely determined by the spectrum of the associated Laplacian. Hezari and Zelditch then proved infinitesimal spectral rigidity for Dirichlet/Neumann boundary conditions, while only letting the domain deform. In their proof, the authors used the symbol calculus in \cite{DuGu75} to compute the trace of the wave kernel near periodic transversal reflecting rays. Our problem allows both the domain and Robin function to deform simultaneously, which doesn't allow us to directly employ the results in \cite{GuMe79a} or \cite{HeZe12}.
\\
\\
The idea of the proof of Theorem \ref{main} is that if the deformation is isospectral, then the variation of the wave trace should also be zero. The Poisson relation, in this case due to Guillemin and Melrose (\cite{GuMe79a}), tells us that the singularities of the wave trace are contained in the length spectrum $\text{Lsp}(\Omega_0)$, the set of lengths of periodic trajectories for the broken bicharacteristic (billiard) flow. Using microlocal analysis and in particular, Chazarain's parametrix, we can localize the wave kernel near the periodic transversal reflecting rays. From this, we obtain a singularity expansion for the wave trace variation, a Fourier integral distribution, near the length spectrum. We then cook up an oscillatory integral which microlocally approximates this distribution by using a special phase function associated to the billiard map. To do this, we actually construct an explicit parametrix for the microlocalized wave kernel near all orbits tangent to a confocal ellipse of rotation number $1/j$ (Theorem \ref{parametrix with sum}). In particular, this involves finding all orbits making approximately one rotation with a prescribed number of reflections which connect two points in an interior neighborhood of the diagonal of the boundary (Lemma \ref{phase function}). This is of independent interest in the theory of dynamical billiards.
\\
\\
In the special case of the ellipse, we can incoorporate action angle coordinates for the billiard map, which allows us to convert the singularity expansion for the wave trace near $\text{Lsp}(\Omega_0)$ into the product of a nonzero distribution and an elliptic integral as in Theorem \ref{wavetrace variation}. Since this expansion is valid near any simple length, we can take a special sequence of caustics creeping closer and closer to the boundary. Following the ideas in \cite{GuMe79a}, we send $j \to \infty$ and analyze the coefficients $c_j$ and $\widetilde{c_j}$ in Theorem \ref{wavetrace variation}. These coefficients are analytic in the paramater $\lambda^2$ and since $\lambda_j^2 \to 0$ as $j\to \infty$, we see that they are actually flat at $\lambda = 0$. An application of the Stone-Weierstrass theorem then shows that $\dot{\rho} = 0$. Upon substituting $\dot{\rho} = 0$, we obtain a new singularity expansion for the subprincipal term with $\dot{K}$ only. The same tricks show $\dot{K} = 0$.

\section{Variation of the Wave Trace}\label{Variation of wave trace}
In this section, we derive variational formulas for the Green's function, simple eigenvalues and wave trace. The PDE \eqref{PDE} has the weak formulation
\begin{align}\label{weak formulation}
\int_{\Omega_\epsilon} \langle \nabla u_\epsilon , \nabla \phi_\epsilon^{-1 *}v \rangle - \lambda_j^2 u_\epsilon \phi_\epsilon^{-1 *} v\,dV = \int_{\Omega_\epsilon} f_\epsilon \phi_\epsilon^{-1 *}v\,dV + \int_{\partial \Omega_\epsilon} K_\epsilon u_\epsilon \phi_\epsilon^{-1 *}v \,dq_\epsilon,
\end{align}
for any $u_\epsilon \in H^2(\Omega_\epsilon)$ such that $\d_{\nu}u_\epsilon = K_\epsilon u_\epsilon$ (i.e. in the domain of $-\Delta$ with Robin boundary conditions), $v \in C^\infty(\Omega_0)$ and $f_\epsilon \in \mathcal{E}'(\text{int}(\Omega_\epsilon))$ which is an inhomogeneous term for the PDE \eqref{PDE}. Here, $dV = dx^1 \wedge dx^2$ is the volume form on $\R^2$ and $dq_\epsilon$ is the natural surface measure on $\partial \Omega_\epsilon$ induced from the Euclidian metric. We refer to the the quantity
$$
e(u,v) = \langle \nabla u, \nabla v \rangle - \lambda^2 uv
$$
as the energy density.

\subsection{Variational derivatives}
Some care is needed to differentiate the expressions above. We begin by making precise our notion of first variation, following closely the presentation in \cite{Pe80}.
\begin{def1}\label{first variation}
If $u_\epsilon \in C^1([0,\epsilon_0], \mathcal{D}'(\Omega_\epsilon))$ is a $C^1$ family of distributions, we write $\delta u$, $\delta u_\epsilon$ or $\dot{u}$ for the first variation of $u_\epsilon$ at $\epsilon = 0$, as a distribution in $\Omega_0$:
$$
\delta u_\epsilon = \frac{d}{d\epsilon} \Big|_{\epsilon = 0} u_\epsilon.
$$
\end{def1}
To simplify notation, for a single function $u$, we will oftentimes write $\delta u = \dot{u}$ and reserve the use of $\delta$ for preceding long formulas. If $\alpha \in C_0^\infty(\text{int}{\Omega_0})$ is a test function, then $\alpha \in C_0^\infty(\text{int}{\Omega_\epsilon})$ for $\epsilon \ll 1$ and we can define $\delta u$ by
$$
\delta u (\alpha) = \frac{d}{d \epsilon} \Big|_{\epsilon = 0} u_\epsilon(\alpha),
$$
i.e. the derivative of a function from $[0,\epsilon_0)$ to $\R$. The issue with this definition is that if the supports of the distributions $u_\epsilon$ actually intersect $\d \Omega_\epsilon$, then the formula above only defines $\delta u_\epsilon$ in the interior of $\Omega_0$ and not on the boundary even when $u_\epsilon$ is defined there. For instance, the Green's kernel is supported near the boundary in the setting of Robin boundary conditions.\\
\\
To resolve this issue, we follow the ideas in \cite{Pe80}, where some geometric heuristics motivate several precise definitions. The set of smooth domains in $\R^2$ is an infinite dimensional manifold $\Lambda$, on which the Lie group $\text{Diff}(\R^2)$ acts. The Lie algebra of $\text{Diff}(\R^2)$ is the space of smooth vector fields on $\R^2$. An initial domain $\Omega_0$ and a curve $\phi_\epsilon$ in $\text{Diff}(\R^2)$ generate a curve in $\Lambda$, given by $\Omega_\epsilon = \phi_\epsilon(\Omega_0)$. Hence, we can associate elements of the Lie algebra to tangent vectors at $\Omega_0$. For any given deformation of $\Omega_0$, we have an infitessimal generator $X = \frac{d \phi_\epsilon}{d\epsilon} \big|_{\epsilon = 0}$. For a fixed $s \in \R$, we can associate to each $\Omega_0 \in \Lambda$ the fiber $H^s(\Omega_0)$, which is the $L^2$ based Sobolev space of order $s$. This defines a smooth vector bundle over $\Lambda$ on which $\text{Diff}(\R^2)$ acts via pullback (diffeomorphism invariance of the Sobolev spaces). These heuristics motivate the following definition:
\begin{def1}
For a curve $(u_\epsilon, \Omega_\epsilon)$ in the above vector bundle, we define the Lie derivative to be
$$
\theta_X u_\epsilon = \frac{d}{d\epsilon} \Big|_{\epsilon = 0} \phi_\epsilon^* u_\epsilon.
$$
\end{def1}\label{Lie derivative}
We sometimes drop the $\epsilon$ subscript and write $ \theta_X u_\epsilon = \theta_X u$ for simplicity. The advantage of this definition is that it is well defined on the boundary of $\Omega_0$ for $s > 1$, via the Sobolev embedding theorem $H^s \hookrightarrow C^0$. 
\begin{lemm}\label{Lie derivative and first variation}
If we suppose that $u_\epsilon \in H^{s+1}(\Omega_\epsilon)$ is supported away from the boundary and $\theta_X u \in H^{s}(\Omega_0)$, then $\delta u_\epsilon$ exists in $H^s(\Omega_0)$ and
$$
\theta_X u = \delta u_\epsilon + X u_0.
$$
\end{lemm}
\begin{proof}
The lemma follows from writing
$$
\phi_\epsilon^* u_\epsilon - u_0 = \phi_\epsilon^* (u_\epsilon - u_0) + (\phi_\epsilon^* - 1)u_0,
$$
dividing both sides by $\epsilon$ and sending $\epsilon \to 0$.
\end{proof}
\begin{rema}
The formula in Lemma \ref{Lie derivative and first variation} is perfectly valid pointwise whenever $x \in \text{int}{\Omega_0}$. Both $\theta_X$ and $X$ are well defined operators on distributions supported near the boundary, so by setting $\delta u = \theta_X u - X u$, we obtain an extension of Definition \ref{first variation} for $s > 1$. From now on, we use $X$ to denote the differential operator acting on distributions on the fixed domain $\Omega_0$ and $\theta_X$ to denote the Lie derivative acting on distributions or differential forms on $\Omega_\epsilon$ which may also depend on $\epsilon$.
\end{rema}

\subsection{A general variational formula}\label{A general variational formula}
We now derive a variational formula using the weak formulation \eqref{weak formulation}. To obtain an integral equation on the fixed domain $\Omega_0$, we pull back the energy density and apply the change of variables formula to \eqref{weak formulation}:
\begin{align*}
\int_{\Omega_0} \phi_{\epsilon}^* e(u_\epsilon, {\phi_{\epsilon}^{-1}}^*v) \phi_{\epsilon}^* dV = \int_{\Omega_0} \phi_\epsilon^* f_\epsilon v \phi_{\epsilon}^* dV + \int_{\d \Omega_\epsilon} K_\epsilon u_\epsilon {\phi_{\epsilon}^{-1}}^*v dq_{\epsilon}.
\end{align*}
The pullback of the surface measure $dq_{\epsilon}$ in the last term on the right is more complicated, so for the time being, we leave it as is. We can rewrite this equation as
\begin{align}\label{pulled back weak formulation}
\int_{\Omega_0} e_\epsilon (\phi_{\epsilon}^*u_\epsilon, v) \phi_\epsilon^* dV = \int_{\Omega_0} \phi_\epsilon^* f_\epsilon v dV + \int_{\d \Omega_\epsilon} K_\epsilon u_\epsilon {\phi_{\epsilon}^{-1}}^*v dq_{\epsilon},
\end{align}
where
$$
e_\epsilon(u,v) = \phi_{\epsilon}^* (e({\phi_{\epsilon}^{-1}}^* u, {\phi_{\epsilon}^{-1}}^*v))
$$
is the conjugated energy density. While $e_\epsilon$ is a composition of operators, it is still of the form
$$
e_\epsilon = \sum_{|\alpha|, |\beta| \leq 2} c_{\alpha, \beta}(\epsilon,x,y) D_x^\alpha D_y^\beta,
$$
for some coefficients $c_{\alpha, \beta}$ depending smoothly on $x,y$ and in a $C^1$
manner on $\epsilon$. This will justify use of the product rule when computing $\epsilon$ derivatives in Lemma \ref{conjugated energy density} below. Differentiating \eqref{pulled back weak formulation} in the parameter $\epsilon$ and setting $\epsilon = 0$ yields
\begin{align}\label{divergence terms}
\begin{split}
\int_{\Omega_0} &( e(\theta_X u, v) + (\theta_X e)(u_0,v) ) dV + e(u_0,v)\theta_X dV\\
&= \int_{\Omega_0} v \theta_X f dV + v f \theta_X dV + \delta \int_{\partial \Omega_\epsilon}K_\epsilon u_\epsilon \phi_\epsilon^{-1 *}v \,dq_\epsilon,
\end{split}
\end{align}
where the quantity
$$
\theta_X e = \frac{d}{d\epsilon} \Big|_{\epsilon = 0}  e_\epsilon
$$
is defined analagously to the formula in Definition \ref{Lie derivative}.
\begin{lemm}\label{conjugated energy density}
For $u, v \in H^2(\Omega_0)$, differentiating the conjugated energy density $e_\epsilon$ yields
$$
(\theta_X e) (u,v) = X e(u,v) - e(Xu, v) - e(u,Xv).
$$
\end{lemm}
\begin{proof}
Consider the family of distributions $w_\epsilon = e({\phi_{\epsilon}^{-1}}^* u, {\phi_{\epsilon}^{-1}}^*v)$. Recalling that
$$
e_\epsilon(u,v) = \phi_{\epsilon}^* w_\epsilon,
$$
we see by Lemma \ref{Lie derivative and first variation} that
$$
\frac{d}{d\epsilon} \Big|_{\epsilon = 0}  \phi_{\epsilon}^* w_\epsilon = X w_0 + \delta w_\epsilon.
$$
The first term is precisely $X e(u,v)$ and the second term $\delta w_\epsilon$ is easily calculated by commuting the $\epsilon$ and $x$ derivatives:
\begin{align*}
\delta w_\epsilon &= \frac{d}{d\epsilon} \Big|_{\epsilon = 0} \langle \nabla {\phi_{\epsilon}^{-1}}^* u, \nabla {\phi_{\epsilon}^{-1}}^* v \rangle - \lambda^2 {\phi_{\epsilon}^{-1}}^* u {\phi_{\epsilon}^{-1}}^* v\\
&= \langle \nabla (-X u), \nabla v \rangle  + \langle \nabla u, \nabla (-X v) \rangle - \lambda^2 (-Xu) v - \lambda^2 u (-Xv)\\
&= -e(Xu,v) - e(u,Xv).
\end{align*}
Commuting derivatives is always valid in the sense of distributions. In order to also restrict to $\epsilon = 0$, note that $[\frac{d}{d \epsilon}, \nabla_x] {\phi_{\epsilon}^{-1}}^* u = 0$ in $L^2$ and this quantity only involves second derivatives of $u$, first order derivates of $\phi_{\epsilon}^{-1}$ in $\epsilon$, and $x$ derivatives of $\phi_{\epsilon}^{-1}$. Here, $[\cdot,\cdot]$ denotes the commutator of two operators. As $u \in H^2$ and $\phi_{\epsilon}$ is $C^1$ with respect to $\epsilon$, $[\frac{d}{d \epsilon}, \nabla_x] {\phi_{\epsilon}^{-1}}^* u$ is a continuous family (in $\epsilon$) of $L_x^2$ functions. Letting $\epsilon \to 0$ then implies $[\delta, \nabla] {\phi_{\epsilon}^{-1}}^* u = 0$.
\end{proof}
Returning to the variation of \eqref{pulled back weak formulation}, the Lie derivative of the volume form in equation \eqref{divergence terms} gives the divergence of $X$, which we would like to convert to a boundary integral, since $X$ is only defined in a tubular neighborhood of $\d \Omega_0$. Using again the formula $\theta_X = \delta + X$ and applying the divergence theorem to equation \eqref{divergence terms} above gives
\begin{align}\label{Variation without last term}
\begin{split}
\delta \int_{\Omega_0} e_\epsilon (\phi_{\epsilon}^*u_\epsilon, v) \phi_\epsilon^* dV &= \int_{\Omega_0} e(\dot{u},v) - e(u_0,Xv) dV + \int_{\d \Omega_0} e(u_0,v) X_\nu dq_0\\
&= \int_{\Omega_0} v \dot{f} - (Xv)f dV + \int_{\d \Omega_0} vf X_\nu dq_0 + \delta \int_{\d \Omega_\epsilon} K_\epsilon u_\epsilon {\phi_{\epsilon}^{-1}}^* v dq_\epsilon.
\end{split}
\end{align}
Here, $X_\nu = \langle X , \nu \rangle$ is the normal component of $X$. For the last term in \eqref{Variation without last term}, we can assume the perturbation is in the normal direction and parametrize the boundary by $\partial {\Omega_0} \ni x \mapsto x + \rho_\epsilon(x) \nu_x \in \partial \Omega_\epsilon$, so that
\begin{align}\label{boundary integral with measure}
\int_{\partial \Omega_\epsilon}K_\epsilon u_\epsilon \phi_\epsilon^{-1 *}v \,dq_\epsilon = \int_{\partial \Omega_0} K_\epsilon(x + \rho_\epsilon(x)\nu_x) u_\epsilon(x + \rho_\epsilon(x)\nu_x) v(x) \,dq_\epsilon.
\end{align}
We now recall a basic result from differential geometry:
%In this section, we will find a variational formula for the surface measure on a general $k$ dimensional immersed submanifold in $\R^n$. In the case of the ellipse, we will set $k = 1$ and $n = 2$. Consider $\phi = \phi_0: \partial \Omega^k \to \R^n$ an isometric immersion, equipped with the pullback metric $g = \phi^*h$, where $h$ is the flat metric on $\R^n$. Consider a $C^1$ variation $\phi_\epsilon : \d \Omega \times (-\epsilon_0, \epsilon_0) \to \R^n$ with variational vector field $\d_\epsilon \phi \in \Gamma (\phi_0^* T\R^n)$, a section of the pullback bundle.
\begin{lemm} \label{variation of surface measure}
The variation of surface measure is given by
$$
\delta dq_{\epsilon} = \kappa \dot{ \rho} dq_0,
$$
where $\kappa$ is the curvature of $\d \Omega_0$.
\end{lemm}
\begin{proof}
A proof using normal coordinates can be found on page 6 of \cite{CoMi11}.
%or \cite{BaWa15}.
\end{proof}
Hence, differentiating the entire boundary integral \eqref{boundary integral with measure}, we obtain
\begin{align}\label{boundary integral}
\frac{d}{d\epsilon} &\Big |_{\epsilon = 0} \int_{\partial \Omega_0} K_\epsilon(x + \rho_\epsilon(x)\nu_x) u_\epsilon(x + \rho_\epsilon(x) \nu_x) v(x) \,dq_\epsilon = \int_{\partial \Omega_0} \dot{K} u_0 v + K_0 \theta_X u v + K_0 u_0 v \dot{\rho} \kappa dq_0.
\end{align}
Actually, there are two ways to define $\dot{K}$ since it is currently ambiguous as to how to differentiate in $x$ on the hypersurface $\d \Omega_\epsilon$. The first way involves extending $K_\epsilon$ radially to a function defined on a tubular neighborhood of $\d \Omega_\epsilon$, so that we may differentiate in $x$ on an open subset of $\R^2$. The second way is to define $\dot{K}_\epsilon(x) = d/d\epsilon|_{\epsilon = 0} K_\epsilon(x + \rho_\epsilon \nu_x)$. While these two definitions differ pointwise, the integral formula remains the same and we adopt the second definition as it appears more naturally in the proof.

\subsection{Variation of Green's Kernel}
Combining equations \eqref{Variation without last term} and \eqref{boundary integral}, we obtain the variational formula
\begin{align*}
\int_{\Omega_0} & \langle\nabla \dot{u} , \nabla v \rangle -  \lambda^2 \dot{u} v dV + \int_{\partial \Omega_0} (\langle \nabla u_0 , \nabla v \rangle - \lambda^2 u_0 v) X_\nu dq_0\\
&= \int_{\Omega_0}  \langle\nabla u_0 , \nabla(Xv) \rangle - \lambda^2 u_0 Xv - fXv + \dot{f} v dV\\
&+ \int_{\partial \Omega_0} f v X_\nu + \dot{K} u_0 v + K_0 \dot{u} v + K_0 (X u_0)v  + K_0 u_0 v \dot{\rho} \kappa \,dq_0.
\end{align*}
Recall that $\phi_\epsilon(x) = x + \rho_\epsilon(x) \nu_x$ for $x \in \partial \Omega_0$, so that $X_\nu = \dot{\rho}$. Integrating by parts and collecting boundary terms, we see that
\begin{align*}
\int_{\Omega_0} \dot{u}  (- \Delta - \lambda^2)  v &
- (Xv) (-\Delta - \lambda^2)u_0 + fXv - \dot{f} v dV\\
= \int_{\partial \Omega_0}& - \dot{u}\nabla^\perp v - (\langle \nabla u_0, \nabla v \rangle - \lambda^2 u_0 v)\dot{\rho} + (Xv) \nabla^\perp u_0 \\
&+ fv \dot{\rho} + \dot{K} u_0v + K_0 \dot{u} v + K_0(Xu_0)v + K_0 u_0 v \dot{\rho} \kappa dq_0.
\end{align*}
If $u$ and $v$ satisfy the PDE \eqref{PDE} with Robin boundary conditions, we have
\begin{align}\label{satisfy robin bdy conditions}
\begin{split}
\int_{\Omega_0} \dot{u}  f_0 - \dot{f}& v dV\\
= \int_{\partial \Omega_0}& - \dot{u}K_0 v - (\langle \nabla u_0, \nabla v \rangle - \lambda^2 u_0 v)\dot{\rho} + (Xv) K_0 u_0 \\
&+ fv \dot{\rho} + \dot{K} u_0 v + K_0 \dot{u} v + K_0(X u_0)v + K_0u_0v \dot{\rho} \kappa dq_0\\
&= \int_{\partial \Omega_0} - (\langle \nabla u_0, \nabla v \rangle - \lambda^2 u_0 v)\dot{\rho} + fv \dot{\rho} + 2 K_0^2 u_0 v \dot{\rho} + \dot{K}u_0 v + K_0u_0v \dot{\rho} \kappa dq_0.
\end{split}
\end{align}
Noting that
\begin{align*}
\langle \nabla u_0, \nabla v \rangle = \langle \nabla^T u_0, \nabla^T v \rangle + \langle \nabla^\perp u_0, \nabla^\perp v \rangle = \langle \nabla^T u_0, \nabla^T v \rangle + K_0^2 u_0 v,
\end{align*}
we can get rid of one of the $K_0^2u_0 v\dot{\rho}$ terms in \eqref{satisfy robin bdy conditions} in exchange for only using tangential derivatives.\\
\\
To prove Theorem \ref{Hadamard}, we now fix $x_0 \in \text{int} (\Omega_0)$ and denote by $G_\epsilon$ be the Greens function on $\Omega_\epsilon$. Formally, the Green's function is the Schwartz kernel of the resolvent $(-\Delta_\epsilon - \lambda^2)^{-1}$. Setting $u_\epsilon(x) = G_\epsilon(\lambda, x, y)$, $v(x) = G_0(\lambda, x_0, x)$ and $f_\epsilon(x) = \delta_{x_0}(x)$, we obtain
\begin{align*}
\delta G_\epsilon(\lambda,x_0, y) = & \\
\int_{\partial \Omega_0} - \langle \nabla_2^T G_0(\lambda, x_0, q), \nabla_1^T G_0 (\lambda, q, y)\rangle  \dot{\rho} + (\lambda^2 \dot{\rho} & + K_0^2 \dot{\rho} + K_0 \kappa \dot{\rho} + \dot{K}) G_0(\lambda, x_0, q) G_0(\lambda, q, y) dq,
\end{align*}
which is precisely Theorem \ref{Hadamard} with $x$ replaced by $x_0$.  As $(\Delta_\epsilon - \lambda^2) \in \Psi^2(\Omega_\epsilon)$ is elliptic for $\Im \lambda$ positive, $G_\epsilon$ is a Lagrangian distribution with principal symbol in $S_{1,0}^{-2}$. Hence, $G_\epsilon$ is a family of distributions and $\delta G_\epsilon$ in particular, is a distribution of order $-1$ with wavefront set conormal to the diagonal $\{x = y\}$. As $x,y \in \text{int}(\Omega_0)$ and $q \in \d \Omega_0$ in Theorem \ref{Hadamard}, the points $(x,q)$ and $(q,y)$ are away from the diagonal, where the distribution is smooth. Hence, the tangential derivatives do not affect the smoothness or integrability. The distribution $\delta G_\epsilon$ can actually be extended up to the boundary using the method of layer potentials, although this is not needed in the remainder of the paper.

\subsection{The wave trace and eigenvalues}\label{the wave trace and eigenvalues}
We now want to find a formula for the variation of the distributional trace of the even wave propagator, $\cos(t \sqrt{- \Delta_\epsilon})$, in terms of that of the Green's kernel. Recall that the wave propagator $e^{i t \sqrt{-\Delta_\epsilon}}$ has a distributional trace in the sense that
\begin{align*}
\int_{\R} e^{it\sqrt{-\Delta_\epsilon}} \phi(t)\,dt
\end{align*}
is trace class for any Schwartz function $\phi$. Its trace is
\begin{align}\label{Weyl law wave trace}
\sum_j \int_{\R} e^{it \lambda_j(\epsilon)} \phi(t)\,dt,
\end{align}
where $(\lambda_j^2(\epsilon))_{j = 1}^\infty$ are the eigenvalues of $\Delta_{\epsilon}$. The sum in \eqref{Weyl law wave trace} can be seen to be convergent via integration by parts combined with Weyl's law on the asymptotic distribution of eigenvalues. While Weyl's law is usually stated for Dirichlet or Neumann eigenvalues, the Robin and Neumman asymptotics actually agree up to leading order due to the fact that the boundary operators $\d_\nu - K$ and $\d_\nu$ of the Robin and Neumann Laplacians respectively have the same principal symbol. For example, see \cite{Zayed2004} or \cite{Irvii16}. Taking real and imaginary parts, the analagous trace formulas hold for the even and odd wave kernels, which we denote by
\begin{align*}
E_R(t,x,y) = \cos t \sqrt{-\Delta_\epsilon} \quad \text{and} \quad S_R(t,x,y) = \frac{\sin t \sqrt{-\Delta_\epsilon}}{\sqrt{- \Delta_\epsilon}},
\end{align*}
respectively. The subscript $R$ here is to denote the Robin boundary conditions. In this section, we shall prove:
\begin{theo}\label{wavetrace}
The variation of the even wave trace is
$$
\delta \text{Tr} \cos t \sqrt{-\Delta_\epsilon} = \int_{\d \Omega_0} L_R^b(t,q,q)\dot{\rho} + \frac{t}{2}S_R(t,q,q)\dot{K}\,dq,
$$
where we have defined
$$
L_R^b(t,q,q') = \frac{t}{2} (-\nabla_1^T \nabla_2^T  - \Delta_2  + K_0^2 + K_0 \kappa) S_R.
$$
\end{theo}
Here, $\nabla_i^T$ is again the tangential derivative in the $i$th spacial variable and $\Delta_2$ is the Euclidean Laplacian in the second spacial variable. The kernels are first differentiated in the interior using an extension of the tangential vector field and then restricted to the diagonal of the boundary. We will also prove:
\begin{theo}\label{simple eigenvalue}
If $\lambda_j^2(0)$ is a simple eigenvalue associated to the $L^2$ normalized eigenfunction $\Psi_j$, then
\begin{align*}
\delta \lambda^2_j(\epsilon) = \int_{\d \Omega_0} |\nabla^T \Psi_j|^2 \dot{\rho}\,dq - \int_{\d \Omega_0} |\Psi_j|^2(\lambda_j^2 \dot{\rho} + K_0^2 \dot{\rho} + \dot{K} + K_0 \kappa \dot{\rho})\,dq.
\end{align*}
\end{theo}
Both theorems are proved together by the same method:
\begin{proof}
Our derivation of the wave trace variation is based on Kato's variational formulas for sums of eigenvalues in \cite{Kato}. One has to be careful, as an eigenvalue of higher multiplicity may not be $C^1$ in $\epsilon$. Such eigenvalues can break off to become many different eigenvalues under deformation. However, if we denote by $m(\lambda_j^2)$ the multiplicity of $\lambda_j^2 = \lambda_j^2(0)$, we will see that the sum $\sum_1^{m(\lambda_j^2)} \lambda_{j,k}^2(\epsilon)$ is in fact $C^1$ in $\epsilon$. We actually prove a more general theorem: let $g$ be holomorphic in a neighborhood of the eigenvalue $\lambda^2_j$ and denote the resolvent operator by $\widetilde{R}_\epsilon(z) = (-\Delta_\epsilon - z)^{-1}$ for $z \notin \text{Spec}(-\Delta_\epsilon)$, with Schwartz kernel $\widetilde{G}_\epsilon(z,x,y)$. We write $\widetilde{R}_\epsilon$ and $\widetilde{G}_\epsilon$ since the spectral parameter is $z$ instead of $z^2$. Then, by the Cauchy integral formula, we have
\begin{align}\label{Kato}
\delta \sum_{k = 1}^{m_j(\lambda_j^2)} g(\lambda_{j,k}^2(\epsilon)) = \delta \text{Tr} T_{g,\epsilon},
\end{align}
where
\begin{align*}
T_{g, \epsilon} = \frac{-1}{2\pi i} \int_{\gamma} g(z) \widetilde{R}_\epsilon(z)dz,
\end{align*}
for $\gamma$ a small, positively oriented circle enclosing only the $\lambda_j^2$ eigenvalue. When $\epsilon = 0$, $T_{g,0}$ is $g$ of the orthogonal projection onto the eigenspace of $\lambda_j^2$. As the eigenvalues do vary continuously in $\epsilon$, for $\epsilon \ll 1$, $T_{g,\epsilon}$ is the total projector, i.e. $g$ composed with the projection onto the direct sum of the eigenspaces of $\lambda_{j,k}^2 (\epsilon)$ for $1 \leq k \leq m(\lambda_j^2)$. $T_{g,\epsilon}$ is in fact a $C^1$ family of operators in $\epsilon$ since the resolvent is. The trace can also be obtained by integrating $\widetilde{G}_\epsilon$ over the diagonal, which combined with equation \eqref{Kato}, gives
$$
\delta \sum_{k = 1}^{m_j(\lambda_j^2(0))} g(\lambda_{j,k}^2(\epsilon)) = \frac{-1}{2\pi i} \int_{\gamma} g(z) \left( \delta \int_{\Omega_\epsilon} \widetilde{G}_\epsilon (z,x,x) dx \right) dz.
$$
As in Section \ref{A general variational formula}, we can pull back $\widetilde{G}_\epsilon (z,x,x) dx$ via $\phi_\epsilon^*$ to obtain an integral over a fixed domain, which we then differentiate. We have
\begin{align*}
\delta \sum_{k = 1}^{m_j(\lambda_j^2(0))} g(\lambda_{j,k}^2(\epsilon)) &= \frac{-1}{2\pi i} \int_{\gamma} g(z) \left\{\int_{\Omega_0}  \delta \phi_{\epsilon}^* \widetilde{G}_\epsilon (z,x,x) dx + \int_{\Omega_0} \widetilde{G}_0 (z,x,x) \delta \phi_{\epsilon}^* dx\right\} dz\\
&= \frac{-1}{2\pi i} \int_{\gamma} g(z) \left\{\int_{\Omega_0}  \theta_X \widetilde{G}_\epsilon (z,x,x) dx + \int_{\Omega_0} \widetilde{G}_0 (z,x,x) \text{div} X dx\right\} dz,
\end{align*}
where the last line follows from the definition of $\theta_X$ and the standard fact from Riemannian geometry that the Lie derivative with respect to $X$ of the volume form gives the divergence of $X$ times the volume form, i.e. $\theta_X(dx) = \text{div} X dx$. Using the formula for $\theta_X$ in Lemma \ref{Lie derivative and first variation} and the divergence theorem, we obtain
\begin{align*}
\frac{-1}{2\pi i} \int_{\gamma} g(z) \left\{\int_{\Omega_0}  (\delta + X) \widetilde{G}_\epsilon (z,x,x) dx - \int_{\Omega_0} X \widetilde{G}_0 (z,x,x) dx + \int_{\d \Omega_0} X_\nu \widetilde{G}_0(z,q,q)dq \right\} dz\\
= \frac{-1}{2\pi i} \int_{\gamma} g(z) \left\{\int_{\Omega_0}  \delta \widetilde{G}_\epsilon (z,x,x) dx + \int_{\d \Omega_0} \dot{ \rho}(q) \widetilde{ G}_0(z,q,q)dq \right\} dz.
\end{align*}
We now plug in our variational formula for the Green's kernel from Theorem \ref{Hadamard} to see that
\begin{align*}
\delta \sum_{k = 1}^{m_j(\lambda_j^2(0))} g(\lambda_{j,k}^2(\epsilon)) &= \frac{-1}{2\pi i} \int_{\gamma} g(z) \left\{\int_{\Omega_0}  \delta \widetilde{G}_\epsilon (z,x,x) dx + \int_{\d \Omega_0} \dot{ \rho}(q) \widetilde{ G}_0(z,q,q)dq \right\} dz\\
&= \frac{- 1}{2\pi i} \int_{\Omega_0} \int_{\partial \Omega_0} \int_\gamma g(z) (-\nabla_2^T \widetilde{G}_0(z , x, q)\cdot \nabla_1^T\widetilde{ G}_0 (z, q, x) \dot{\rho})\,dzdq
dx\\
& + \frac{- 1}{2\pi i} \int_{\Omega_0} \int_{\partial \Omega_0} \int_\gamma g(z)(z \dot{\rho} + K_0^2 \dot{\rho} + K_0 \kappa \dot{\rho} + \dot{K}) \widetilde{G}_0(z, x, q) \widetilde{G}_0 (z, q, x)\,dzdqdx\\
& + \frac{- 1}{2\pi i} \int_{\gamma} \int_{\d \Omega_0} g(z) \dot{ \rho}(q) \widetilde{ G}_0(z,q,q)dq dz.
\end{align*}
%\blue{It seems that we will get an extra term when we differentiate the trace on the first line. See the last equation on page 1113 of my paper with Steve. I think the correct formula on the first line is:
%	$$\delta \sum_{k = 1}^{m_j(\lambda_j(0))} g(\lambda_{j,k}^2(\epsilon)) = \delta Tr T_{g,\epsilon} = - Tr \frac{1}{2\pi i} \int_{\gamma} g(z) \delta \tilde{R}(z)dz + \frac{-1}{2\pi i} \int_\gamma \int_{\partial \Omega} \tilde R(q, q, z) \dot{ \rho}(q) dq dz .$$ This will change the subsequent terms for the better! And explains why we should not have a factor 2.}
Denote these three integrals by $I_1, I_2$ and $I_3$ and let $(\Psi_{j,k})_{k = 1}^{m(\lambda_j^2)}$ be an orthonormal basis for the eigenspace corresponding to the eigenvalue $\lambda_j^2$. Then, via the Cauchy integral formula, we have
\begin{align*}
I_1 &= \sum_1^{m(\lambda_j^2)} \int_{\gamma} \int_{\partial \Omega_0} \frac{g(z)}{(2\pi i)(z - \lambda_j^2)^2} \nabla^T \Psi_{j,k}(q) \cdot \nabla^T \Psi_{j,k}(q) \dot{\rho} (q) dq dz\\
& = g'(\lambda_j^2) \sum_{k = 1}^{m(\lambda_j^2)} \int_{\partial \Omega_0} |\nabla^T \Psi_{j,k} (q)|^2 \dot{\rho} dq.
\end{align*}
Similarly, we have
\begin{align*}
I_2 = & -g(\lambda_j^2) \sum_{k=1}^{m(\lambda_j^2)} \int_{\d \Omega_0} |\Psi_{j,k} (q)|^2 \dot{\rho}(q) dq\\
  & - g'(\lambda_j^2) \int_{\d \Omega_0} \left( \sum_{k = 1}^{m(\lambda_j^2)} | \Psi_{j,k}(q)|^2\right) \left(\lambda_j^2 \dot{\rho} + K_0^2 \dot{\rho} + K_0 \kappa \dot{\rho} + \dot{K} \right) dq\\
& = \int_{\d \Omega_0} \left( \sum_{k = 1}^{m(\lambda_j^2)} | \Psi_{j,k}(q)|^2\right) \left(- g(\lambda_j^2) \dot{\rho} - g'(\lambda_j^2)(\lambda_j^2 \dot{\rho} + K_0^2 \dot{\rho} + K_0 \kappa \dot{\rho} + \dot{K}) \right) dq
\end{align*}
and
\begin{align*}
I_3 = \sum_{k = 1}^{m(\lambda_j^2)}  \int_{\d \Omega_0} g(\lambda_j^2) \dot{ \rho}(q) |\Psi_{j,k}(q)|^2 dq.
\end{align*}
Combining these terms and noticing that $I_3$ cancels with one of the terms in $I_2$, we obtain
\begin{align}\label{Variational formula for g}
\delta \sum_{k = 1}^{m_j(\lambda_j^2)} g(\lambda_{j,k}^2(\epsilon)) = g'(\lambda_j^2) \sum_{k =  1}^{m_j(\lambda_j^2)} \int_{ \d \Omega_0 }  |\nabla^T \Psi_{j,k}|^2 \dot{ \rho} - |\Psi_{j,k}|^2 (\lambda_j^2 \dot{ \rho} + K_0^2 \dot{ \rho} + K_0 \kappa \dot{ \rho} + \dot{K}) dq.
\end{align}
To compute the variation of the even wave trace in particular, set $g(z) = \cos(t \sqrt{z})$ in equation \eqref{Variational formula for g}. Despite the square root, this is in fact an entire function since cosine is even. We have
\begin{align*}
& \delta \sum_{j = 1}^{\infty} \sum_{k = 1}^{m(\lambda_j^2)} \cos(t \lambda_{j,k}(\epsilon)) = 
\sum_{j = 1}^{\infty} \delta \sum_{k = 1}^{m(\lambda_j^2)} \cos(t \lambda_{j,k}(\epsilon))\\
& = \sum_{j = 1}^\infty \sum_{k =  1}^{m(\lambda_j^2)} \frac{- t \sin(t \lambda_j)}{2 \lambda_j} \int_{ \d \Omega_0 }   |\nabla^T \Psi_{j,k}|^2 \dot{ \rho} - |\Psi_{j,k}|^2 (\lambda_j^2 \dot{ \rho} + K_0^2 \dot{ \rho} + K_0 \kappa \dot{ \rho} + \dot{K}) dq.
\end{align*}
Writing this in terms of the wave kernels, we obtain
\begin{align}\label{wave kernel variation}
\begin{split}
\delta \text{Tr} \cos t & \sqrt{-\Delta_\epsilon} =\\
\int_{\d \Omega_0} \frac{-t}{2} \nabla_1^T \nabla_2^T S_R(t,q,q) \dot{\rho} &  + \frac{t}{2} \dot {\rho} (-\Delta_{2} + K_0^2 + K_0\kappa) S_R(t,q,q)  + \frac{t}{2} S_R(t,q,q) \dot{K} \,dq.
\end{split}
\end{align}
Note that all but one of the terms above contain $\dot{\rho}$. Recalling that in the begining of the section, we defined
$$
L_R^b(t,q,q') = \frac{t}{2} (-\nabla_1^T \nabla_2^T  -\Delta_2  + K_0^2 + K_0 \kappa) S_R
$$
to be the coefficient of $\dot{\rho}$ in the expression \eqref{wave kernel variation}, we obtain Theorem \ref{wavetrace}. Similarly, setting $g(z) = z$ in equation \eqref{Kato} easily yields Theorem \ref{simple eigenvalue} on the variation of simple eigenvalues.
\end{proof}

\section{Billiards}\label{Billiards}
Before obtaining a singularity expansion for the wave trace, we first review the relevant background needed on billiards. This will also be useful in our discussion of Chazarain's parametrix in Section \ref{Chazarain's Parametrix}. In this section, we drop the subscript $0$ from our domain in Section \ref{Variation of wave trace} and let $\Omega$ denote any bounded strictly convex region in $\R^2$ with smooth boundary. This means that the curvature of $\d \Omega$ is a strictly positive function. The billiard map is defined on the coball bundle of the boundary $B^* \d \Omega = \{(q, \zeta) \in T^*\d \Omega : |\zeta| < 1 \}$, which can be identified with the inward part of the circle bundle $S_{\d \Omega}^* \R{^2}$ via the natural orthogonal projection map. We can also identify $B^* \partial \Omega$ with $\R / \ell \Z \times [0,\pi]$, where $\ell=|\partial \Omega|$ is the length of the boundary. Define
\begin{align*}
t_{\pm}^1(y,\eta)& = \inf\{t > 0 : \pi_1( g^{\pm t}(y,\eta)) \in  \d \Omega \},\\
t_{\pm}^{-1}(y,\eta) &= \sup \{t < 0 : \pi_1(g^{\pm t}(y,\eta)) \in  \d \Omega \},
\end{align*}
where $\pi_1$ is projection onto the first factor and $g^{\pm t}$ is the forwards $(+)$ or backwards $(-)$ geodesic flow on $\R^2$, corresponding to the Hamiltonian $H_\pm = \pm|\eta|$. We then define
$$
\beta^{\pm1}(y,\eta) = \widehat{g^{t_\pm^1}(y,\eta)},
$$
where a point $\widehat{(x,\xi)}$ is the reflection of $\xi$ through the cotangent line $T_x^*\d \Omega$. In otherwords, $\widehat{(x,\xi)}$ has the same footpoint and cotangential component as $(x,\xi)$, but reflected conormal component, so that it is again in the inward facing portion of the circle bundle. We call $\beta := \beta^{+1}$ the billiard map. It is well known that $\beta$ preserves the natural symplectic form induced on $B^*(\d \Omega)$. Associated to this map is the billiard flow, or broken bicharacteristic flow, which we denote by $\Phi^t$.\\
\\
The times $t_{\pm}^j$ are defined inductively by
\begin{align*}
t_{\pm}^j(y,\eta)& = \inf\{t > 0 :\pi_1( g^{\pm t_{j-1} + t}(y,\eta)) \in  \d \Omega \}, \quad j \in \Z,\\
t_{\pm}^{-j}(y,\eta) &= \sup \{t < 0 : \pi_1(g^{\pm t_{j-1} + t}(y,\eta)) \in  \d \Omega \}, \quad j \in \Z,
\end{align*}
and the maps $\beta^{\pm n}$ are defined via iteration. We also define $T_\pm^n = \sum_{j = 1}^n t_\pm^j$ to be the total time of the flow from $(y,\eta)$ to $\beta^n(y,\eta)$.\\
\\
Geometrically, a billiard orbit corresponds to a union of line segments which are called links. A smooth closed curve $\mathcal C$ lying in $\Omega$ is called a {caustic} if any link drawn tangent to $\mathcal C$ remains tangent to $\mathcal C$ after an elastic reflection at the boundary of $\Omega$. By elastic reflection, we mean that the angle of incidence equals the angle of reflection at an impact point on the boundary. We map $\mathcal C$ onto the total phase space $B^* \partial \Omega$ to obtain a smooth closed curve which is invariant under $\beta$. In the case that the dynamics are integrable, these invariant curves are precisely the Lagrangian tori which folliate the phase space. A point $P$ in $B^*\partial \Omega$ is $q$-periodic, $q \geq 2$, if $\beta^q(P)=P$. We define the rotation number of a $q$-periodic point $P$ by $\omega(P)= \frac{p}{q}$, where $p$ is the winding number of the orbit generated by $P$, which we now define. We may consider the modified billiard map $\widetilde{\beta} = \Pi^* \beta$, where $\Pi$ is the natural mapping from $\R / \ell \Z \times [0,\pi]$ to the closure of the coball bundle $\overline{B^*\d \Omega}$. Pulling back by $\Pi$ clearly preserves the notion of periodicity. There exists a unique lift $\widehat{\beta}$ of the map $\widetilde{\beta}$ to the closure of the universal cover $\R \times [0,\pi]$ which is continuous and satisfies $\widehat{\beta}(x,0) = (x,0)$. Given this normalization, for any point $(x,\theta) \in \R/\ell \Z \times [0,\pi]$ in a $q$ periodic orbit of $\widetilde{\beta}$, we see that $\widehat{\beta}(x,\theta) = (x + p \ell, \theta)$ for some $p \in \Z$. We define this $p$ to be the winding number of the orbit generated by $\Pi (x,\theta) \in \overline{ B^* \d \Omega}$. We see that even if a point $\Pi (x,\theta)$ generates an orbit which is not periodic in the full phase space but is such that $\pi_1 (\widetilde{\beta}^q(x,\theta)) = x$ for some $q \in \Z$, we can still define a winding number in this case. Such orbits are called loops or geodesic loops. For deeper results and a more thorough treatment of general dynamical billiards, we refer the reader to \cite{Tabachnikov}, \cite{Katok}, \cite{Popov1994} and \cite{PopovTopalov}.

\subsection{Elliptical billiards}\label{Elliptical billiards}
From here on, we let $\Omega$ be an ellipse with horizontal major axis, given by the equation
$$
\frac{x^2}{a^2} + \frac{y^2}{b^2}  \leq 1.
$$
The eccentricity of $\Omega$ is defined by $E = a^2 - b^2$.
\\
\\
Birkhoff conjectured that the only strictly convex {integrable} billiard tables are ellipses. Completely integrable means that there exists a folliation of the phase space by invariant submanifolds. In the context of Hamiltonian systems, this can be shown to be equivalent to the existence of a maximal number of Poisson commuting invariants, called first integrals. The (compact) energy level sets of regular values of these invariants can be shown to be diffeomorphic to tori and the leaves of a maximal such foliation are called Lagrangian tori. The Lagrangian tori are naturally parametrized by so called ``angle coordinates'' while the transversal directions in phase space are then parametrized by ``action coordinates''. It is well known that for each $Z \in (-E,0) \cup (0,b]$, the confocal ellipse ($Z>0$) or hyperbola ($Z<0$) given by 
$$
\frac{x^2}{E + Z} + \frac{y^2}{Z} = 1,
$$
is also a caustic (see Figure \ref{caustic pictures}). A short proof of this can be found using elementary planar geometry in the appendix of \cite{GuMe79a}. For elliptical caustics, we follow the notation in \cite{KaSo16} and \cite{DCR17} by setting $\lambda^2 = Z \geq 0$. In the context of \cite{KaSo16}, integrable is taken to mean that the union of all convex caustics has a non-empty interior in $\R^2$. Ellipses are both completely integrable and integrable in the sense of \cite{KaSo16}.\\
\\
\begin{figure}
\includegraphics[scale=0.8]{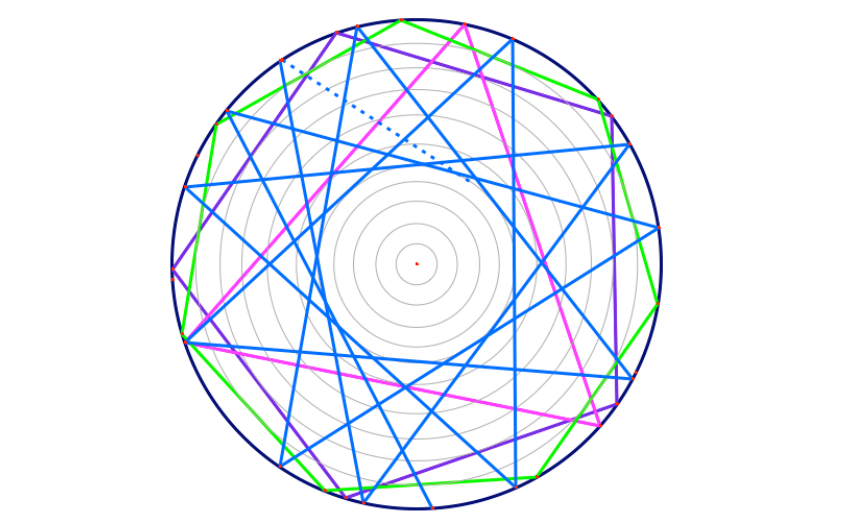}
\includegraphics[scale=0.7]{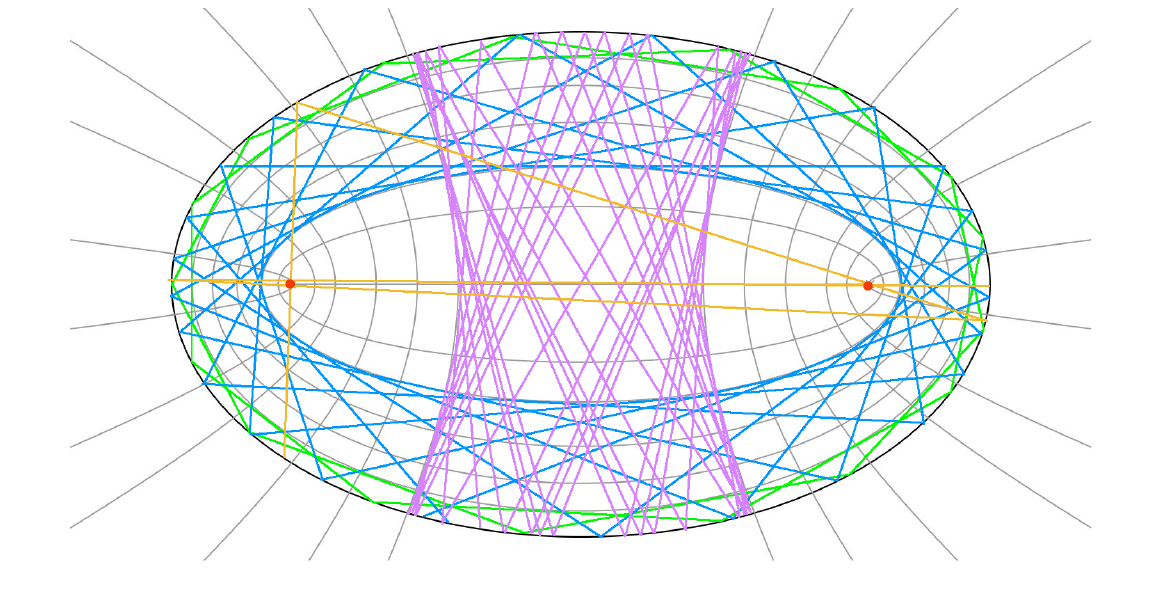}
\caption[Billiards and caustics on the disk and ellipse.]{Billiards and caustics on the disk and ellipse. \protect \footnotemark}
\label{caustic pictures}
\end{figure}
\footnotetext{Images courtesy of Vadim Kaloshin and Alfonso Sorrentino.}
In 1822, Poncelet proved the following remarkable theorem:
\begin{theo}  [\cite{Poncelet1}, \cite{Poncelet2}] \label{poncelet}
Given an ellipse, if a primitive periodic billiard trajectory is tangent to a confocal conic section, then all orbits tangent to that caustic are also periodic, have the same periods, and have the same lengths.
\end{theo}
Hence, periodic orbits of a given length come in 1-parameter families. Poncelet's original theorem, as stated in his 1865 treatises \cite{Poncelet1} and \cite{Poncelet2}, was much more general and concerned inscribing and circumscribing polygons about other confocal conic sections. There are several modern proofs of Poncelet's theorem. Because of this result, we can project confocal ellipses to invariant curves in the coball bundle $B^*\d\Omega$. We will use $\omega(\mathcal C)$ for the rotation number of such an invariant curve. For ellipses, the rotation number of a periodic point is always the same as the rotation number of its corresponding invariant curve.\\
\\
Birkhoff proved in \cite{Birkhoff} that for any smooth, strictly convex domain and any $\frac{p}{q} \in (0, \frac12]$ in lowest terms, there exist at least two geometrically distinct periodic orbits with rotation number $\frac{p}{q}$. A strictly convex billiard table is said to be {rationally integrable} if for each $q \geq 3$, there exists a caustic consisting of periodic points of rotation number $\frac{1}{q}$. In particular, the rotation number of the invariant curve corresponding to such a caustic must be $\frac{1}{q}$. A third version of Birkhoff's conjecture is that ellipses are the only rationally integrable strictly convex billiard tables. We have seen that there exist many notions of integrability, yet Birkhoff's conjecture remains open for all of them. However, as mentioned in Section \ref{Background}, a local version was proven in \cite{KaSo16}.
\\
\\
In \cite{GuMe79a}, it is shown that periodic points of the billiard map for an ellipse are dense in phase space. Given a length $T \in \text{Lsp}(\Omega)$, the associated fixed point set is denoted $F_T = \{(q, \zeta) \in B^* \d \Omega : \Phi^T(q, \eta) = (q,\eta)\}$. In \cite{GuMe79a}, the authors construct a special sequence of caustics converging to the boundary such that the associated fixed point submanifold in $B^*\d \Omega$ has exactly two connected components, corresponding to forwards and backwards flow:
\begin{prop} [\cite{GuMe79a}] \label{caustics} 
Let $T_0 = |\d \Omega|$ be the perimeter of the ellipse $\Omega$. Then in every interval $(T_0 - \epsilon, T_0)$, there exist infinitely many lengths $T \in \text{Lsp}(\Omega)$. For all but a finitely many such $T \in \text{Lsp}(\Omega)$, $F_T$ is the union of two invariant curves which are mapped to each other by $(q,\zeta) \to (q, - \zeta)$.
\end{prop}
The time reversal map $(q,\zeta) \to (q,-\zeta)$ reverses the direction of a closed geodesic, which of course preserves its length and geometry. We call the lengths $T$ in Proposition \ref{caustics} simple, as all periodic trajectories of length $T$ are tangent to a single caustic. Proposition \ref{caustics} will be crucial in evaluating our singularity expansion in Section \ref{Converting the singularity into an elliptic integral} and allowing us to differentiate the constants $c_j$ and $\widetilde{c_j}$ from Theorem \ref{wavetrace variation} near $\lambda = 0$. That there is only one connected component up to symmetry will rule out any cancellation between terms in the variation of the wave trace.
%\blue{Recent results on length spectral rigidity are contained in \cite{KaAvDS16} and \cite{KaDSWe17} and \cite{KaSo16}. In \cite{KaAvDS16}, it is shown that a rationally integrable billiard table which is sufficiently close to a circle must be an ellipse. It is then shown in \cite{KaSo16} that a rationally integrable billiard table which is sufficiently close to an ellipse must be an ellipse. Using the work of Siburg and Aubry-Mathur theory, it then follows that ellipses are {length spectrally rigid}. This means that if there is a $C^1$ family $\{\Omega_\epsilon\}_{\epsilon \in [0, 1]}$ of smooth domains with $\Omega_0=E$, and if $Lsp (\Omega_t)$ = $Lsp (\Omega_0)$, then after a rigid motion $\Omega_\epsilon= \Omega_0$. However, we can not use this in the proof of our result since the length spectrum may not be a complete $\Delta$ isospectral invariant.} \red{I think I should delete this last paragraph, as this literature was already mentioned in Section 2 (Background)}

\section{A parametrix for $S_R$ and Singularity Expansion}\label{A parametrix for SR and singularity expansion}
In this section, we use microlocal analysis to obtain a singularity expansion for the variation of the wave trace near the length spectrum. In particular, we microlocalize the wave kernels near periodic transversal reflecting rays in order to obtain a Fourier integral operator (FIO). Using Theorem \ref{wavetrace}, we can rewrite
\begin{align}\label{composition of FIOs}
\delta \text{Tr} \cos t \sqrt{-\Delta_\epsilon} = \pi_* \Delta^* r_1 r_2( L_R^b \dot{\rho} + \frac{t}{2}S_R \dot{K}),
\end{align}
where $r_1, r_2$ are the boundary restriction operators, $\Delta: \d \Omega \to \d \Omega \times \d \Omega$ is the diagonal embedding, $\pi_*$ is integration over the fibers, and the modified propagator $L_R^b$ is given by
\begin{align*}
L_R^b = \frac{t}{2} (-\nabla_1^T \nabla_2^T  -\Delta_{2}  + K_0^2 + K_0 \kappa) S_R.
\end{align*}
This manipulation of notation is just to illustrate how one can decompose the wave trace variation into the composition of simpler Fourier integral operators. However, the wave kernels $L_R^b$ and $S_R$ in our formula are \textit{not} FIOs near the glancing set $S^*\d \Omega$. In \cite{HeZe12}, the  authors microlocalize the wave kernels near periodic nonglancing orbits and calculate the principal symbol of the composition \eqref{composition of FIOs} for Dirichlet and Neumann boundary conditions using the symbol calculus in \cite{DuGu75}. In contrast to the methods employed in \cite{HeZe12}, we instead take a more direct approach which avoids an application of the trace formula in \cite{DuGu75}.
\\
\\
We begin by reviewing FIOs and Chazarain's parametrix for the wave propagator. In Section \ref{Computing the singularity in elliptical polar coordinates}, we then cook up an explicit oscillatory integral representation for each term in Chazarain's parametrix, which microlocally approximates the wave propagator in the interior by using action angle coordinates for the billiard map. This will rely heavily on the symbol calculus in Section \ref{Chazarain's Parametrix} and a new phase function for wave propagator. This technique can also be extended to deal with other convex billiard tables, although the dynamics are not as simple.
%Hence, we should microlocolize these operators by a pseudodifferential cutoff $\chi_T(t,D_t,q,D_q) \in \Psi^0(\R \times \d \Omega)$ supported near a periodic transversal reflecting ray. Let $T \in Lsp(\Omega) \backslash \Z |\d \Omega|$ and suppose the fixed point set $F_T$ is a smooth submanifold of the unit coball bundle $B^*\d \Omega$. We set $\chi_T$ to be a cutoff function supported near $F_T$ and away from $S^*\d \Omega$. Quantizing the symbol $\chi_T(q, \zeta/\tau)$, we obtain a pseudodifferential operator independent of time, which we also denote $\chi_T(D_t, q, D_q)$. By possibly shrinking the support, we may assume that the $\text{supp} \chi_T$ is invariant under the billiard map $\beta$. If we compose a $\Psi DO$ or special FIO with a given FIO, we obtain another FIO. We want to understand the symbol of the composition
%$$
%\delta Tr \cos t \sqrt{-\Delta_\epsilon} = \pi_* \Delta^* \chi_T(D_t, q, D_q) \chi_T(D_t, q', D_{q'}) (L_R^b(t,q,q') + \frac{t}{2} S_R(t,q,q')).
%$$
%for $t$ near $T$.

\subsection{Fourier Integral Operators}\label{Fourier Integral Operators}
Let $X$ and $Y$ be open sets in $\R^{n_X}$ and $\R^{n_Y}$ respectively. Recall that a continuous linear operator $A : C_0^\infty(Y) \to \mathcal{D}'(X)$ has an associated Schwartz kernel $K_A \in \mathcal{D}'(X \times Y)$. If $a \in S^\mu_{1, 0}(X \times \R^N)$ is a classical symbol of order $\mu$ and $\Theta \in C^\infty(X  \times \R^N)$ is a nondegenerate phase function, then the linear form
$$
A(u) = \int_X \int_{\R^N} e^{i \Theta(x,\theta)} a(x,\theta) u(x) \,d\theta dx
$$
is called a Lagrangian or Fourier integral distribution on $X$. If we assume that $K_A$ is given by a locally finite sum of Lagrangian distributions on $X \times Y$, then we say $A$ is a Fourier integral operator (FIO). One can then show that the wavefront set of the kernel is contained in the image of the map $\iota_{\Theta}: (x, y) \mapsto (x, y, d_x \Theta, d_y\Theta)$ when restricted to the critical set $C_\Theta = \{d_\theta \Theta = 0\}$. The image of $\iota_\Theta$ is in fact a conic Lagrangian submanifold $\Lambda_\Theta \subset T^* (X \times Y)$ and the map $\iota_{\Theta}$ is actually a local diffeomorphism from $C_\Theta$ onto $\Lambda_\Theta$. In this case, we say that ``$\Theta$ parametrizes $\Lambda_\Theta$.'' The canonical relation or wavefront relation of $A$ is defined by
$$
WF'(A) = \{(x,\xi), (y, \eta) : (x,y,\xi, -\eta) \in WF(K_A)\} \subset T^*X \times T^*Y,
$$
which describes how the operator propagates the singularities of distributions on which it acts. More invariantly, one can consider FIOs associated to general conic Lagranigan submanifolds $\Lambda \subset T^*X \times T^*Y$ (canonical relations), with respect to the symplectic form $\omega_X - \omega_Y$. The notion of a principal symbol for Fourier integral operators is more subtle than that for pseudodifferential operators: the principal symbol of $A$ is a half density on $\Lambda$ given in terms of the parametrization $\iota_\Theta$:
\begin{align}\label{principal symbol of an FIO}
e = {\iota_\Theta}_* (a_0 |dC_\Theta|^{1/2}),
\end{align}
where $a_0$ is the leading order term in the asymptotic expansion for $a$ and $|dC_\Theta|^{1/2}$ is the half density associated with the Leray measure on the level set $\{d_\theta \Theta = 0\}$. Here, we have ignored Maslov factors coming from the Keller-Maslov line bundle over $\Lambda_\Theta$. These are nonzero factors $e^{i\sigma \pi/4}$ ($\sigma$ is known as the Maslov index) which appear in front of the principal symbol as a result of the multiplicity of phase functions parametrizing the canonical relation $\Lambda_\Theta$, possibly in different coordinate systems. While these factors allow the principal symbol to be defined in a more geometrically invariant way, we defer computation of the Maslov indices until Section \ref{Maslov}. For a more thorough reference on the global theory of Lagrangian distributions, see \cite{Du96}. The order of a Fourier integral operator is defined in such a way that when two Fourier integral operators' canonical relations meet transversally, then the composition is again a Fourier integral operator and order of the composition is the sum of the orders:
\begin{align*}
\text{order}(A) = m = \mu + \frac{1}{2}N - \frac{1}{4}(n_X + n_Y).
\end{align*}
Recall that here, $n_X$ and $n_Y$ are the dimensions of $X$ and $Y$ respectively. In this case, we write $A \in I^m(X \times Y, \Lambda)$. This convention on orders also generalizes pseudodifferential operators, where $X = Y$ and $m = \mu$ coincides with the order of the corresponding symbol class. A sufficient condition which guarantees that the composition exists is clean or transversal intersections of the two operators' canonical relations. In general, composition of Fourier integral operators and the associated symbol calculus is somewhat complicated, but is discussed in \cite{Du96}, \cite{Ho71} and \cite{DuHo72}. We will not directly use the composition formula in what follows.
%The restriction operator $r \in I^{1/4}(\d \Omega \times \R^2, \Gamma_{\d \Omega})$ is a Fourier integral operator with canonical relation
%$$
%\Gamma_{\d \Omega} = \{(q,\zeta, q, \xi)\in T^*\d \Omega \times T_{\d \Omega}^* \R^2 : \xi |_{T_q \d \Omega} = \zeta \}.
%$$
%Hence, $r^* \in I^{1/4}(\R^2 \times \d \Omega, \Gamma_{\d \Omega}^*)$ where
%$$
%\Gamma_{\d \Omega}^* = \begin{pmatrix}
%0 & 1\\
%1 & 0
%\end{pmatrix} \Gamma_{\d \Omega}.
%$$
%This can be seen by localizing in $X, Y$ and applying the Fourier inversion theorem. More generally, for an immersion $\kappa: X \to Y$, the pullback $\kappa^*: C^\infty(Y) \to C^\infty(X)$ is an FIO.

\subsection{Chazarain's parametrix}\label{Chazarain's Parametrix}
Chazarain's parametrix provides a microlocal description of the wave kernels near periodic transversal reflecting rays. The parametrices for $E_R$ and $S_R$ are constructed in the ambient Euclidean space $\R \times \R^n \times \R^n$. We only consider $S_R$, as the formula for $E_R$ is easily obtained from that of $S_R$ by differentiating in $t$. Following the work in \cite{Ch76} and \cite{GuMe79b}, we can find a Lagrangian distribution
\begin{align}\label{Chazarain sum}
\widetilde{S_R}(t,x,y) = \sum_{j = - \infty}^{\infty} S_j(t,x,y), \qquad S_j \in I^{-5/4}(\R \times \R^n \times \R^n, \Gamma_\pm^j),
\end{align}
which approximates $S_R(t,x,y)$ microlocally away from the tangential rays modulo a smooth kernel. We will describe the canonical relations $\Gamma_\pm^j$ momentarily and in particular, show that the sum in \eqref{Chazarain sum} is locally finite. We first explain what is meant by approximating $S_R(t,x,y)$ ``microlocally away from the tangential rays.'' In general, two distributions $f,g \in \mathcal{D}'(\R^n)$ are said to agree microlocally near a closed cone $\Lambda \subset T^*\R^n$ if $WF(u - v) \cap \Lambda = \emptyset$. Similarly, using the language from Section \ref{Fourier Integral Operators}, two operators $A, B : C^\infty(Y) \to C^\infty(X)$ are said to agree microlocally near a given closed cone $\widetilde{\Lambda} \subset T^*X \times T^*Y$ if $WF'(A-B) \cap \widetilde{\Lambda} = \emptyset$. This second notion is what we will use to say that our parametrix approximates $S_R$ microlocally near the canonical relations $\Gamma_{\pm}^j$.
\\
\\
In this section, we study the problem
\begin{align*}
\begin{cases}
u_{tt} - \Delta u = 0, & x \in \Omega,\\
(\d_\nu - K)u = 0, & x \in \d \Omega,\\
u(0,x) = 0,\\
u_t(0,x) = \delta(x- x_0),
\end{cases}
\end{align*}
for $x_0 \in \text{int}\Omega$ and $K$ a Robin function as in equation \eqref{PDE}. Chazarain's construction begins with the solution of the homogeneous wave equation in the ambient Euclidean space:
\begin{align*}
\begin{cases}
u_{tt} - \Delta u = 0, & x \in \R^2,\\
u(0,x) = 0,\\
u_t(0,x) = \delta(x- x_0).
\end{cases}
\end{align*}
In this case, we have an explicit representation for the fundamental solution of $\Box = \d_{t}^2 - \Delta$, given by Kirchhoff's formula. If we restrict back to $\Omega$, it is clear that the fundamental solution on $\R^2$ will in general not satisfy the Robin boundary condition. However, finite speed of propogation implies that it \textit{does} satisfy the boundary condition for small time, since $u$ vanishes identically in a neighborhood of $\d \Omega$. If we let $d = d(x_0, \d \Omega)$, then we obtain a solution $u$ of the boundary value problem for $|t| < d$. The idea in \cite{Ch76} is to use the billiard flow and properties of FIOs to inductively extend the time interval on which the fundamental solution of the boundary value problem is defined.
\\
\\
With $t_\pm^j$ defined as in Section \ref{Billiards} on billiards and $(y,\eta) \in T^*\Omega$ or $T_{\d\Omega}^* \R^2$ inward pointing, we define
\begin{align*}
\lambda_\pm^1 &= g^{\pm t_\pm^1}(y,\eta),\\
\lambda_\pm^{-1} &= g^{\pm t_\pm^{-1}}(y,\eta).
\end{align*}
Notice that if $y \in \d \Omega$ and $|\eta| = 1$ is inward pointing, then $\beta(y, \iota^* \eta) = \widehat{\lambda_+^1(y, \eta)}$, where $\beta$ is the billiard map on $B^*\d \Omega$ and $\iota: B \d \Omega \to S_{\d \Omega} \R^2$ is the natural projection mapping the ball bundle of the boundary to the inward pointing portion of the circle bundle with footpoints on the boundary. Even though we can obtain one from the other, we define both $\lambda_\pm$ and $\lambda_\pm^{-1}$ in order to separate the forwards and backwards wave propagators corresponding to $t > 0$ and $t < 0$. After a reflection at the boundary, we can similarly define $\lambda_\pm^j(y,\eta)$ for any $j \in \Z$. Recall that in Section \ref{Billiards} we defined $T_\pm^j = \sum_{k = 1}^j t_\pm^k$ for $j > 0$ and $T_\pm^j = \sum_{k = j}^{-1} t_\pm^k$ for $j < 0$.\\
\\
To study how the fundamental solution behaves at the boundary, we propagate the intial data by the free wave propagator on $\R^2$, restrict it to the boundary, reflect, and then propagate again. If such a construction is continued for $j \in \Z$ reflections at the boundary, it is shown in \cite{Ch76} that the phase functions corresponding to the FIOs $S_j$ parametrize the canonical relations
$$
\Gamma_\pm^j = \begin{cases}
(t, \tau, g^{\pm t}(y,\eta), y, \eta): \tau = \pm |\eta| & j =0,\\
(t, \tau, g^{\pm (t - T_\pm^j(y,\eta))}\widehat{\lambda_\pm^j(y,\eta)}, y, \eta): \tau = \pm |\eta| & j \in \Z \backslash \{0\}.
\end{cases}
$$
Again, $j > 0$ and $j < 0$  correspond to reflections in forward and backward time. In fact, there are four modes of propagation, corresponding to $\pm \tau \geq 0$ and $\pm j \geq 0$ in the canonical relations $\Gamma_\pm^j$. Chazarain actually showed that there exists FIOs $S_j$ such that the sum in \eqref{Chazarain sum} is in fact a parametrix for the wave propagator $S_R$ with canonical relation
$$
\Gamma = \bigcup_{j \in \Z, \pm} \Gamma_{\pm}^j.
$$
However, the principal symbols of the operators $S_j$ are never computed in \cite{Ch76} and we concern ourselves with the task of explicitly computing them in the special case of an ellipse for the remainder of this section.
\\
\\
Recall that the Hadamard type variational formula for the wave trace in Theorem \ref{wavetrace} involved the integral of wave kernels over the diagonal of the boundary. To understand the principal symbol first in the interior, we study how the propagator reflects at the boundary. In particular, we want to study the canonical relations $\Gamma_\pm^j$ restricted to the fibers over the boundary, which we now describe. Denote by
$$
A_\pm^0 = \{(0,\tau, y, \eta, y, \eta): \tau = \pm |\eta|\}
$$
the fibers corresponding to zero reflections. If we flow out from $A_\pm^0$ by the Hamiltonian flow $\psi^{\pm t}$ of $H = \tau \pm |\xi|$, we obtain $\Gamma_\pm^0$. Note that $\psi^{\pm t}$ consists of geodesics lifted to $T^*(\R \times \Omega)$. Consider the following subsets of $\Gamma_\pm^0$:
\begin{align*}
A_\pm^1 = \{(t,\tau, \psi^{\pm t}(y,\eta), y, \eta): t > 0, \psi^{t}(y,\eta) \in T_{\d \Omega }^* \R^2, \tau = \pm |\eta|\},\\
A_\pm^{-1} = \{(t,\tau, \psi^{\pm t}(y,\eta), y, \eta): t < 0, \psi^{t}(y,\eta) \in T_{\d \Omega }^* \R^2, \tau = \pm |\eta|\}.
\end{align*}
If \, $\widehat{}$ \, denotes reflection in the left factor, we have
\begin{align*}
\Gamma_\pm^{1} = \bigcup_{t \in \R} \widehat{A_\pm^1}, \qquad \Gamma_\pm^{-1} = \bigcup_{t \in \R} \widehat{A_\pm^{-1}} .
\end{align*}
We define $A_\pm^j$ and $\widehat{A_\pm^j}$ similarly and note that these are precisely the fibers of the canonical relations $\Gamma_\pm^j$ lying over the boundary. In the next section, we will compute the principal symbol of the wave propagator in coordinates on the critical set. Therefore, we first need to better understand the forwards and backwards symbols on $\Gamma$.
\begin{prop}\label{symbol prop}
Let $e_\pm$ denote the principal symbol of $\widetilde{S_R}$ on $\Gamma = \bigcup_{j \in \Z, \pm} \Gamma_\pm^j$. Then, we have
$$
e_\pm = \frac{1}{2 \tau i} |dt \wedge dy \wedge d\eta|^{1/2}.
$$
Furthermore, the principal symbol for the wave propagator with $K = 0$ (Neumann boundary conditions) coincides with $e_\pm$ (Robin boundary conditions).
%If we denote by $\Gamma_{\d \Omega}$ the canonical relation of the boundary restriction operator, then on $\Gamma_{\d \Omega} \circ A_\pm^j$ (composed in the appropriate variables), the direct symbol and reflected symbol coincide:
%$$
%\sigma_r \circ e_\pm (t_\pm^j, \pm \tau, \widehat{\lambda_\pm^j (y,\eta)}, y, \eta ) = \sigma_r \circ e_\pm (t_\pm^j, \pm \tau, {\lambda_\pm^j (y,\eta)}, y, \eta ).
%$$
\end{prop}
\begin{proof}
As in \cite{Ch76} and \cite{HeZe12}, denote by $\sigma_0$ the symbol of the restriction to $t = 0$, $\sigma_r$ the symbol of the boundary restriction operator $r$, and $\sigma_B$ the symbol of $rN - K \in \Psi^1$. Here, $N$ is an extension of the unit normal vector field on $\d \Omega$ to a tubular neighborhood of the boundary acting as a differential operator and $K$ is again a Robin function. We have the following implications:
\begin{align}\label{symbol formulas}
\begin{split}
\Box \widetilde{S_R} &= 0 \implies \mathcal{L}_{H} e_\pm = 0,\\
\widetilde{S_R} \big|_{t = 0} &= 0 \implies \sigma_0 \circ e_+ + \sigma_0 \circ e_- = 0,\\
\frac{d}{dt} \widetilde{S_R} \big|_{t = 0} &= Id \implies \tau \sigma_0 \circ e_+ - \tau \sigma_0 \circ e_- = \frac{1}{i}\sigma_{Id},\\
(rN - r K) \widetilde{S_R} &= 0 \implies \sigma_B \circ e_\pm = \sigma_B \circ e_\pm\big|_{A_\pm^j} + \sigma_B \circ e_\pm \big|_{\widehat{A}_\pm^j} = 0.
\end{split}
\end{align}
The first assertion in \eqref{symbol formulas} follows from Theorem 5.3.1 of \cite{DuHo72} and the remaining formulas are clear. At the boundary, the symbol $\sigma_B$ is given by 
$$
\langle \lambda(y,\eta), \nu_y \rangle \sigma_r
$$
on $\Gamma_{\d \Omega} \circ A_\pm^j$ and 
$$\langle \widehat{\lambda(y,\eta)}, \nu_y \rangle \sigma_r = -\langle \lambda(y,\eta), \nu_y \rangle \sigma_r$$ 
on $\Gamma_{\d \Omega} \circ \widehat{A}_\pm^j$. The symbol of $r K$ doesnt appear since multiplication by $K$ is a $\Psi$DO of order $0$ while, $N \in \Psi^1$. Hence, the fourth equation in \eqref{symbol formulas} implies that on the boundary, we have
\begin{align}\label{boundary symbol}
\langle \lambda(y,\eta), \nu_y \rangle \sigma_r \circ e_\pm(\lambda(y,\eta))  - \langle \lambda(y,\eta), \nu_y \rangle \sigma_r\circ e_\pm(\widehat{\lambda(y,\eta)}) = 0.
\end{align}
Note that in equation \eqref{boundary symbol}, both sides involve the composition $\sigma_r \circ e_\pm$. The formula for the principal symbol of the composition of FIOs is quite complicated, but is discussed more thoroughly in \cite{GuMe79b}, \cite{HeZe12}, \cite{DuGu75} and \cite{DuHo72}. Since $\langle \lambda(y,\eta), \nu_y \rangle$ is nonvanishing, equation \eqref{boundary symbol} tells us that the direct and reflected symbols coincide on the boundary. Multiplying the second equation in \eqref{symbol formulas} by $\tau$ and adding/subtracting it to the third equation gives
$$
\sigma_0 \circ e_\pm = \frac{1}{2 \tau i} \sigma_{Id}
$$
in the interior. It is elementary to see that $\sigma_{Id}$ is the canonical half density $|dt \wedge dy \wedge d\eta|^{1/2}$. The first equation in \eqref{symbol formulas} implies that the symbol is invariant under geodesic flow, so the claim follows on $\Gamma_{\pm}^1$. In fact, since we have already noted that the direct and reflected symbols coincide over the boundary, the fact that $\Gamma_{\pm}^j$ is the flowout of $\widehat{A}_\pm^j$ then implies that the claim extends to all $\Gamma_\pm^j$.
\end{proof}

\begin{rema}
Chazarain's parametrix actually computes the full symbol by solving successive transport equations and Borel summing the terms. The full proof and explicit computation of the symbol can be found in the original French paper \cite{Ch76}. The actual solution operator $S_R$ could be obtained from $\widetilde{S_R}$ by adding correction terms via Duhamel's principle. However, we only need the principal symbol in our calculation. In \cite{GuMe79b}, a more general situation is treated in which both $M = \Omega \times \{t = 0\}$ and $M' = \d \Omega \times \R$ are nonglancing, noncharacteristic hypersurfaces for the wave propagator. A Fourier integral operator is then constructed iteratively to solve the localized hyperbolic pseudodifferential equation.
\end{rema}

We now make precise the notion of microlocalized FIOs. Recall Theorem \ref{caustics} in Section \ref{Billiards}, which provides an ample number of caustics having simple length in any neighborhood of $|\d \Omega|$. Hence, for $j \in \Z$ large and positive, we can consider periodic orbits having simple length $T_j$, making a single rotation and precisely $j$ reflections at the boundary. We would like to microlocalize $S_R$ near orbits of such a simple length $T_j$. Let $\chi_1(t)$ be a smooth cutoff function which is identically equal to $1$ on an open neighborhood of $T_j$ and vanishes in a neighborhood of all other $T \in \text{Lsp}(\Omega)$. As we remarked above, each propagator $S_j$ has canonical relations $\Gamma_{\pm}^j$. Denote by $\chi_2$ a smooth cutoff function which is identically equal to $1$ on $\cup_\pm \Gamma_\pm^j$ and is conic in the fiber variables $\tau, \xi$ and $\eta$. Quantizing $\chi_2$ gives a pseudodifferential operator with wavefront set contained in support of $\chi_2$. For a reference, see Chapter 18 of \cite{Ho385}. We call such an operator a microlocal cutoff on $\Gamma_\pm^j$. The composition $\chi_1(t)\chi_2(t,x,y, D_t, D_x, D_y) S_R$ is then smoothing away from the periodic orbits of rotation number $1/j$. Since $T_j$ was assumed to be simple, the trace of the above composition is equal to the wave trace modulo $C^\infty$ in a neighborhood of $T_j$.

\subsection{Computing the singularity in elliptical polar coordinates}\label{Computing the singularity in elliptical polar coordinates}
In the previous section, we reviewed Chazarain's parametrix and computed the principal symbol and canonical relation for the wave propagator. In contrast to the methods employed in \cite{HeZe12}, we now want to cook up an oscillatory integral such that microlocally near $\Gamma_\pm^j$,
$$
S_j(t,x,y) = \int_{-\infty}^\infty e^{i\Theta(t,\tau, x,y)} { \sigma}(\tau, x,y)\,d\tau + L.O.T.
$$
where $S_j$ is the $j^{th}$ term in Chazarain's parametrix corresponding to a wave with $j$ reflections. Here, $L.O.T$ denotes lower order terms in the sense of Lagrangian distributions. Due to the presence of different Maslov factors for $\pm \tau > 0$ (see Sections \ref{Fourier Integral Operators} and \ref{Maslov}), it is actually more convenient to find operators
$$
S_j^\pm(t,x,y) = \int_{0}^\infty e^{i\Theta_\pm(t,\tau, x,y)} { \sigma_\pm}(\tau, x,y)\,d\tau,
$$
so that $S_j = S_j^+ + S_j^-$ and the phase functions associated to $S_j^\pm$ paramaterize $\Gamma_{+}^j$ and $\Gamma_-^j$ individually. We first find suitable phase functions $\Theta$ parametrizing $\Gamma_\pm^j$, which we can do only after learning more about elliptical billiards. The following geometric description of almost periodic orbits in the ellipse is crucial:
\begin{lemm}{\label{phase function}}
For $j \in \Z$ sufficiently large and any two points $x,y \in \text{int} \Omega$ near the diagonal of the boundary, there exist precisely four distinct, broken geodesics of $j$ reflections making approximately one counterclockwise rotation, emanating from $x$ and terminating at $y$. Similarly, there exist four such orbits in the clockwise direction.
\end{lemm}
\noindent We first explain what is meant by approximately one roation. Let $\xi \in S_x^* \Omega$ be one of the $4$ covectors corresponding to the initial condition of a counterclockwise orbit described in Lemma \ref{phase function}. Denote by $\widehat{x} = \pi_1 g^{t_1^+}(x,\xi)$ the first point of reflection at the boundary ($\pi_1$ is projection onto the first factor) and by $\widehat{y}$ the $(j+1)$st point of reflection at the boundary before the orbit reaches $y$. If $x, y$ are $O(j^{-1})$ close to the diagonal of the boundary, then $|\widehat{x} - \widehat{y}| = O(j^{-1})$ (see Section \ref{proof of lemma}). Also let $\omega$ be the angle of reflection made by the orbit at $\widehat{x}$ and note that $\widehat{x}, \widehat{y}$ and $\omega$ all depend implicitly on $\xi$. By approximately one counterclockwise rotation, we mean that for each of the initial covectors $\xi \in S_x^*\Omega$ of the $4$ counterclockwise orbits provided by Lemma \ref{phase function}, we have
\begin{equation*}
|\pi_1 \widehat{\beta}^j(\widehat{x}, \omega) - \widehat{y} - \ell| \leq \ell/100.
\end{equation*}
\noindent Here, $\ell = |\d \Omega|$ and $\widehat{\beta}$ is the lift of the billiard map to the closure of the universal cover $\R \times [0,\pi]$ as described in Section \ref{Billiards}. The choice of $\ell/100$ is somewhat arbitrary, but having $\ell$ in the numerator allows for scale invariance and finding the optimal constant in the denominator is irrelevant for our purposes. The notion of approximately one clockwise rotation is defined similarly. We relegate the proof this theorem to Section \ref{proof of lemma}, as it uses formulas we haven't yet discussed and is of independent interest. The proof of Lemma \ref{phase function} actually provides more information. Of the four counterclockwise orbits emanating from $x$, two of them become tangent to a confocal ellipse before making a reflection at the boundary. We denote these orbits by $T$ orbits (for tangency) and call their first links $T$ links. The other two orbits make a reflection at the boundary before becoming tangent to a confocal ellipse and we call these $N$ orbits (for nontangency) with first link called an $N$ link. Within either $T$ or $N$ category for the first link, the final link of one of the orbits reaches $y$ before becoming tangent to a confocal ellipse (an $N$ link) and the other has a point of tangency before reaching $y$ (a $T$ link). In this way, we obtain four types of counterclockwise orbits from $x$ to $y$, which we denote by $TT$, $TN$, $NT$, and $NN$. See Figure \ref{configurations} for an example with $j=4$. The same characterization also applies to the four clockwise orbits in Lemma \ref{phase function}, which can be obtained by reflecting the domain through the vertical axis, finding all orbits making approximately one counterclockwize rotation from $x$ to $y$, and then reflecting these orbits back through the vertical axis (see Section \ref{proof of lemma} for a more detailed discussion). These configurations will be important in determining which limiting orbits give periodic trajectories of precisely $j$ reflections as $(x,y) \to \Delta \d \Omega$.
\begin{figure}
\includegraphics[scale=0.25]{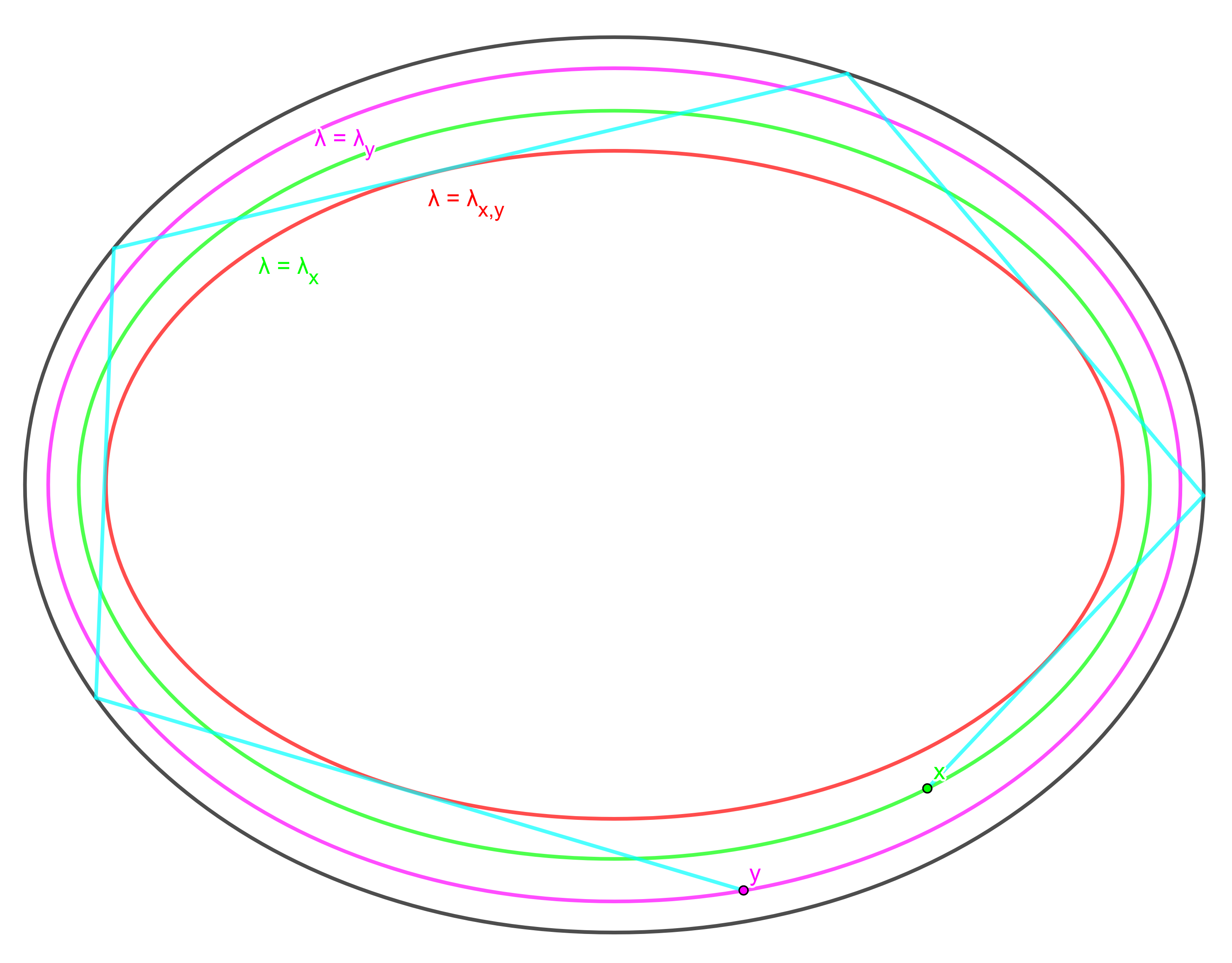}
\includegraphics[scale=0.25]{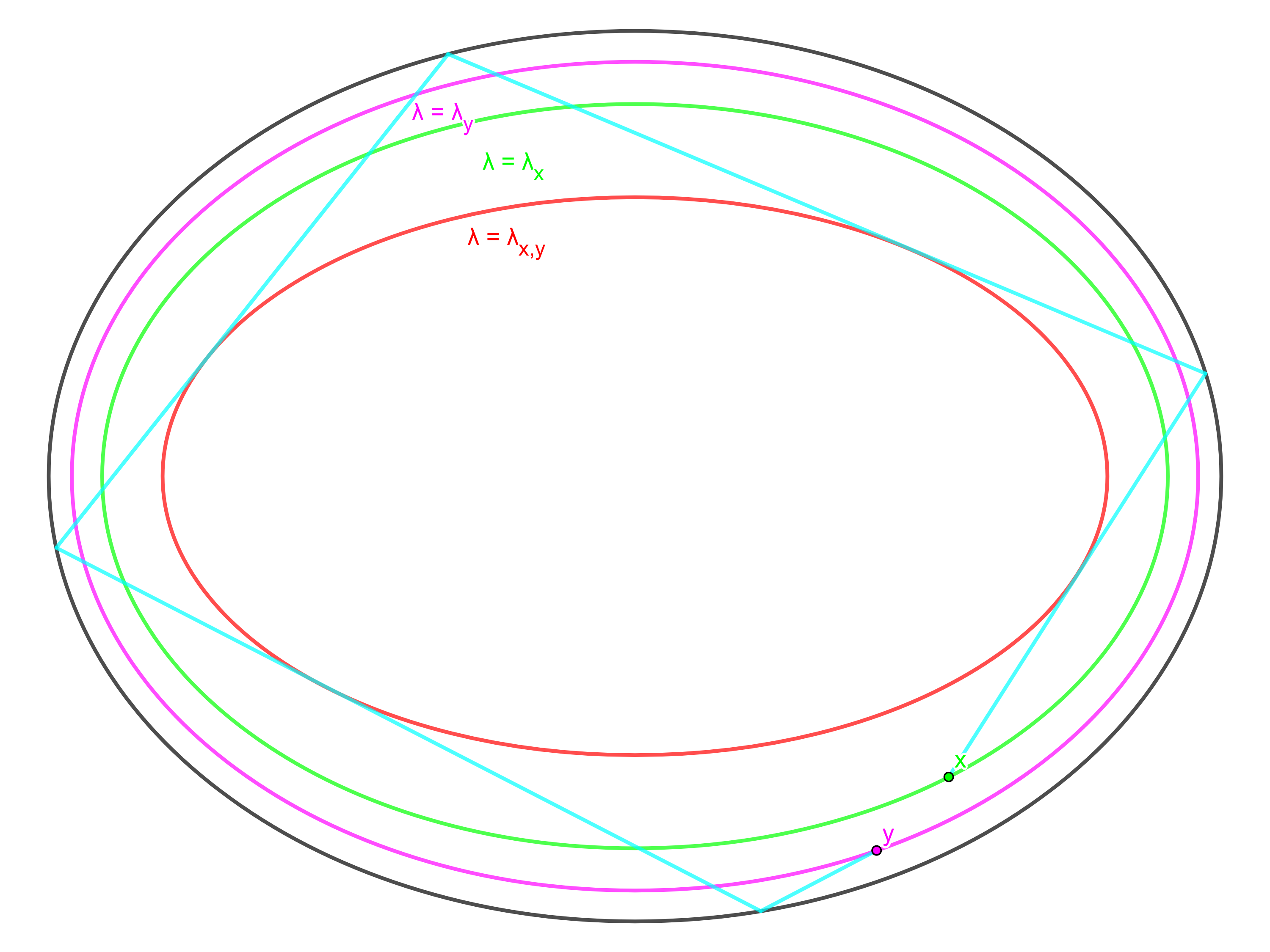}
\includegraphics[scale=0.25]{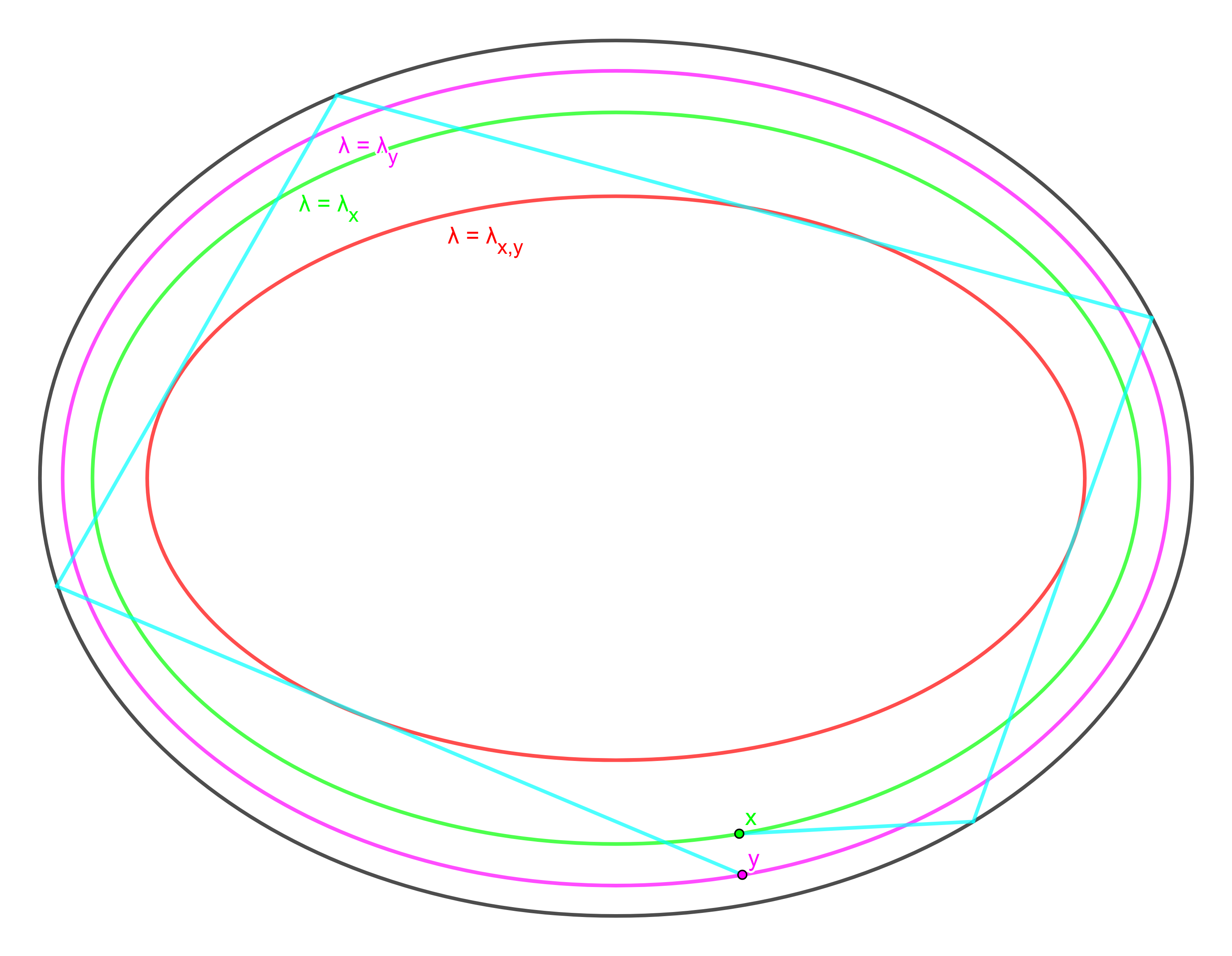}
\includegraphics[scale=0.25]{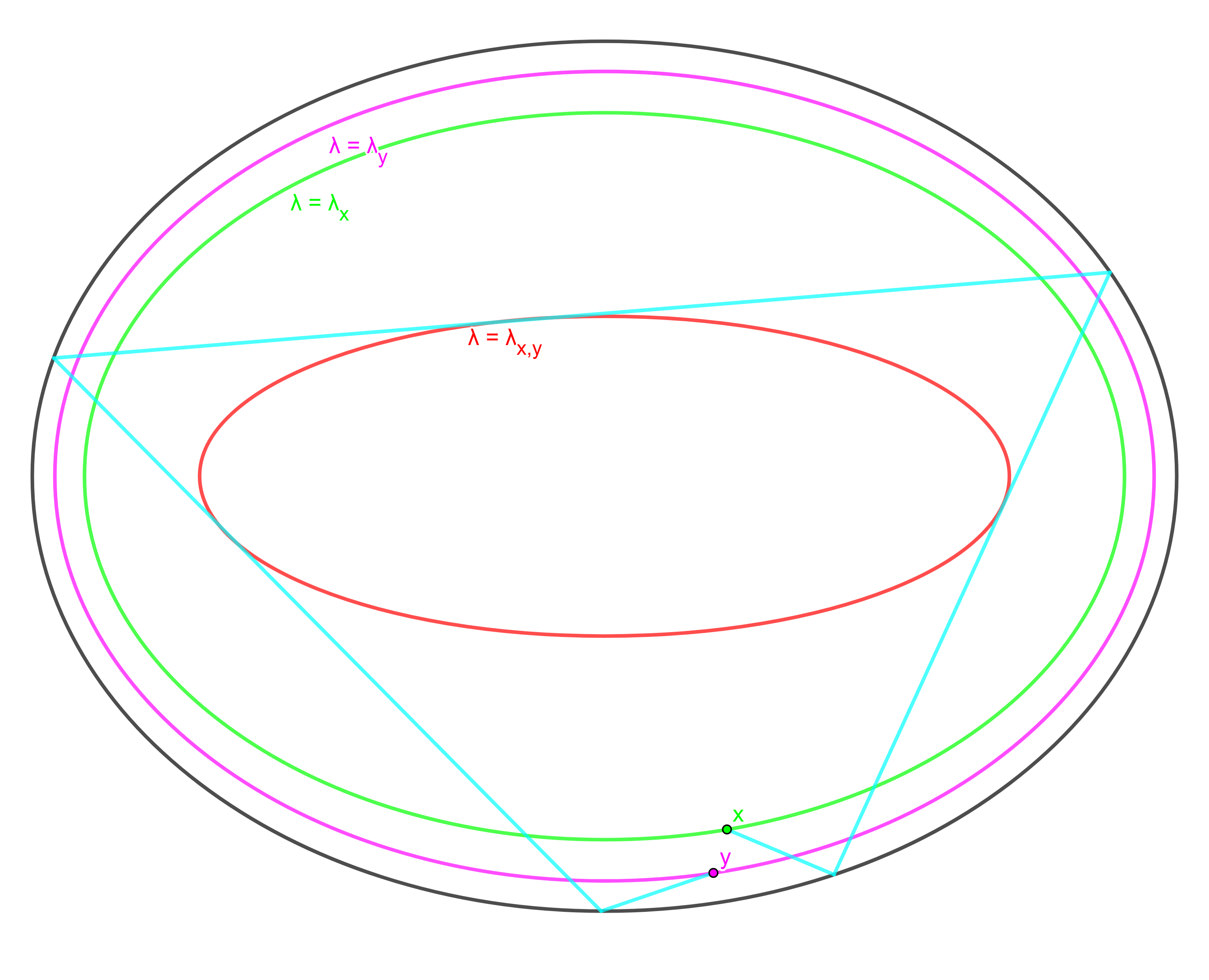}
\caption{Counterclockwise orbit configurations TT, TN, NT, and NN corresponding to $j = 4$. The green and pink curves are the confocal ellipses on which $x$ and $y$ lie, respectively. The red curve is the caustic of parameter $\lambda_{xy}$ to which the billiard orbit is tangent.}
\label{configurations}
\end{figure}
\begin{def1}\label{psi jk}
For $1 \leq k \leq 8$, we set $\Psi_j^k(x,y)$ to be a branch of the length functional corresponding to one of the orbits in Lemma \ref{phase function}. It depends only on $x,y, j$ and $k$. We use the convention that the indices $1 \leq k \leq 4$ correspond to the counterclockwise orbits $TT, TN, NT, NN$ and the indices $5 \leq k \leq 8$ correspond to their reflections about their clockwise counterparts (reflections of the first four orbits through the vertical axis). 
\end{def1}
%\begin{rema}\label{order of configurations}
%We use the convention that the indices $1 \leq k \leq 4$ correspond to the counter clockwise orbits $TT, TN, NT, NN$ and the indices $5 \leq k \leq 8$ correspond to their reflections about their clockwise counterparts (reflections of the first four orbits through the vertical axis). 
%\end{rema}
The author learned of a similar function in \cite{MaMe82} (page 492), where its restriction to the boundary is defined. In such a case, i.e. if $x, y \in \d \Omega$, it is shown in \cite{GuMeCohomological}, \cite{MaMe82} and \cite{Popov1994} that only a single counterclockwise orbit of $j$ reflections exists between the boundary points if they are sufficiently close and $j$ is sufficiently large. Upon inspection of the geometric proof given in Section \ref{proof of lemma}, one can actually see that as $x$ and $y$ approach the diagonal of the boundary from the interior, the corresponding orbits coalesce and converge to the orbits described in \cite{MaMe82}. However, the limiting orbits may have a different number of reflections (see proof of Lemma \ref{Sj lemma}). We define phase functions $\Theta_{j, \pm}^k$ by the formula
$$
\Theta_{j, \pm}^k(t,\tau, x,y) = \pm \tau(t - \Psi_j^k(x,y)).
$$
\begin{lemm}\label{parametrization of canonical relations}
The phase functions $\Theta_{j, \pm}^k(t,\tau,x,y)$ are smooth in an open neighborhood of the diagonal of the boundary and locally parametrize the canonical graphs $\Gamma_\pm^j$. In particular, both $\Gamma_+^j$ and $\Gamma_-^j$ are unions of $8$ connected components, which we denote by $\Gamma_\pm^{j,k}$.
\end{lemm}
\begin{proof}
For any $x,y \in \Omega$ let
\begin{align}\label{length functional}
\begin{cases}
L_{x,y}: \d \Omega^j \to \R_+\\
L_{x,y}(q_1, q_2, \cdots q_j) = |x - q_1| + \left\{ \sum_{m = 2}^j |q_m - q_{m-1}| \right\} + |q_j - y|
\end{cases}
\end{align}
denote the length functional. We first show that billiard trajectories from $x$ to $y$ are in one to one correspondence with critical points of \eqref{length functional} with respect to $q \in \d \Omega^j$. Let $g \in C^\infty(\R^2)$ be a defining function for $\d \Omega$ and consider $q$ as a variable in $\R^2 \times \cdots \times \R^2 = \R^{2j}$ rather than $\d \Omega^{j}$. If $q$ is a critical point of \eqref{length functional}, then as in the method of Lagrange multipliers, by setting $x = q_0$ and $y = q_{j+1}$, we find that for $1 \leq m \leq j$, there exists $\lambda_m \in \R$ such that
$$
\frac{\d L_{x,y}}{\d q_m} = \frac{q_m - q_{m-1}}{|q_m - q_{m-1}|} + \frac{q_m - q_{m+1}}{|q_m - q_{m+1}|} = \lambda_m \nabla_{q_m} g.
$$
Since $\nabla_{q_m} g \perp \d \Omega$, this implies that the two unit vectors in the formula for $\d_{q_m} L_{x,y}$ have opposite tangential components, which is precisely the condition giving elastic collision at the boundary (angle of incedince equals angle of reflection). Similarly, if this condition is satisfied, then $q$ is a critical point for \eqref{length functional}.\\
\\
We now consider the functions $\Psi_j^k$ in Definition \ref{psi jk}. We have
\begin{align}\label{psi j}
\Psi_j^k(x,y) = |x - q_1^k| + \left\{ \sum_{m = 2}^j |q_m^k - q_{m-1}^k| \right\} + |q_j^k - y|,
\end{align}
where $q_m^k(x,y)$ is the $m$th impact point on the boundary for the billiard trajectory corresponding to $\Psi_j^k$. As opposed to the $q_m$ in the length functional \eqref{length functional}, $q_m^k$ will in general have a nontrivial dependence on $x$ and $y$. Differentiating \eqref{psi j} in $x$, we obtain
\begin{align}\label{psi j derivative}
\frac{\d \Psi_j^k}{\d x_i} = \frac{x - q_1^k}{|x - q_1^k|} \cdot \frac{\d}{\d x_i} (x - q_1^k) + \left\{\sum_{m =2}^j \frac{q_m^k - q_{m-1}^k}{|q_m^k - q_{m-1}^k|} \cdot \frac{\d}{\d x_i}(q_m^k - q_{m-1}^k) \right\} + \frac{q_j^k - y}{|q_j^k - y|} \cdot \frac{\d}{\d x_i} (q_j^k - y).
\end{align}
Since for each $x, y \in \Omega$, the path defined by $(x,q^k,y)$ corresponds to a billiard trajectory, we see that all of the terms except the first telescope in \eqref{psi j derivative}. Hence,
\begin{align}\label{psi j x derivative}
d_x \Psi_j^k = \frac{x - q_1^k}{|x - q_1^k|}.
\end{align}
Similarly, differentiating \eqref{psi j} in $y$, we obtain
\begin{align}\label{psi j y derivative}
d_y \Psi_j^k = \frac{y - q_j^k}{|y - q_j^k|}.
\end{align}
Geometrically, these gradients are the incident and (reflected) outgoing unit directions of the billiard trajectories described in Lemma \ref{phase function}.\\
\\
We now consider the maps
\begin{align}\label{parametrize Gamma}
\iota_{{\Theta_{j, \pm}^k}}: (t, \tau, x, y) \mapsto (t,\tau,x, d_x \Theta_j^k,y, -d_y\Theta_j^k) = (t, \tau, x, -\tau d_x \Psi_j^k,y, \tau d_y \Psi_j^k)
\end{align}
on the critical set $C_{\Theta_{j, \pm}^k} = \{ t - \Psi_j^k = 0 \}$. Inserting formulas \eqref{psi j x derivative} and \eqref{psi j y derivative} into \eqref{parametrize Gamma} and comparing with the canonical graphs
$$
\Gamma_\pm^j = \begin{cases}
(t, \tau, g^{\pm t}(y,\eta), y, \eta): \tau = \pm |\eta| & j =0,\\
(t, \tau, g^{\pm (t - T_\pm^j(y,\eta))}\widehat{\lambda_j(y,\eta)}, y, \eta): \tau = \pm |\eta| & j \in \Z \backslash \{0\}
\end{cases}
$$
from Section \ref{Chazarain's Parametrix}, we see that $\iota_{{\Theta_{j, \pm}^k}}: C_{\Theta_{j, \pm}^k} \to \Gamma_\pm^j$ is a local diffeomorphism. Since $1 \leq k \leq 8$, it follows that both $\Gamma_+^j$ and $\Gamma_-^j$ are the unions of $8$ connected components.
\end{proof}
We now want to derive an explicit formula for the principal symbol $e_\pm$ of $S_R$ in coordinates. Referring to formula \eqref{principal symbol of an FIO} for the principal symbol of an FIO, we easily see that in our setting, $dC_{\Theta_{j, \pm}^k} = \mp d\tau \wedge dx \wedge dy$. In Proposition \ref{symbol prop}, we calculated that $e_\pm = \frac{1}{2 \tau i} |dt \wedge dy \wedge d\eta|^{1/2}$. Since we now know that the phase functions $\Theta_{{j, \pm}, \pm}^k(t,\tau,x,y) = \pm \tau(t - \Psi_j^k(x,y))$ locally parametrize the connected components of $\Gamma_\pm^j$, we now want to calculate $a_0$ (as in Section \ref{Fourier Integral Operators}) by changing variables. It is ultimately more convenient to introduce a conformal change of coordinates which is suitable to computing the symbol $e_\pm$ for the ellipse:
\begin{def1}\label{elliptical polar coordinates}
Elliptical polar coordinates are defined on $\R^2 \backslash [-c, c]$ by the equations:
\begin{align*}
x_1= c\cosh \mu \cos \phi, \qquad
x_2 = c\sinh \mu \sin \phi.
\end{align*}
Here, $(\mu, \phi) \in (0,\infty) \times \R/(2\pi\Z)$ and $c = \sqrt{a^2 - b^2}$ is the semifocal distance, i.e. the distance between the origin and a focal point of the ellipse
$$
\Omega = \left\{ (x,y): \frac{x^2}{a^2} + \frac{y^2}{b^2} \leq 1 \right\}.
$$
\end{def1}
It is easy to check that the Euclidean metric in these coordinates is a conformal multiple of $d\mu d\phi$:
$$
dx_1dx_2 = c^2 (\cosh^2 \mu - \cos^2 \phi)d\mu d\phi.
$$
In particular, the vector fields $\d/\d \mu$ and $\d/\d \phi$ are orthogonal at each point. For a fixed $\mu > 0$, the $\phi$ coordinate parametrizes a confocal ellipse of eccentricity $1/\cosh(\mu)$. When projected onto the phase space $B^*\d \Omega$, these curves are precisely the invariant Lagrangian tori for the billiard map. In particular, fixing
$$
\mu = \mu_0 = \cosh^{-1}\left(\frac{a}{\sqrt{a^2 - b^2}}\right)
$$
gives a parametrization of the boundary $\d \Omega$. Similarly, for a fixed $\phi > 0$, the $\mu$ coordinate parametrizes a branch of a confocal hyperbola. The reason we use these coordinates is because up to a conformal factor, tangential differentiation on the ellipse becomes $\d/\d\phi$ while differentiating in the normal direction becomes $\d/\d \mu$. Since the wave kernel $S_R$ is a function of both the $x$ and $y$ variables, we use elliptic coordinates for $y$ as well:
\begin{align*}
y_1= c\cosh \nu \cos \theta, \qquad
y_2 = c\sinh \nu \sin \theta.
\end{align*}
Note that the Leray form is coordinate independent. Hence, without loss of generality, we compute that in elliptical coordinates,
\begin{align*}
dC_{\Theta_{j, \pm}^k} &= \mp d\tau \wedge dx \wedge dy\\
&= \mp a^4 (\cosh^2 \mu - \cos^2 \phi)(\cosh^2 \nu - \cos^2 \theta) d\tau \wedge d\mu \wedge d\phi \wedge d\nu \wedge d\theta.
\end{align*}
On the critical set, we have
\begin{align*}
(t,\tau, \mu, \phi, \xi, \nu, \theta, \eta) = (\Psi_j^k, \tau, \mu, \phi, - \tau d_{\mu,\phi} \Psi_j^k, \nu, \theta, \tau d_{\nu, \theta} \Psi_j^k).
\end{align*}
From now on, we drop the the $j,k$ subscripts and write $\Psi_j^k = \Psi$ so that we may use subscripts to denote derivatives. We have
\begin{align*}
dt &= \Psi_{\mu} d\mu + \Psi_{\phi}d\phi +\Psi_{\nu} d\nu + \Psi_{\theta}d\theta,\\
dy &= a^2(\cosh^2 \nu -\cos^2 \theta) d\nu \wedge d\theta,\\ 
d{\eta_1} &= \Psi_{\nu} d\tau + \tau( \Psi_{\mu \nu} d\mu + \Psi_{\nu \phi} d\phi + \Psi_{\nu \nu} d\nu + \Psi_{\nu \theta} d\theta),\\
d\eta_2 &= \Psi_{\theta} d\tau +  \tau( \Psi_{ \theta\mu} d\mu + \Psi_{\theta \phi} d\phi + \Psi_{ \theta \nu} d\nu + \Psi_{ \theta \theta} d\theta  ).
\end{align*}
Wedging all these terms, we find that
\begin{align*}
dt \wedge dy \wedge d\eta = & \tau a^2 (\cosh^2 \nu - \cos^2 \theta)  ( \Psi_{\mu} \Psi_{ \theta} \Psi_{\nu \phi} + \Psi_{\phi} \Psi_{\nu} \Psi_{ \theta \mu}\\
&- \Psi_{\mu}\Psi_{\nu} \Psi_{\theta \phi} - \Psi_{\phi} \Psi_{\theta} \Psi_{\nu \mu} ) d\tau \wedge d\mu \wedge d\phi \wedge d\nu \wedge d\theta.
\end{align*}
Keeping in mind that all $\Psi$ terms depend on $j$ and $k$, we denote the above factor by
\begin{align}\label{def of Aj}
A_j^k(\mu,\phi,\nu, \theta) = a^2 (\cosh^2 \nu - \cos^2 \theta)  ( \Psi_{\mu} \Psi_{ \theta} \Psi_{\nu \phi} + \Psi_{\phi} \Psi_{\nu} \Psi_{ \theta \mu} - \Psi_{\mu}\Psi_{\nu} \Psi_{\theta \phi} - \Psi_{\phi} \Psi_{\theta} \Psi_{\nu \mu} ).
\end{align}
Then, on each of the canonical relations $\Gamma_\pm^{j,k}$ (see Lemma \ref{parametrization of canonical relations}), we have
\begin{align*}
e_\pm = \frac{|dt \wedge dy \wedge d\eta|^{1/2}}{2\tau i} = \frac{\pm 1}{2 |\tau|^{1/2} i}\left|A_j^k d\tau \wedge d\mu \wedge d\phi \wedge d \nu \wedge d\theta \right|^{1/2}.
\end{align*}
As a result, we have proved:
\begin{theo}\label{parametrix with sum}
Microlocally near $\Gamma_\pm^j$, the following oscillatory integral is a parametrix for $S_j$ in an open neighborhood of $\Delta \d \Omega \subset \Omega \times \Omega$:
\begin{align*}
S_j(t, \mu, \phi,\nu, \theta) = \sum_{\pm, k = 1}^8 \int_{0}^\infty e^{\pm i\tau(t - \Psi_j^k(x,y))} \frac{\pm 1}{2|\tau|^{1/2} i} |A_j^k(\mu, \phi,\nu, \theta)|^{1/2} d\tau + L.O.T.
\end{align*}
We denote the operators in this sum by $S_{j, \pm}^k$ and and also define $S_j^k = S_{j, +}^k + S_{j,-}^k$.
\end{theo}
Recall that according to Theorem \ref{wavetrace}, the variation of the wave trace is given by
$$
\delta \text{Tr} \cos t \sqrt{-\Delta_\epsilon} = \int_{\d \Omega} L_R^b(t,q,q)\dot{\rho} + \frac{t}{2}S_R(t,q,q)\dot{K}\,dq,
$$
where $L_R^b$ is defined as
$$
L_R^b(t,q,q') = \frac{t}{2} (-\nabla_1^T \nabla_2^T  -\Delta_2  + K_0^2 + K_0 \kappa) S_R.
$$
The highest order terms come only from the differentiated sine kernels, which implies that
$$
\delta \text{Tr} \cos t \sqrt{-\Delta_\epsilon} = \int_{\d \Omega} \frac{t}{2} ((-\nabla_1^T \nabla_2^T - \Delta_2)S_R)(t,q,q) \dot{\rho}\,dq + L.O.T.
$$
In fact, we can discard even more terms in the singularity expansion:
\begin{lemm}\label{Sj lemma}
Modulo Maslov factors and distributions of lower order, the variation of the localized (even) wave trace near a simple length $T_j$ is given by
\begin{align*}
\sum_\pm \int_{\d \Omega} \frac{t}{2} ((-\nabla_1^T \nabla_2^T& - \Delta_2)(S_{j-1,\pm}^1  +S_{j-1,\pm}^5 + S_{j,\pm}^2 + S_{j,\pm}^3 \\
&+ S_{j,\pm}^6 + S_{j,\pm}^7 + S_{j + 1, \pm}^4 +  S_{j +1,\pm}^8 ))(t,q,q) \dot{\rho}\,dq
\end{align*}
\end{lemm}
\begin{proof}
For the localized wave trace, we only need to consider orbits which contribute to the singularity at $T_j$. Recall that for positive time, Lemma \ref{phase function} gives $8$ orbits connecting $x$ to $y$. These orbits coalesce into one of the orbits from \cite{MaMe82} as $(x, y) \to \Delta \d \Omega$. However, as the orbits coalesce within various configurations, not all of the limiting orbits will have $j$ reflections. As $T_j$ is simple, only the limiting periodic orbits having exactly $j$ reflections will contribute to the wave trace near $t = T_j$. Figure \ref{configurations} may be useful in visualizing the geometric arguments which follow. As $(x, y) \to \Delta \d \Omega$, the two corresponding orbits in $TT$ configuration $(k = 1,5)$ converge geometrically to a periodic orbit of $j+1$ reflections. The additional vertex appears at the boundary point where $x$ and $y$ coalesce. Similarly, the $NN$ orbits $(k = 4,8)$ can be seen to converge to a periodic orbit of $j-1$ reflections. In this case, the first and last moments of reflection at the boundary converge to a single impact point. The four orbits in $TN$ $(k = 2,6)$ and $NT$ $(k = 3,7)$ configurations preserve exactly $j$ reflections in the limit. Hence, when $x, y \in \text{int} \Omega$ converge to the boundary, only $4$ out of the $8$ orbits contribute to periodic trajectories of $j$ reflections on the boundary. However, in the limit, two additional $TT$ orbits of $j-1$ reflections converge to a periodic orbit of $j - 1 + 1 = j$ reflections. Similarly, two $NN$ orbits of $j+1$ reflections converge to a periodic trajectory of $j+1 -1 = j$ reflections. Any other orbit from $x$ to $y$ with strictly less than $j-1$ or strictly more than $j+1$ impact points at the boundary cannot converge to a periodic orbit of $j$ reflections. As we have localized the wave trace near the simple length $T_j$, only the $4 + 2 + 2 = 8$ orbits which converge geometrically to a periodic orbit of exactly $j$ reflections will contribute to the singularity near $T_j$. All additional orbits contribute smooth errors to the wave trace in a small neighborhood of $T_j$. It should also be clarified that although the parametrices $S_j$ are constructed in the interior, we can in fact extend them continuously to the diagonal of the boundary and this extension coincides with that of the true propagator $S_R(t,q,q)$ modulo lower order terms. Both propagators agree up to lower order Lagrangian distributions in the interior, microlocally near the canonical relations $\Gamma_\pm^j$. The explicit oscillatory integral representation for each $S_j$ in fact shows that they extend continuously up to the boundary since the functions $\Psi_j^k$ do. The true wave kernels $S_R$ also extend continuously up to the boundary as a family of distributions. To see this, note that
$$
S_R(t,x,y) = \sum_j \sin t \lambda_j \psi_j(x) \overline{\psi_j(y)},
$$
where $(\psi_j)_{j =1}^\infty$ is an $L^2$ orthonormal basis of Robin eignenfunctions. Multiplying by a test function $\phi$ in time and integrating by parts $4k$ times, we see that
\begin{align}\label{wave kernel}
\int_{-\infty}^\infty S_R(t,x,y) \phi(t) dt = \int_{-\infty}^\infty \sum_j \frac{\sin t \lambda_j}{\lambda_j^{4k}} \psi_j(x) \psi_j(y) \d_t^{4k}  \phi(t) dt. 
\end{align}
Combining Weyl's law on the asymptotic growth of $\lambda_j$ (see \cite{Zayed2004}, \cite{Irvii16}) and H\"ormander's $L^\infty$ bounds for eigenfunctions (\cite{HormanderLinfinity}), we see that the integrand in \eqref{wave kernel} can be made absolutely convergent for $k$ sufficiently large. An application of the dominated convergence theorem then shows that \eqref{wave kernel} is actually smooth in $x,y$, so $S_R$ has a smooth extension to the diagonal of the boundary as a distribution in $t$. In particular, both distributions agree up to lower order terms microlocally near the fibers of $\Gamma_\pm^j$ lying over diagonal of the boundary, which is required for the trace formula.
\end{proof}
\begin{def1}\label{psi j coincide on boundary}
As shown in the proof of Lemma \ref{Sj lemma} above, for each $j$, there exist $8$ limiting trajectories which converge geometrically to periodic orbits of exactly $j$ reflections. We denote the set of these trajectories by $\mathcal{G}_j(x,y)$ and say that $\gamma_{m,k} \in \mathcal{G}_j$ if $\gamma_{m,k}$ makes $m = j-1, j$ or $j+1$ reflections at the boundary and corresponds to the length functional $\Psi_m^k$. By the results in \cite{MaMe82}, \cite{GuMe79a} and \cite{Popov1994}, the length functionals $\Psi_j^2,\Psi_j^3 , \Psi_j^6 , \Psi_j^7 , \Psi_{j + 1}^4 , \Psi_{j +1}^8 , \Psi_{j-1}^1$ and $\Psi_{j-1}^5$ corresponding to orbits in $\mathcal{G}_j$ actually coincide for $x,y \in \d \Omega$ near the diagonal. We denote their common value by $\Psi_j$.
\end{def1}
%As the integral for the variation of the wave trace is evaluated over the diagonal of the boundary, we can actually replace the sum in Theorem \ref{parametrix with sum} by a factor of $8$.
As we obtained a rather explicit formula for $S_j$ in Theorem \ref{parametrix with sum}, it now remains to differentiate the kernels $S_j^k$ and substitue them into Lemma \ref{Sj lemma}. %Since $(\mu,\phi,\nu,\theta)$ are conformal to $(x,y)$, we find that in elliptical coordinates,
%\begin{align*}
%\nabla_1^T &= {({a^2(\cosh^2 \mu - \cos^2 \phi)})}^{-1/2} \frac{\d}{\d\phi},\\
%\nabla_2^T &= {({a^2(\cosh^2 \nu - \cos^2 \theta)})}^{-1/2} \frac{\d}{\d\theta},\\
%\Delta_2 &= (a^2(\cosh^2 \nu - \cos^2 \theta))^{-1} \left( \frac{\d^2}{\d \mu^2} + \frac{\d^2}{\d \nu^2} \right).
%\end{align*}
Using our oscillatory integral representation for $S_{j, \pm}^k$ in Theorem \ref{parametrix with sum}, we find that microlocally near $\Gamma_\pm^{j,k}$ and $t = T_j$,
{\begin{align}\label{parametrix with 8}
(- \nabla_1^T \nabla_2^T - \Delta_2) S_{j, \pm}^k = \frac{\mp t}{2} \int_{0}^\infty e^{\pm i \tau(t - \Psi_j^k(x,y))} \frac{ (-i\tau)^2|\nabla_2^\perp \Psi_j^k|^2 {|A_j^k|^{1/2}}}{2 |\tau|^{1/2} i} \dot{\rho}\,d\tau + L.O.T.
\end{align}}
in an open neighborhood of the diagonal of the boundary. We have only written the terms coming from $-\nabla^T \nabla^T - \Delta$ acting on the phase function in equation \eqref{parametrix with 8}, as all other terms don't contribute positive powers of $\tau$ and can be regarded as lower order in the singularity expansion. The operator $\nabla^\perp$ in the integrand of \eqref{parametrix with 8} is a conformal multiple of the vector field $\frac{\d}{\d \nu}$ coming from elliptical polar coordinates, which gives an extension of the normal vector field to a neighborhood of the boundary. As Theorem \ref{wavetrace} tells us that the variation of the wave trace is given by integrating the kernels $L_R^b$ and $S_R$ over the diagonal of the boundary, we want to understand the restriction of \eqref{parametrix with 8} to the boundary. In Definition \ref{psi j coincide on boundary}, we noted that the length of the unique orbit connecting two boundary points with $j$ reflections is well defined. When $x = y = q \in \d \Omega$, $\Psi_j(q,q)$ gives precisely the length $T_j$ of a periodic orbit with $j$ reflections emanating from and terminating at $q$. By Poncelet's Theorem (\ref{poncelet}), $\Psi_j(q,q)$ is actually equal to the constant function $T_j$, which simplifies the phase in equation \eqref{parametrix with 8}. The differentiated kernels in equation \eqref{parametrix with 8} also have factors of $A_j^k$ and $\nabla_2^\perp \Psi_j^k$ in the integrand. We now discuss how to extend these derivatives of $\Psi_j^k$ to the boundary in a manner analogous to that of Definition \ref{psi j coincide on boundary}.
\\
\\
We have already computed the $x,y$ gradient of the functions $\Psi_j^k$ in equations \eqref{psi j x derivative} and \eqref{psi j y derivative} in the proof of Lemma \ref{parametrization of canonical relations}:
$$
d_x \Psi_j^k = \frac{x - q_1^k}{|x - q_1^k|}, \quad d_y \Psi_j^k = \frac{y - q_j^k}{|y - q_j^k|}.
$$
Geometrically, these are the incident and reflected outgoing unit directions of the corresponding billiard trajectories at $x$ and $y$. The expression $|\nabla_y^\perp \Psi_j^k|^2$ in \eqref{parametrix with 8} can easily be seen to be $\sin^2 \omega_j^k$, where $\omega_j^k$ is the angle made between the terminal link of the billiard trajectory and the positively oriented tangent line to the confocal ellipse on which $y$ lies. As $x,y \to \d \Omega$, the absolute value of these angles associated to trajectories in the $\mathcal{G}_j$ converge to the terminal angle of the unique limiting orbit connecting boundary points in \cite{MaMe82}. We are careful to point out that only the absolute values of the angles converge, since the angles associated to orbits in $TN$ and $NN$ configurations actually converge to \textit{minus} the angle of incidence of the limiting trajectory. All of the limiting orbits which connect boundary points in \cite{MaMe82} are automatically in $TT$ configuration.
\\
\\
In order to understand the factor $A_j^k$, we must compute the Hessian of $\Psi_j^k$ and its restriction to the boundary. Recall that $A_j^k$ is given by equation \eqref{def of Aj} in elliptical polar coordinates. We can further simplify that expression to obtain the following:
\begin{lemm}\label{Aj w derivatives}
On the boundary, all of the factors $A_{j-1}^1, A_{j-1}^5 A_j^2, A_j^3, A_j^6, A_j^7, A_{j+1}^4$ and $A_j^8$ corresponding to orbits in $\mathcal{G}_j$ coincide up to sign. We denote their common (absolute) value by $|A_j|$, which on $\Delta \d \Omega$ in particular, satisfies the equation
\begin{align*}
|A_j(\mu_0, \theta, \mu_0, \theta)| = \left|\frac{f^5(\mu_0, \theta)}{\sin \omega} \frac{\d \omega}{\d \theta} \right|.
\end{align*}
Here, $\omega$ is the angle of incidence of the unique periodic orbit with $j$ reflections at $(a \cos \theta, b \sin \theta) \in \d \Omega$ and $f(\mu_0, \theta) = {({a^2(\cosh^2 \mu_0 - \cos^2 \theta)})}^{1/2}$ is the inverse of the conformal factor in elliptical coordinates.
\end{lemm}
\begin{proof}
Let $(m,k)$ denote an admissible pair of indices corresponding to an orbit $\gamma_{m,k} \in \mathcal{G}_j$. Recall the notation in the proof of Lemma \ref{parametrization of canonical relations}, where we described a billiard trajectory by the point $(x,q,y) \in \Omega \times \d \Omega^m \times \Omega$. If $x \in \Omega$, let us denote the angle between $x - q_1$ and the positively oriented tangent line to the confocal ellipse on which $x$ lies by $\omega_1$. Similarly, if $y \in \Omega$, let us also denote the angle between $y- q_j$ and positively oriented tangent line to the confocal ellipse on which $y$ lies by $\omega_2$. Since the elliptical coordinates $(\mu, \phi, \nu, \theta)$ are conformally equivalent to Euclidean $(x,y)$ coordinates, we have
\begin{align*}
\nabla_x^T &= {({a^2(\cosh^2 \mu - \cos^2 \phi)})}^{-1/2} \frac{\d}{\d\phi}, \qquad
\nabla_y^T = {({a^2(\cosh^2 \nu - \cos^2 \theta)})}^{-1/2} \frac{\d}{\d\theta},\\
\nabla_x^\perp &= {({a^2(\cosh^2 \mu - \cos^2 \phi)})}^{-1/2} \frac{\d}{\d\mu}, \qquad
\nabla_y^\perp = {({a^2(\cosh^2 \nu - \cos^2 \theta)})}^{-1/2} \frac{\d}{\d\nu}.
\end{align*}
Equations \eqref{psi j x derivative} and \eqref{psi j y derivative} in the proof of Lemma \ref{parametrization of canonical relations} then tell us that
\begin{align}\label{first derivatives}
\begin{split}
\Psi_\mu &= \pm f(\mu, \phi) \sin \omega_1, \qquad \Psi_\nu = \pm f(\nu, \theta) \sin \omega_2,\\
\Psi_\phi &= - f(\mu, \phi) \cos \omega_1, \qquad \Psi_\theta = f(\nu, \theta) \cos \omega_2,
\end{split}
\end{align}
where the $\pm$ in the equations for $\Psi_\mu$ and $\Psi_\nu$ are dependent on the configuration of the orbit. They are $+$ if the corresponding initial or final link is a $T$ link and $-$ if it is an $N$ link. Using \eqref{first derivatives} to calculate the second derivatives, we have
\begin{align}\label{second derivatives}
\begin{split}
\Psi_{ \theta \phi} &= - f(\mu,\phi) \sin \omega_1 \frac{\d\omega_1}{\d \theta}, \qquad 
\Psi_{ \nu \phi} = f(\mu,\phi) \sin \omega_1 \frac{\d\omega_1}{\d \nu},\\
\Psi_{\theta \mu} &= \pm f(\mu, \phi) \cos \omega_1 \frac{\d\omega_1}{\d \theta}, \qquad \Psi_{\nu \mu} = \pm f(\mu,\phi) \cos \omega_1 \frac{\d\omega_1}{\d \nu}.
\end{split}
\end{align}
Then, inserting \eqref{second derivatives} into the expression \eqref{def of Aj} for all possible configurations, we find that on the boundary,
\begin{align}\label{new Aj w derivatives}
A_m^k(\mu_0, \phi, \mu_0, \theta) =
\begin{cases}
f^3(\mu_0, \theta) f^2(\mu_0, \phi) (\cos \omega_2 \frac{\d\omega_1}{\d \nu} - \sin \omega_2 \frac{\d \omega_1}{\d \theta}) & \gamma_{m,k} \in \mathcal{G}_j^{TT} \cup \mathcal{G}_j^{NT}\\
f^3(\mu_0, \theta) f^2(\mu_0, \phi) (\cos \omega_2 \frac{\d\omega_1}{\d \nu} + \sin \omega_2 \frac{\d \omega_1}{\d \theta}) & \gamma_{m,k} \in \mathcal{G}_j^{TN}\\
f^3(\mu_0, \theta) f^2(\mu_0, \phi) (-\cos \omega_2 \frac{\d\omega_1}{\d \nu} - \sin \omega_2 \frac{\d \omega_1}{\d \theta}) & \gamma_{m,k} \in \mathcal{G}_j^{NN},
\end{cases}
\end{align}
where the superscripts on $\mathcal{G}_j$ indicate the subcollection of orbits within a particular configuration. Before evaluating this expression on the diagonal of the boundary, we differentiate $\omega_1$ in the direction $L$ of the last link to see that
\begin{align*}
\nabla_L \omega_1 = 0 &= \cos \omega_2 \nabla^T \omega_1 \pm \sin \omega_2 \nabla^\perp \omega_1 = \frac{1}{f} \cos \omega_2  \frac{\d \omega_1}{\d \theta} \pm \frac{1}{f} \sin \omega_2  \frac{\d \omega_1}{\d \nu},
\end{align*}
where the $\pm$ correspond to whether the last link is a $T$ link $(+)$ or an $N$ link $(-)$. This implies that,
\begin{align}\label{d omega d nu}
\frac{\d \omega}{\d \nu} = \begin{cases}
\cot \omega_2 \frac{\d \omega_1}{\d \theta} & \gamma_{m,n} \in \mathcal{G}_j^{TN} \cup \mathcal{G}_j^{NN}\\
- \cot \omega_2 \frac{\d \omega_1}{\d \theta} & \gamma_{m,n} \in \mathcal{G}_j^{TT} \cup \mathcal{G}_j^{NT}.
\end{cases}
\end{align}
Note that on the diagonal of the boundary, $\omega_1 = \omega_2$ by the law of equal reflection for billiards. We denote their common value by $\omega$. Inserting formula \eqref{d omega d nu} into \eqref{new Aj w derivatives} and evaluating on $\Delta \d \Omega$, we find that
\begin{align*}
|A_m^k(\mu_0, \phi, \mu_0, \theta)| =  f^5 \left(\frac{\cos^2 \omega}{\sin \omega} + \frac{\sin^2 \omega}{\sin \omega}  \right) \left| \frac{\d \omega}{\d \theta} \right|  = \frac{ f^5}{\sin \omega} \left| \frac{\d \omega}{\d \theta} \right|,
\end{align*}
which proves the lemma.
\end{proof}
Substituting the formula for $A_j$ in Lemma \ref{Aj w derivatives} into \eqref{parametrix with 8} and performing the integral \eqref{parametrix with 8} in $\tau$, we obtain:
\begin{coro}\label{distribution with Aj factor}
The variation of the wave trace localized near $t = T_j$ is
\begin{align*}
%\delta \text{Tr} \cos t &\sqrt{-\Delta_\epsilon} =
4t \Re \left\{ e^{i\sigma \pi /4}(t - T_j - i0^+)^{-5/2}\right\} \int_{\d \Omega} |\nabla^\perp \Psi_j|^2 {|A_j(q,q)|^{1/2}} \dot{\rho}(q) \,dq + L.O.T.,
\end{align*}
where $\sigma$ is a Maslov index.
\end{coro}
\begin{proof}
For a fixed $j$, the parametrices corresponding to $\Gamma_+^j$ and $\Gamma_-^j$ are first multiplied by Maslov factors of the form $e^{\pm i\sigma \pi/4}$. This is due to the multiplicity of phase functions parametrizing the canonical relation of the wave propagator, as briefly described in Section \ref{Fourier Integral Operators}. It is well known that the Maslov factors on the two branches $\Gamma_-^j$ and $\Gamma_+^j$ of the canoncial relation are conjugate to one another, owing to the two modes of propagation:
$$
\sin t \sqrt{-\Delta} = \frac{e^{i t\sqrt{-\Delta}} - e^{-it\sqrt{-\Delta}}}{2 i t \sqrt{-\Delta}}.
$$
While in principle, the Maslov indices might depend on $j$, it is shown in \cite{GuMe79b} that the Maslov indices in Chazarain's parametrix remain unchanged after a reflection at the boundary. The Maslov indices $\pm \sigma$ will in fact be explicitly computed in Section \ref{Maslov} below. The contributions of the wave kernels corresponding to $\Gamma_+^j$ and $\Gamma_-^j$ are then added and the result follows from a limiting argument for the Fourier transform of the homogeneous distribution $\tau^{3/2} \mathbbm{1}_{(0,\infty)}$, as can be found in Chapter 7 of \cite{Ho90}. Since the $8$ terms in \eqref{parametrix with 8} corresponding to orbits in $\mathcal{G}_j$ coincide on the boundary, we multiply the final integral by a factor of $8$.
\end{proof}

\subsection{Computing the Maslov index}\label{Maslov}
To explicitly compute the Maslov factors $e^{i \pi \sigma_j^\pm / 4}$ on $\Gamma_{\pm}^{j}$, we use an argument due to Keller (\cite{Keller}), following the presentation in \cite{HassellHillairetFord}. The free wave propagator $U(t) = e^{-i t \sqrt{-\Delta}}$ on $\R^2$ has an integral kernel given by
\begin{align}
U(t,x,y) = \int_{\R_\xi^2} e^{i(\langle x-y, \xi \rangle - |\xi| t)}d\xi|dx \wedge dy|^{1/2},
\end{align}
considered as a distributional half density (see Section \ref{Fourier Integral Operators}). By changing variables and applying the method of stationary phase, it is shown in \cite{HassellHillairetFord} that the principal symbol of $U(t)$ on $N^*\{|x-y| = t\} = \Gamma_{\pm}^0$ is
\begin{align}\label{Mas}
e^{-i\pi/4} \left(\frac{\tau}{t}\right)^{1/2} |d\tau \wedge dx \wedge dy|^{1/2}.
\end{align}
Hence, the Maslov indices are given by $\sigma_0^{\pm} = \pm 1$ on $\Gamma_{\pm}^0$. As mentioned in the proof of Corollary \ref{distribution with Aj factor}, the arguments in \cite{GuMe79b}, which are in turn based on the construction in \cite{Ch76}, then show that after a reflection at the boundary the Maslov factors remain unchanged. Hence, $\sigma_j^\pm = \pm 1$ for all $j \in \Z$. Since both the kernels on $\Gamma_{\pm}^j$ contribute to the wave trace singularity near $T_j$, we sum together the contributions of $S_j^-$ and $S_j^+$ which explains the real parts in Theorem \ref{wavetrace variation} and Corollary \ref{distribution with Aj factor}.

%\subsection{Poincar\'e map and Leray measure on Caustics}\label{Poincare map and Leray measure on Caustics}

\subsection{Calculating $\frac{\d \omega}{\d \theta}$}\label{d omega d theta}
To calculate the angular derivative in Lemma \ref{Aj w derivatives}, we will first relate it to the billiard map and then utilize some special dynamical properties of the ellipse in action angle coordinates. Recall that the billiard map takes place on the coball bundle $B^*\d \Omega$, which is diffeomorphic to the inward facing portion of the circle bundle with footpoints on the boundary and can be parametrized by coordinates $(\phi, \omega) \in S^1 \times S^1$ (although we only consider the nontangential, inward pointing directions corresponding to $0 < \omega < \pi$). Between any two points $x,y \in \d \Omega$, the results of \cite{MaMe82} show that there exists a unique broken geodesic of $j$ reflections emanating from $x$ and terminating at $y$. This geodesic makes an initial angle of $\omega$ (depending on both $x$ and $y$) with the tangent line $T_x \d \Omega$. Setting $x = y$ above gives the angle $\omega(x)$ corresponding to a periodic orbit, which we considered in Lemma \ref{Aj w derivatives}. Letting $\phi, \theta \in S^1$ be the angular variables which parametrize $x$ and $y$ respectively in elliptical polar coordinates, we need to calculate the quantity $\frac{\d \omega}{\d \theta}$, evaluated on the diagonal $\{\phi = \theta\}$. Consider the map
\begin{align}\label{definition of Bj}
\begin{cases}
B_j : S^1 \times (0,\pi) \to S^1,\\
B_j(\phi,\omega(\phi, \theta)) = \pi_1 \circ \beta^{j+1}(\phi, \omega) = \theta,
\end{cases}
\end{align}
where $\pi_1$ is the projection onto the first factor. Fixing $\phi$ and differentiating both sides in $\theta$ gives
\begin{align*}
\frac{\d B_j}{\d \omega} (x, \omega(\theta)) \frac{\d \omega}{\d \theta} = 1.
\end{align*}
Hence, we have
\begin{align}\label{formula for domega d theta}
\frac{\d \omega}{\d \theta} = \frac{1}{\frac{\d B_j}{\d \omega} (x, \omega(\theta))}.
\end{align}
We will use formula \eqref{formula for domega d theta} to calculate $\frac{\d \omega}{\d \theta}$. Recall that the linearized Poincar\'e map of the iterated billiard map $\beta^{j+1}$ at a periodic point $(\phi,\omega)$ is given by
\begin{align}
P_{j+1}(\phi,\omega) = \begin{pmatrix}
\frac{\d \beta_1^{j+1}}{\d \phi} & \frac{\d \beta_1^{j+1}}{\d \omega}\\
\frac{\d \beta_2^{j+1}}{\d \phi} & \frac{\d \beta_2^{j+1}}{\d \omega}
\end{pmatrix},
\end{align}
where $\beta_1^{j+1}$ and $\beta_2^{j+1}$ are the first and second components (in elliptical coordinates) of $\beta^{j+1}$. We are precisely interested in the $(1,2)$ entry of this matrix.\\
\\
To evaluate this quantity, we will use action angle coordinates for elliptical billiards, which we now describe, following the presentation in \cite{KaSo16}, \cite{DCR17} and \cite{ChFr88}. We begin by developing some basic elliptic function theory.
%Recall that a dynamical system is said to be completely integrable (in the sense of Liouville) if the phase space can be folliated by caustics (invariant submanifolds). Billiards are an example of a Hamiltonian system and in the setting of the ellipse, Poncelet's theorem tells us that the confocal ellipses are caustics for the billiard map which folliate the interior of $\Omega \backslash [-c,c]$. In order to describe the action angle coordinates,
\\
\\
Elliptic functions and elliptic integrals were first studied in the context of computing the arclength of an ellipse. It is therefore no surprise that these same objects appear naturally in the study of elliptical billiards. Formally, an elliptic function is given by a doubly periodic, meromorphic function on the complex plane. One way to obtain elliptic functions is by inverting elliptic integrals:
\begin{def1}
An incomplete elliptic integral of the first kind is an integral of the form
$$
F(\phi, k) = \int_0^\phi \frac{d\tau}{\sqrt{1 - k^2 \sin^2\tau}}.
$$
The quantity $\phi$ is referred to as the amplitude and $k$ is the modulus. A complete elliptic integral of the first kind is given by fixing $\phi = \pi/2$:
$$
K(k) = F(\pi/2, k).
$$
\end{def1}
Note that for a fixed $k$, $F(\phi,k)$ is an increasing function of $\phi$. The amplitude function $\am(s;k)$ is obtained by inverting $F$ in the variable $\phi$:
\begin{align}\label{amplitude inverse}
s = \int_0^{\am(s;k)} \frac{d\tau}{\sqrt{1 - k^2 \sin^2 \tau}}.
\end{align}
\begin{def1}
The Jacobi elliptic functions are defined by
\begin{align*}
\cn(s;k) = \cos(\am(s;k)),\\
\sn(s;k) = \sin(\am(s;k)).
\end{align*}
\end{def1}
These are elliptic functions with periods $4K(k)$ and $4iK(k')$, where $k \in (0,1)$ is called the modulus and $k' = \sqrt{1- k^2}$ is the complimentary modulus. The reason these elliptic functions are useful is that they provide coordinates on phase space in which the billiard map becomes a simple translation. To the confocal ellipse
$$
C_\lambda = \left\{(x,y) \in \R^2:\frac{x^2}{a^2 - \lambda^2} + \frac{y^2}{b^2 - \lambda^2}  = 1 \right\},
$$
we associate the following parameters:
\begin{align}\label{confocal parameters}
\begin{split}
k_\lambda^2 &= \frac{a^2 - b^2}{a^2 - \lambda^2},\\
\delta_\lambda &= 2 F(\arcsin(\lambda/b); k_\lambda) = 2\int_0^{\arcsin(\lambda/b)} \frac{d\tau}{\sqrt{1 - k_\lambda^2 \sin^2 \tau}}.
\end{split}
\end{align}
Let us also denote the $4K(k_\lambda)$ periodic boundary parametrization associated to the caustic $C_\lambda$ by $q_\lambda(s)  = (- a \sn(s; k_\lambda), b \cn(s; k_\lambda))$ (a reflection about the $y$ axis of the parametrization considered in \cite{DCR17}). It is proven in \cite{ChFr88} that for all $s \in \R$ the line segment connecting $q_\lambda(s)$ and $q_\lambda(s + \delta_\lambda)$ is tangent to the caustic $C_\lambda$. In other words, $(s, \lambda)$ are precisely the action-angle coordinates from Section \ref{Billiards} and in these coordinates, the billiard map is given by a linear rotation along the invariant tori, which are the projections of $C_\lambda$ onto $B^* \d \Omega$.\\
\\
To relate these action angle coordinates to the elliptical polar coordinates in which we calculated $A_j$ (see Lemma \ref{Aj w derivatives}), note that
\begin{align*}
(x_1,x_2) = (a \cos \phi, b \sin \phi) &= (-a \sn(t_\phi;k_\lambda), b \cn(t_\phi; k_\lambda)),
\end{align*}
which implies
\begin{align}\label{phase shift}
\phi &= \am(t_\phi;k_\lambda) + \frac{\pi}{2}.
\end{align}
Here, $t_\phi$ is defined implicitly by the equations \eqref{amplitude inverse} and \eqref{phase shift}. Similarly, we find that
\begin{align*}
B_j(\phi,\omega) = \pi/2 + \am(t_\phi + (j+1) \delta_{\lambda}; k_\lambda).
\end{align*}
Differentiating $B_j$ in $\theta$, we obtain
\begin{align*}
\frac{\d B_j}{\d \omega} (\phi, \omega(\theta)) \frac{\d \omega}{\d \theta} = 1 \implies \frac{\d \omega}{\d \theta} = \frac{1}{\frac{\d B_j}{\d \omega} (\phi, \omega(\theta))}.
\end{align*}
By the chain rule, we see that
\begin{align}\label{differentiated Bj 2}
\frac{\d B_j}{ \d \omega} = \frac{\d \am}{\d s} \left( \frac{\d t_\phi}{\d \lambda^2} \frac{\d \lambda^2}{\d \omega} + (j+1) \frac{d \delta_{\lambda}}{d \lambda^2} \frac{\d \lambda^2}{\d \omega} \right) + \frac{\d \am}{\d k_\lambda^2} \frac{d k_\lambda^2}{d \lambda^2} \frac{\d \lambda^2}{\d \omega}.
\end{align}
We can factor out $\frac{\d \lambda^2}{\d \omega}$ from \eqref{differentiated Bj 2} and calculate each of the individual terms explicitly. Using the implicit function theorem, we find that
\begin{align*}
\frac{\d \am}{\d s} = \sqrt{1 - k_\lambda^2 \sn^2(s; k_\lambda)},
\end{align*}
evaluated at the point $s = t_\phi + (j+1) \delta_{\lambda}$. We now find the derivative of $t_\phi$ with respect to $\lambda$. Since $x$ is fixed, we know that the argument $\phi = \am(t_\phi; k_\lambda)$ must also be fixed. Differentiating in $\lambda$ under the integral, we see that
\begin{align*}
\frac{\d t_\phi}{\d \lambda^2} = \frac{k_\lambda^2}{(a^2 - \lambda^2)} \int_0^{\am(t_\phi;k_\lambda)} \frac{\sin^2 \tau d\tau}{(1- k_\lambda^2 \sin^2 \tau)^{3/2}}.
\end{align*}
Using the formula \eqref{confocal parameters} for $\delta_{\lambda}$, we also calculate that
\begin{align}\label{d delta d lambda}
\frac{d \delta_{\lambda}}{d \lambda^2} = \frac{1}{b \lambda \sqrt{1 - \frac{k_\lambda^2 \lambda^2}{b^2}}} \frac{1}{\sqrt{1 - \lambda^2/b^2}} +  \frac{2 k_\lambda^2}{a^2 - \lambda^2} \int_0^{\arcsin(\lambda/b)} \frac{\sin^2\tau d\tau}{(1- k_\lambda^2 \sin^2 \tau)^{3/2}}.
\end{align}
For the last term in \eqref{differentiated Bj 2}, write
$$
s = \int_0^{\am(s;k_\lambda)} \frac{d\tau}{\sqrt{1 - k_\lambda^2 \sin^2 \tau}}
$$
and differentiate both sides in the variable $k_\lambda^2$. We find that
\begin{align*}
\frac{\d \am}{\d k_\lambda^2} &= - \sqrt{1 - k_\lambda^2 \sn^2(s;k_\lambda)} \int_0^{\am(s;k_\lambda)} \frac{\sin^2\tau d\tau}{(1 - k_\lambda^2 \sin^2 \tau)^{3/2}},
\end{align*}
where $s$ is evaluated at $t_x + (j+1) \delta_{\lambda}$. It is also easy to see that
$$
\frac{d k_\lambda^2}{d \lambda^2} = \frac{k_\lambda^2}{(a^2 - \lambda^2)}.
$$
At the critical $\lambda$ corresponding to the angle ${\omega}(x)$ generating a periodic orbit at $x$, we have $\am(t_\phi + (j+1) \delta_{\lambda}) = \am(t_\phi) + 2\pi$. Using this, we calculate that
\begin{align}\label{differentiated Bj at periodic orbit}
\frac{\d B_j}{\d\omega} = \frac{\d \lambda^2}{\d \omega} \sqrt{1 - k_\lambda^2 \sn^2(t_\phi; k_\lambda)} \left(\frac{-k_\lambda^2}{(a^2 - \lambda^2)} \int_0^{2\pi} \frac{\sin^2\tau d\tau}{(1 - k_\lambda^2 \sin^2 \tau)^{3/2}} + (j+1) \frac{d \delta_\lambda}{d \lambda^2}\right).
\end{align}
We have not simplified the expression in parentheses in \eqref{differentiated Bj at periodic orbit}, since what will ultimately be important is that this term depends only on $\lambda$ and not on $\phi$ or $\omega$. Let us denote this factor by
\begin{align}\label{G of lambda}
G(\lambda) =  \frac{-k_\lambda^2}{(a^2 - \lambda^2)} \int_0^{2\pi} \frac{\sin^2\tau d\tau}{(1 - k_\lambda^2 \sin^2 \tau)^{3/2}} + (j+1) \frac{d \delta_\lambda}{d \lambda^2}.
\end{align}
The term $\d \lambda^2 / \d \omega$ is more difficult to compute and will rely on a geometric lemma, which we now present. As the ellipse is folliated by caustics, for each $(x,\omega) \in S_{\d \Omega}^*(\R^2)$, there exists a unique $\lambda$ such that the line segments of the billiard flow are always tangent to the confocal ellipse/hyperbola of parameter $\lambda$. The following lemma expresses this relationship between the confocal caustic and angle of incidence:
\begin{lemm}\label{angle caustic lemma}
The billiard ray emanating from $(a \cos \phi, b \sin \phi)$ at angle $\omega$ is tangent to the elliptical caustic $C_\lambda$, where the relationship between $\lambda$ and $\omega$ is given by
$$
\lambda^2 = \sin^2(\omega) \left (b^2 + (a^2-b^2) \sin^2(\phi) \right ).
$$
\end{lemm}
\begin{proof}
For simplicity, consider the complexified parametrization of the ellipse given by $\gamma(\phi) = a \cos \phi + i b \sin \phi$. The tangent line at $\gamma(\phi)$ is then parametrized by $$(a \cos \phi + i b \sin \phi) + t(-a \sin \phi + i b \cos \phi).$$ Rotating this line counterclockwise by the angle $\omega$, we see that the billiard ray is parameteriszed by
$$
L_{\omega}(t,\phi) = (a \cos \phi + i b \sin \phi) + t e^{i\omega} (-a \sin \phi + i b \cos \phi).
$$
Taking real and imaginary parts, we find that
\begin{align*}
\Re L_\omega(t,\phi) &= a \cos \phi + t (-a \sin \phi \cos \omega - b \cos\phi \sin \omega),\\
\Im L_\omega(t,\phi) &= b \sin \phi + t(b\cos \phi \cos \omega - a \sin \phi \sin \omega).
\end{align*}
For a given $\phi, \omega$, there exist infinitely many caustics which intersect the line $L_\omega$. However, only one such caustic intersects $L_\omega$ at a single point of tangency. To find the parameter of this caustic, we look for a solution of the equation
$$
\frac{\Re (L_\omega(t,\phi))^2}{a^2 - \lambda^2} + \frac{\Im(L_\omega(t,\phi))^2}{b^2 - \lambda^2} = 1.
$$
This is a quadratic equation in the variable $t$ and a caustic corresponding to a point of tangency will give rise to a repeated root. Thus, to find $\lambda^2$, we set the discriminant of this equation equal to zero. For convenience let us put
\begin{align*}
A = &-a \sin \phi \cos \omega - b \cos\phi \sin \omega ,\\
B = & \quad b\cos \phi \cos \omega - a \sin \phi \sin \omega .
\end{align*}
If we set the discriminant equal to zero, we obtain
\begin{align*}
\left(\frac{a A \cos \phi}{a^2-\lambda^2}+ \frac{b B \sin \phi}{b^2-\lambda^2} \right)^2 = \left (\frac{A^2}{a^2 -\lambda^2}+ \frac{B^2}{b^2-\lambda^2} \right ) \left (\frac{a^2\cos^2\phi}{a^2-\lambda^2}+ \frac{b^2\sin^2\phi}{b^2-\lambda^2}-1 \right ).
\end{align*}
After some obvious cancellations and simplifications, multiplying both sides by $(a^2-\lambda^2)(b^2 - \lambda^2)$ gives
\begin{align*}
\lambda^2 = \frac{(Ab\cos \phi +Ba \sin \phi)^2}{A^2+B^2}.
\end{align*}
Two simple computations show that 
$$
A^2+B^2= a^2 \sin^2 \phi +b^2 \cos^2 \phi
$$
and
$$
Ab \cos \phi+ Bb \sin \phi = -\sin (\omega) (a^2 \sin^2 \phi + b^2 \cos^2 \phi).
$$
Plugging these into the above equation for $\lambda^2$ completes the proof of the lemma. 

\end{proof}
Differentiating the formula in Lemma \ref{angle caustic lemma} by $\omega$ gives
\begin{align}\label{differentiated lambda}
\frac{\d \lambda^2}{\d \omega} = 2 \sin \omega \cos \omega (b^2 + (a^2 - b^2)\sin^2 \phi),
\end{align}
which is the last term in \eqref{differentiated Bj 2} we needed to calculate.
\begin{rema}
In \cite{GuMe79b}, a different formula is derived in the usual circular polar coordinates with angular parameter $\alpha$. In this case,
$$
\lambda^2 = M(\alpha)(1 - \cos (2\omega)),
$$
where
$$
M(\alpha) = \frac{a^2 + b^2}{2} - \frac{a^2b^2}{a^2 + b^2 - (a^2 - b^2) \cos(2\alpha)}.
$$
The relationship between $\alpha$ and the elliptical angular coordinate $\phi$ is given by
$$
\alpha = \arctan\left(\frac{b}{a} \tan \phi \right).
$$
\end{rema}

\section{Converting the singularity into an elliptic integral}\label{Converting the singularity into an elliptic integral}

\subsection{Proof of $\dot{\rho} = 0$}\label{rho dot zero}
It now remains to convert the singularity expansion in Corollary \ref{distribution with Aj factor} into an elliptic integral, following the ideas in \cite{GuMe79a}. Recall that according to Corollary \ref{distribution with Aj factor}, the variation of the localized wave trace is given by
$$
\delta \text{Tr} \cos t \sqrt{-\Delta_\epsilon}
 = 4t \Re \left\{ e^{\sigma i \pi/4}(t - T_j - i0^+)^{-5/2}\right\} \int_{\d \Omega_0} |\nabla^\perp \Psi_j|^2 {|A_j(q,q)|^{1/2}} \dot{\rho}(q) \,dq + L.O.T.
$$
The computations in Section \ref{d omega d theta} combined with Lemma \ref{Aj w derivatives} lead us to the formula
\begin{align}\label{full Aj term}
|A_j| = \left| \frac{ f^5}{\sin{\omega}} \frac{\d\omega}{\d \theta} \right| = \left|  \frac{ f^5}{\sin{\omega}} \frac{1}{\frac{\d \lambda^2}{\d \omega } G(\lambda) \sqrt{1 - k_\lambda^2 \sn^2(t_\phi)}  }\right|,
\end{align}
on the diagonal of the boundary, where $G(\lambda)$ is given by \eqref{G of lambda} and $\frac{\d  \lambda^2}{\d \omega}$ is given by \eqref{differentiated lambda}. The formulas \eqref{G of lambda} and \eqref{d delta d lambda} show that $G$ is in fact a nonvanishing analytic function of $\lambda$. Putting this all together, we see that the principal term in the variation of the wave trace in Corollary \ref{distribution with Aj factor} is given by the product of the distribution
$$
4t \Re \left\{ e^{\sigma i \pi /4}(t - T_j - i0^+)^{-5/2}\right\}
$$
and the factor
\begin{align}\label{factor}
c_j = \int_{\d \Omega} \frac{\sin(\omega) |f(\mu_0,\phi)|^{5/2}}{\left|2 G(\lambda)\cos(\omega) (b^2 + (a^2 - b^2)\sin^2\phi )\sqrt{1 - k_\lambda^2 \sn^2(t_\phi; k_\lambda)} \right|^{1/2}} \dot{\rho}(\phi) \,dq(\phi).
\end{align}
Since we are evaluating on the diagonal of the boundary, Poncelet's Theorem (\ref{poncelet}) guarantees that the parameter $\lambda$ corresponding to a $j$-periodic geodesic is in fact independent of $\phi$. Hence, the $G(\lambda)$ factor can be pulled outside of the integral in \eqref{factor}. However, both $\omega$, and $t_\phi$ in the integrand depend on $\phi$, so it remains to compute this dependency and parametrize the boundary explicitly. Recall that Lemma \ref{angle caustic lemma} gives us
$$
\lambda^2 = \sin^2(\omega) \left (b^2 + (a^2-b^2) \sin^2(\phi) \right ).
$$
Since $\omega \notin \{0, \pi/2, \pi\}$ for orbits which are tangent to a confocal ellipse, this equation determines the trigonometric terms appearing in the integrand of \eqref{factor} up to a sign:
\begin{align}\label{trig terms}
\begin{split}
\sin \omega &= \sqrt{\frac{\lambda^2}{b^2 + (a^2-b^2) \sin^2(\phi)}},\\
\cos \omega &= \pm \sqrt{1 - {\frac{\lambda^2}{b^2 + (a^2-b^2) \sin^2(\phi)}}}.
\end{split}
\end{align}
To simplify notation, set $C(\phi) = \left (b^2 + (a^2-b^2) \sin^2(\phi) \right )$. Recalling that in elliptic coordinates, if we fix
$$
\mu_0 = \cosh^{-1}\left(\frac{a}{\sqrt{a^2 - b^2}}\right),
$$
then we obtain a parametrization of $\d \Omega_0$ in terms of $\phi \in \R / 2\pi \Z$, given by
$$
\gamma(\phi) = (a\cos \phi, b \sin \phi).
$$
In these coordinates, the line element on $\d \Omega_0$ is
$$
dq(\phi) = \sqrt{a^2 \cos^2 \phi + b^2 \sin^2\phi} d\phi.
$$
Recalling that $\cos \phi = - \sn(t_\phi; k_\lambda)$ due to our convention \eqref{phase shift} on the phase shift, a simple computation using formula \eqref{trig terms} then shows that
$$
\cos \omega \sqrt{1 - k_\lambda^2 \sn^2(t_\phi;k_\lambda)} = \frac{1}{|C(\phi) (a^2 - \lambda^2)|^{1/2}}(C(\phi) - \lambda^2),
$$
which simplifies the integrand in \eqref{factor} to
\begin{align}
\left|\frac{\lambda^2 f(\phi)^5}{2G(\lambda)} \right|^{1/2} \left| \frac{a^2 - \lambda^2}{C(\phi)} \right|^{1/4} \frac{\gamma^*\dot{\rho}}{\sqrt{C(\phi) - \lambda^2}}.
\end{align}
Reinserting this into the integrand in \eqref{factor}, we obtain
\begin{align}\label{wave trace full expansion}
\int_0^{2\pi} \left| \frac{\lambda^2 f(\phi)^5 (a^2 \cos^2\phi + b^2 \sin^2 \phi)}{2G(\lambda)} \right|^{1/2} \left| \frac{a^2 - \lambda^2}{C(\phi)} \right|^{1/4} \frac{\gamma^*\dot{\phi} d\phi}{\sqrt{C(\phi) - \lambda^2}},
\end{align}
from which the first formula in Theorem \ref{wavetrace variation} follows.\\
\\
We now set the integral in equation \eqref{wave trace full expansion} equal to zero. Since $\lambda$ is nonzero and independent of $\phi$, we can divide out the separated terms depending only on $\lambda$ from both sides and what is left will be a nonzero analytic function $F(\phi)$ multiplied by $\dot{\rho}(\phi) (C(\phi) - \lambda^2)^{-1/2}$:
\begin{align}\label{2pi period}
\int_0^{2\pi} \frac{F(\phi) \gamma^* \dot{\rho}(\phi) d\phi}{\sqrt{C(\phi) - \lambda^2}} = 0.
\end{align}
Here, we have written
$$
F(\phi) = \left|f(\mu_0, \phi)^5 (a^2 \cos^2\phi + b^2 \sin^2 \phi) \right|^{1/2} \left| {C(\phi)} \right|^{-1/4}.
$$
Using the $\Z_2 \times \Z_2$ symmetry condition on $\dot{\rho}$, the equation \eqref{2pi period} reduces to
\begin{gather}\label{symmetrized integral}
\int_0^{\pi/2} \frac{F(\phi) \gamma^* \dot{\rho}(\phi) d\phi}{\sqrt{C(\phi) - \lambda^2}} = 0.
\end{gather}
We now recall that since the singularity expansion is localized near a simple length $T_j$, we have $\lambda = \lambda_j$ and $\lambda_j \to 0$ as $j \to \infty$. The expression \eqref{symmetrized integral} is actually analytic in the parameter $\lambda^2$ and vanishes at each $\lambda_j^2$. Since it has an accumulation point of zeros, it is actually flat at $\lambda = 0$. Differentiating $k$ times under the integral \eqref{symmetrized integral} in the parameter $\lambda^2$ and evaluating at $\lambda = 0$, we see that 
$$
\int_0^{\pi/2} \frac{F(\phi) \gamma^* \dot{\rho}(\phi) d\phi}{|C(\phi)|^{1/2 + k}} = 0.
$$
It is clear that the functions $|C(\phi)|^{-k}$ form a subalgebra of $C(S^1)$. Since we have restricted the domain to $(0,\pi/2)$, this subalgebra also separates points, and hence by the Stone-Weierstrass theorem,
\begin{align}\label{vanishing}
\frac{F(\phi) \gamma^* \dot{\rho}(\phi)}{|C(\phi)|^{1/2}} \equiv 0.
\end{align}
Since $F$ and $C$ are nonvanishing, equation \eqref{vanishing} implies that $\dot{\rho} = 0$.

\subsection{Proof of $\dot{K} = 0$}
We now return to the first variation of the Robin function $K$. According to Theorem $\ref{wavetrace}$,
$$
\delta \text{Tr} \cos t \sqrt{-\Delta_\epsilon} = \int_{\d \Omega} L_R^b(t,q,q)\dot{\rho} + \frac{t}{2}S_R(t,q,q)\dot{K}\,dq.
$$
Since we have just shown that $\dot{\rho} = 0$ in Section \ref{rho dot zero}, the variational trace formula in Theorem \ref{wavetrace} becomes
\begin{align}\label{wave trace for K dot}
\delta \text{Tr} \cos t \sqrt{-\Delta_\epsilon} = \int_{\d \Omega} \frac{t}{2}S_R(t,q,q)\dot{K}\,dq.
\end{align}
In Section \ref{Computing the singularity in elliptical polar coordinates}, we cooked up an oscillatory integral representation for $S_j$ microlocally near the canonical relations $\Gamma_\pm^j$ lying over an open neighborhood of the diagonal of the boundary (Theorem \ref{parametrix with sum}). Without the $L_R^b \dot{\rho}$ terms, we do not differentiate the sine kernel in the integrand of \eqref{wave trace for K dot}. Again following the formulas for homogeneous distributions in \cite{Ho90}, we see that near $t = T_j$,
$$
\delta \text{Tr} \cos t \sqrt{-\Delta_\epsilon} = -8t \Re\left\{e^{\sigma i \pi/4} (t - T - i0^+)^{-1/2}\right\} \int_{\d \Omega} |A_j^k|^{1/2} dq,
$$
where the minus sign is due to the appearance of the same Maslov index $\sigma$ appearing in the principal term. Plugging in the formula \eqref{full Aj term} for $A_j$ gives the second formula in Theorem \ref{wavetrace variation}. Similar computations to those in Section \ref{rho dot zero} lead us to the equation
$$
\int_0^{\pi/2} \left| \frac{f^5 (a^2 \cos^2 \phi + b^2 \sin^2 \phi) \sqrt{C(\phi)} \sqrt{a^2 -\lambda_j^2}}{2 G(\lambda) \lambda_j^2}\right|^{1/2} \frac{\dot{K}(\phi) d\phi}{\sqrt{C(\phi) - \lambda_j^2}} = 0.
$$
Discarding the separated terms depending only on $\lambda_j$ and Taylor expandng at $\lambda = 0$ as before, we see via the Stone-Weierstrass theorem that $\dot{K} = 0$, which concludes the proof of Theorem \ref{main}.
\\
\\
\begin{rema}
In \cite{HeZe12}, the principal symbol computation for Neumann boundary conditions follows closely the work of \cite{DuGu75}. The principal symbols for the Neumann and Robin wave propagators agree and the computations in \cite{HeZe12} yield
$$
\delta \text{Tr}\cos(t \sqrt{-\Delta}) \sim \frac{t}{2} \Re \left\{ \left( \sum_{\Gamma \subset F_T} C_\Gamma \int_{\Gamma}  \dot{\rho} \gamma_1\,d\mu_L \right)(t - T+i0^+)^{-{5/2}}\right\},
$$
where the sum is over connected components $\Gamma$ of the fixed point set $F_T$
and $\gamma_1 \in C(B^*\d \Omega)$ is given by $\gamma_1(q,\zeta) = \sqrt{1 - |\zeta|^2}$. The coefficients $C_\Gamma$ are nonzero Maslov factors coming from the stationary phase computation in \cite{HeZe12} and $\mu_L$ is the {Leray measure} on $\Gamma \subset B^*\d \Omega$, which is computed in \cite{GuMe79a}. In our computations, all of these factors are explicit.
\end{rema}

%\subsection{Analytic Deformations}
%In the previous sections, we showed infinessimal spectral rigidity amongst $\Z_2 \times \Z_2$ symmetric domains and Robin functions. As in \cite{HeZe12}, we note that if the deformations are $C^\infty$ in the $\epsilon$ parameter, then the higher order variations also vanish at $\epsilon = 0$. \red{To see this for $\rho$, assume the Taylor expansion is nonvanishing:
%\begin{align*}
%\rho_\epsilon(x) = \frac{1}{k!} \frac{\d^k \rho}{\d \epsilon}(x) \epsilon^k + O(\epsilon^{k+1}),
%\end{align*}
%where $k$th term has the first nonzero coefficient. Then $\epsilon \to  \epsilon^{1/k}$ is a $C^1$ reparameterization, so
%\begin{align*}
%\rho_{\epsilon^{1/k}} = \frac{1}{k!} \frac{\d^k \rho}{\d \epsilon}(x) \epsilon^1 + O(\epsilon^{1 + 1/k}).
%\end{align*}
%As the Hadamard's variational formula is valid for $C^1$ reparameterizations, the arguments above show that $\delta \rho_{\epsilon^{1/k}} = 1/k! \rho^{(k)}(x) = 0$, which contradicts our assumption. The same logic applies to the Robin function $K$. As such, we see that there are no analytic isospectral deformations of $\Omega_0$, $K_0$ through $\Z_2 \times \Z_2$ symmetric domains/Robin functions, which is precisely Corollary \ref{analytic}.}

\section{Proof of Lemma \ref{phase function}}\label{proof of lemma}
This section is dedicated to the proof of Lemma \ref{phase function}. Fix $j\in \Z$ large and choose a corresponding tubular neighborhood $U$ of $\d \Omega$ with the following property: \textit{for any $x \in U$, the $j$-fold broken geodesic emanating from $x$ which is tangent to the confocal ellipse on which $x$ lies makes less than a quarter rotation.} By this we mean that if $x = (a \cosh \mu_x \cos \phi_x, b \sinh \mu_x \sin \phi_x)$, then the $j+1$ impact point at the boundary has angular component in the interval $\phi_x \leq \phi_{x,j} \leq \phi_x + \pi/2$. This is certainly possible since the billiard flow is continuous on the closure of its phase space and if $x$ is on the boundary, the corresponding orbit is stationary. Denote by $C_{\lambda_x}$ the confocal ellipse on which $x$ lies. We will perturb this orbit by holding $x$ fixed and increasing the parameter $\lambda^2$ of the confocal caustic
$$
C_\lambda = \left\{z \in \R^2 : \frac{z_1^2}{a^2 - \lambda^2} + \frac{z_2^2}{b^2 - \lambda^2} = 1 \right\},
$$
to which the orbit is tangent. This can be done by rotating the initial covector of the trajectory slightly in any direction within $S_x^*(\R^2)$ so that it makes a nonzero angle with the tangent line to $C_{\lambda_x}$ at $x$.  As $\lambda^2$ increases, the associated angle also increases, the confocal ellipses shrink and heuristically, the $j$-fold broken geodesic begins to rotate more and more around $\Omega$. This is precisely the twist property of the billiard map: for a fixed point $x$ in the base, the straight line in phase space obtained by letting the angular component vary becomes twisted under iteration of the billiard map. If $y$ is another point which is sufficiently close to $x$, we claim that there exist $8$ angles in $S_x^*(\R^2)$ such that the last link of the corresponding billiard trajectory intersects $y$.
\\
\\
Four of these orbits will be oriented in the counterclockwise direction and four in the clockwise direction. Of the four counterclockwise orbits, two of them will correspond to rotating the initial covector in $S_x^*(\R^2)$ in the counterclockwise direction and two will result from a rotation in the opposite direction. We show that in elliptical polar coordinates, the angular components of both intersection points of the last link (after $j$ reflections) with the confocal ellipse on which $y$ lies are increasing as we rotate the inital covector within either direction in $S_x^*(\R^2)$. Hence, both intersection points will wind around the confocal ellipse until they eventually coincide with $y$. The clockwise orbits will then be constructed by a simple reflection argument.

\subsection{Notation}
Let $\lambda_x$ and $\lambda_y$ denote the parameters of the confocal ellipses on which $x$ and $y$ lie respectively. If $\lambda_x > \lambda_y$, then it is clear that the last link of any billiard emanating from $x$ will intersect $C_{\lambda_y}$ exactly twice. If $\lambda_y > \lambda_x$ but $x$ and $y$ are sufficiently close, then $\lambda_x$ and $\lambda_y$ are also close, so we can arrange that the billiard emanating from $x$ which is tangent to $C_{\lambda_y}$ makes less than a half rotation. Hence, the orbits making approximately one full rotation will necessarily have final links which intersect $C_{\lambda_y}$ twice. We only consider such $x,y$, which lie in an open neighborhood of $\Delta \d \Omega \subset \Omega \times \Omega$, where $\Delta : \d \Omega \to \d \Omega \times \d \Omega$ is the diagonal embedding.
\\
\\
The aim is to prove that the angular components of both intersection points of the last link with the caustic $C_{\lambda_y}$ are increasing. To do this, we will consider a variant of the map $B_j$ defined in Section \ref{Computing the singularity in elliptical polar coordinates}. Recall that we defined
\begin{align*}
\begin{cases}
B_j : S^1 \times (0,\pi) \to S^1,\\
B_j(\phi,\omega(\phi, \theta)) = \pi_1 \circ \beta^{j+1}(\phi, \omega) = \theta,
\end{cases}
\end{align*}
where $\beta$ is the billiard map on the coball bundle $B^*(\d \Omega)$ of the boundary. Since the coball bundle of the boundary is diffeomorphic to the collection of inward facing covectors in $S_{\d \Omega}^*(\R^2)$ and the latter is more geometrically natural, we lose no generality by considering $B_j$ or $\beta$ as a map on $S_{\d \Omega}^*(\R^2)$. We've assumed that $x, y \in \text{int} \Omega$, so we must first flow to the boundary in order to study the billiard map. Let $\alpha \in S^1$ and denote by $(x_1(\alpha),\omega_1(\alpha))$ the point obtained by evolving $(x,\alpha)$ under the forward billiard {flow} for $t_+^1(x,\alpha)$ units of time and then reflecting at the boundary. Recall from Section \ref{Billiards} that we defined
$$
t_{\pm}^1(y,\alpha) = \inf\{t > 0 : g^{\pm t}(y,\alpha) \in  \d \Omega \}.
$$
To the point $(x_1(\alpha), \omega_1(\alpha)) \in S_{\d \Omega}^*(\R^2)$ which is now fixed, we may apply $B_{j-1}$ to obtain the angular component of the $j$th impact point. Denote the corresponding boundary point by
\begin{align*}
x_j(\alpha) = (a \cos B_{j-1}(x_1(\alpha), \omega_1(\alpha)), b\sin B_{j-1}(x_1(\alpha), \omega_1(\alpha))).
\end{align*}
Denote by $\alpha_x$ the angle which the positively oriented tangent line $T_x C_{\lambda_x}$ makes with the positive $x$ axis. If $\alpha$ is sufficiently close to $\alpha_x$, the first $j$ iterates under billiard map will also make less than a quarter rotation by our original set up. Our strategy is to let $\alpha$ vary within an open cone of directions which parametrizes all possible counterclockwise orbits emanating from $x$, and show that precisely $4$ of these angles result in orbits which reach $y$ after $j$ reflections, making approximately one rotation.\\
\\
For each $\alpha$, there exists a unique $\lambda$ such that the corresponding orbit is tangent to the caustic $C_\lambda$. For example, $\lambda_x$ corresponds to $\alpha_x$.

\begin{def1}
Let $U_x^*$ be the set of $\alpha \in S_x^* (\R^2)$ such that the corresponding orbits avoid the region between the focal points. Then, denote the homogeneous extension of $U_x^*$ to $T_x^* \R^2$ by $C_x^*$, which is precisely the fiber cone at $x$ of admissible initial covectors we consider. 
\end{def1}

\begin{rema}
If $x$ varies smoothly in $\Omega$, so do the associated fibers and in this way, we obtain a smooth cone bundle, which we denote by $C^*(\Omega) \subset T^*(\Omega)$, over a tubular neighborhood of $\d \Omega$.
\end{rema}
\subsection{Derivatives of $\beta$}
If $\alpha$ increases, so does the parameter $\lambda$ of the confocal caustic, i.e. $d \lambda / d \alpha > 0$ for $\alpha > \alpha_x$. Similarly, as $\alpha$ decreases, the parameter $\lambda$ also increases, as can be seen by considering the backwards (clockwise) orbit. Hence, $\lambda_x$ is a local minimum for $\alpha$ near $\alpha_x$. By the implicit function theorem, it is clear that $\lambda$ is a smooth function of $\alpha$ as long as the corresponding forwards and backwards orbits do not enter between the focal points. Let us first consider the case in which $\alpha$ increases. For $\alpha > \alpha_x$, there is a one to one correspondence between $\alpha$ and $\lambda$. In this case, it is geometrically clear that the angular component of $x_1$ is also increasing in $\alpha$ and $\d x_1/\d \alpha$ is bounded independently of $j, \lambda$, $\omega$.
\\
\\
To obtain a large lower bound on the speed at which the two intersection points wind around $C_{\lambda_y}$, we first find a corresponding lower bound on
\begin{align}\label{Bj differentiated away from periodic orbit}
\frac{\d}{\d \lambda} B_j(x_1(\alpha(\lambda)), \omega_1(\alpha(\lambda))) = \frac{\d B_j}{\d x_1} \frac{\d x_1}{\d \alpha}\frac{ \d \alpha}{ \d \lambda} + \frac{\d B_j}{\d \omega} \frac{\d \omega}{\d \alpha} \frac{ \d \alpha}{ \d \lambda}. 
\end{align}
By a slight abuse of notation, we have systematically confused $x_1$ with its angular component $\theta(x_1)$, since they are in one to one correspondence modulo factors of $2\pi$. We begin with a simple lemma which will be used throughout the section.
\begin{lemm}\label{lambda small}
Let $x, y \in \Omega$ be $O(1/j)$ close to the diagonal of the boundary and $\gamma$ be an orbit of $j$ reflections which is tangent to the caustic $C_\lambda$ and connects $x$ to $y$. Also denote by $\omega_k$ ($1 \leq k \leq j$) the angle of reflection made at the $k$th impact point on the boundary. Then $\lambda = O(1/j)$ and $\omega_k = O(1/j)$ for all $1 \leq k \leq j$.
\end{lemm}
\begin{proof}
Recall that by $\gamma$ making approximately one rotation we mean that $|\widehat{\beta}(\widehat{x}, \omega_1) - \widehat{y} - \ell| < \ell / 100$ (see the remarks following Lemma \ref{phase function}). If in action angle coordinates (cf. Section \ref{d omega d theta}), $\widehat{x}$ is given by $q_\lambda(s) = (-a \sn (s,k_\lambda), b \cn(s, k_\lambda))$ for some $s \in \R$, we have that the link
\begin{align*}
q_\lambda(s + \delta_{\lambda}) - q_\lambda(s) = (-a \sn (s + \delta_\lambda,k_\lambda), b \cn(s + \delta_{\lambda}, k_\lambda)) -  (-a \sn (s,k_\lambda), b \cn(s, k_\lambda)),
\end{align*}
is tangent to the confocal ellipse $C_\lambda$. Lifting to the universal cover $\R \times (0,\pi)$ as in Section \ref{Billiards} and setting $\widetilde{s} \in \R / \ell \Z$ to be the arclength parameter corresponding to $s$, we have $\pi_1 \widehat{\beta}(\widetilde{s}, \omega_1) = \widetilde{s} + \delta_{\lambda}$ where $\pi_1$ is projection onto the first factor. Similarly,
\begin{align}\label{j iterates}
\pi_1 \widehat{\beta}^j(\widetilde{s},\omega_1) = \widetilde{s} + j \delta_{\lambda}.
\end{align}
If $y$ is given by $(-a \sn (t,k_\lambda), b \cn(t, k_\lambda))$ in action angle coordinates and $\widetilde{t}$ is the associated arclength parameter, then $x$ being $O(1/j)$ close to $y$ and the boundary is equivalent to requiring $|\widetilde{s} - \widetilde{t}| = O(1/j)$. Then \eqref{j iterates} implies that
\begin{align}
|\pi_1 \widehat{\beta}^j(\widetilde{s},\omega_1) - \widehat{y} - \ell| = |\widetilde{s} + j \delta_{\lambda}  - \widetilde{t} - \ell + O(1/j)| \leq \ell / 100.
\end{align}
Hence, $\delta_{\lambda} = \ell / j + O(j^{-2})$. Note that
\begin{align}
2 \arcsin \lambda/b \leq \delta_{\lambda}  = 2 \int_0^{\arcsin(\lambda/b)} \frac{d\tau}{\sqrt{1 - k_\lambda^2 \sin^2 \tau}} \leq {2}{\sqrt{1- \lambda^2 k_\lambda^2/ b}} \arcsin \lambda/b.
\end{align}
As ${\sqrt{1- \lambda^2 k_\lambda^2/ b}} = 1 + O(\lambda)$ and $\arcsin \lambda / b = \lambda/b + O(\lambda^2)$, $\delta_{\lambda} = O(1/j)$ immediately implies that $\lambda = O(1/j)$. The relationship between $\omega_1$ and $\lambda$ is given by Lemma \ref{angle caustic lemma}: $\lambda^2 = \sin^2 \omega_1 (b^2 + (a^2 - b^2) \sin^2 \widehat{x})$. As the coefficient of $\sin^2 \omega_1$ is bounded above by $a$ and below by $b$, $\lambda$ and $\omega_1$ are of the same order near zero. Furthermore, each link of $\gamma$ is tangent to $C_\lambda$ for a fixed $\lambda = O(1/j)$ and hence, the same logic also implies that $\omega_k = O(1/j)$ for each $1 \leq k \leq j$. This concludes the proof of the lemma.
\end{proof}
\noindent We are now ready to estimate \eqref{Bj differentiated away from periodic orbit}.
\begin{lemm}\label{billiard iterates}
For $j$ large, we have
\begin{align*}
\frac{\d}{\d \lambda} B_j(x_1(\alpha(\lambda)), \omega_1(\alpha(\lambda))) \gtrsim j
\end{align*}
\end{lemm}
\begin{proof}
In Section \ref{Computing the singularity in elliptical polar coordinates}, we found a rather explicit expression for
\begin{align}\label{d Bj d omega}
\frac{\d B_j}{ \d \omega} = \frac{\d \am}{\d s} \left( \frac{\d t_\phi}{\d \lambda^2} \frac{\d\lambda^2}{\d\omega} + (j+1) \frac{d \delta_{\lambda}}{d \lambda^2} \frac{\d \lambda^2}{\d \omega} \right) + \frac{\d \am}{\d k_\lambda^2} \frac{d k_\lambda^2}{d \lambda^2} \frac{\d \lambda^2}{\d \omega},
\end{align}
where $x_1 \in \d \Omega$ is fixed and $\omega$ is evaluated at the critical angle ${\omega}(x_1)$ corresponding to a periodic orbit. Using the formulas in Section \ref{Computing the singularity in elliptical polar coordinates}, we can actually let $x_1$ vary in $\alpha$ and evaluate at any $\omega$; in particular, we can evaluate at $\omega = \omega_1$ corresponding to the angle of reflection at the first impact point on the boundary. Let us first examine the term in \eqref{d Bj d omega} with a coefficient of $(j+1)$. We see that
\begin{align}
\frac{\d \am}{\d s} &= \sqrt{1 - k_\lambda^2 \sn^2(t_\phi + (j+1) \delta_{\lambda}; k_\lambda)}, \label{first term}
\\
\frac{d \delta_{\lambda}}{d \lambda^2} &= \frac{1}{b \lambda \sqrt{1 - \frac{k_\lambda^2 \lambda^2}{b^2}}} \frac{1}{\sqrt{1 - \lambda^2/b^2}} +  \frac{2 k_\lambda^2}{a^2 - \lambda^2} \int_0^{\arcsin(\lambda/b)} \frac{\sin^2\tau d\tau}{(1- k_\lambda^2 \sin^2 \tau)^{3/2}}, \label{second term}
 \\
\frac{\d \lambda^2}{\d \omega} &= 2 \sin \omega \cos \omega (b^2 + (a^2 - b^2) \sin^2(\phi)). \label{third term}
\end{align}
If $j$ is suffciently large, we can ensure that $\lambda = O(1/j)$ is in turn small. Hence, \eqref{first term} can bounded below by $(1 - b^2/a^2)/2$ and recalling Lemma \ref{angle caustic lemma}, the product of $1/\lambda$ and $\sin \omega$ coming from equations \eqref{second term} and \eqref{third term} can be estimated below by $1/a$. For $\lambda \ll 1$, all of the remaining terms can easily be bounded below by a positive constant depending only on $a$ and $b$.  Hence, we have
$$
(j+1) \frac{\d \am}{\d s} \frac{d \delta_{\lambda}}{d \lambda^2} \frac{\d \lambda^2}{\d \omega} \gtrsim j.
$$
We now consider the first term $\frac{\d \am}{\d s} \frac{\d t_\phi}{\d \lambda^2} \frac{\d\lambda^2}{\d\omega}$ in \eqref{d Bj d omega}. The first and third factors of this product are clearly bounded above and below, independently of $\omega, \lambda$, and $j$ near $\lambda = 0$ by the same arguments as before. Recall from Section \ref{Computing the singularity in elliptical polar coordinates} that
$$
\frac{\d t_\phi}{\d \lambda^2} = \frac{k_\lambda^2}{(a^2 - \lambda^2)} \int_0^{\am(t_\phi;k_\lambda)} \frac{\sin^2 \tau d\tau}{(1- k_\lambda^2 \sin^2 \tau)^{3/2}},
$$
which is also clearly bounded in magnitude by a positive constant independent of $\omega, \lambda$ and $j$.
\\
\\
In a similar manner, we would like to estimate the final term $\frac{\d \am}{\d k_\lambda^2} \frac{d k_\lambda^2}{d \lambda^2} \frac{\d \lambda^2}{\d \omega}$
in \eqref{d Bj d omega}. Recall from Section \ref{Computing the singularity in elliptical polar coordinates} that
\begin{align*}
\frac{\d \am}{\d k_\lambda^2} &= - \sqrt{1 - k_\lambda^2 \sn^2(s;k_\lambda)} \int_0^{\am(s;k_\lambda)} \frac{\sin^2\tau d\tau}{(1 - k_\lambda^2 \sin^2 \tau)^{3/2}},\\
\frac{d k_\lambda^2}{d \lambda^2} &= \frac{k_\lambda^2}{(a^2 - \lambda^2)}.
\end{align*}
Both terms can be bounded independently of $\omega, \lambda$ and $j$. Combining this with the earlier bound for \eqref{third term}, we see that
\begin{align}\label{big part}
\frac{\d B_j}{\d \omega} \gtrsim j.
\end{align}
We also need to estimate the remaining terms in \eqref{Bj differentiated away from periodic orbit}. In particular, we must bound
$$
\frac{\d \omega}{\d \alpha} \frac{\d \alpha}{\d \lambda} = \frac{\d \omega}{\d \lambda}
$$
from below by a positive constant and
$$
\frac{\d B_j}{\d x_1} \frac{\d x_1}{\d \alpha}\frac{ \d \alpha}{ \d \lambda} = \frac{\d B_j}{\d x_1} \frac{\d x_1}{\d \lambda}
$$
in magnitude. Lemma \ref{angle caustic lemma} tells us that if $0 \leq \omega \leq \pi/3$, then
\begin{align*}
\frac{\d \lambda}{\d \omega} = \frac{1}{\cos \omega} \frac{1}{\sqrt{b^2 + (a^2 - b^2)\sin^2 \phi}} \leq \frac{2}{b},
\end{align*}
so that
\begin{align}\label{product factors}
\frac{\d \omega}{\d \lambda} \geq \frac{b}{2}.
\end{align}
It is also geometrically clear that
\begin{align}\label{term 1 small}
\left|\frac{\d x_1}{\d \alpha} \frac{\d \alpha}{\d \lambda}\right| = \left|\frac{\d x_1}{\d \lambda}\right| \lesssim 1.
\end{align}
By writing out $B_j$ in action angle coordinates, we see that
\begin{align}\label{term 1 easy}
\frac{\d B_j}{\d x_1} = \frac{\d \am}{\d s} (t_\phi + (j+1) \delta_{\lambda}; k_\lambda) \frac{\d t_\phi}{\d \phi},
\end{align}
where $t_\phi$ is defined implicitly by the equation $\phi = \am(t_\phi; k_\lambda) +\pi/2$. Hence, by the implicit function theorem and our previous bound for \eqref{first term}, $\frac{\d B_j}{\d x_1}$ is both nonnegative and bounded independently of $\omega, \lambda$ and $j$. Combining \eqref{big part}, \eqref{product factors}, \eqref{term 1 small} and \eqref{term 1 easy}, we see that
\begin{align*}
\frac{\d B_j}{\d \lambda} (x_1(\alpha(\lambda)), \omega(\lambda)) \gtrsim j.
\end{align*}
\end{proof}
%\red{As $\zeta$ is fixed, this clearly implies that
%\begin{align*}
%\frac{\d \widetilde{B_j}^{(i)}}{\d \lambda} \gtrsim j.
%\end{align*}
%Letting the parameter $\zeta$ increase or decrease, it is geometrically clear that $\widetilde{B_j}$ varies independently of $\omega, j$ and $\lambda$. Hence,
%\begin{align*}
%\left| \frac{\d \widetilde{B_j}^{(i)}}{\d \zeta}\right| \lesssim 1,
%\end{align*}
%and this implies that
%\begin{align*}
%\frac{\d}{\d \lambda}\widetilde{B_j}^{(i)}(x, \lambda, \lambda) \gtrsim  j.
%\end{align*}}
%\marginpar{\blue{this is bad... gotta go. Replace with new lemma}}
%\\
%\\
%\blue{\textbf{To be added}:
%\\
%\\
%(1) Lemma on $\alpha_j$, $\lambda$ being $O(j^{-1})$
%\\
%\\
%(2) Make new subsection called intersection points
%\\
%\\
%(3) Lemma on $\d \alpha_j/\d \lambda  = O(1)$
%\\
%\\
%(4) Input new lemma (from last page) change $\alpha_j$s to $\omega_j$s in order to match old notation (since $\alpha \in S_x^* \R^2$ and $\omega$ is an angle on the boundary)
%\\
%\\
%(5) Go back and fix integral from $0$ to infinity and change $\sigma$ Maslov factors to $\pm 1$}
%\\
%\\
\subsection{Intersection points}
We are now ready to prove that the intersection points with the last link wind monotonically around the caustic $C_{\lambda_y}$ on which $y$ lies as we increase $\lambda$. 
\noindent In tandem with Lemma \ref{billiard iterates}, we will also need information about how the final angle $\omega_j$ at the $j$th reflection depends on $\lambda$.
\begin{lemm}\label{omega derivative bounded}
If $\omega_j = \pi_2\beta^j(x_1, \omega_1)$ denotes the angle of reflection at the $j$th impact point on the boundary, then $\d \omega_j / \d \lambda = O(1)$, i.e. it's derivative bounded independently of $j$.
\end{lemm}
\begin{proof}
Recall Lemma \ref{angle caustic lemma}, which gave
\begin{align*}
\lambda^2 = \sin^2 \omega_j (b^2 + (a^2 - b^2)\sin^2 x_j).
\end{align*}
Differentiating this as in the proof of Lemma \ref{billiard iterates} but now using that $x_j$ depends on $\lambda$, we see that
\begin{align}\label{differentiated angle}
2\lambda = 2 \frac{\d \omega_j}{\d \lambda} \sin \omega_j \cos \omega_j (b^2 + (a^2 - b^2) \sin^2 x_j) + 2 \sin^2\omega_j (a^2-b^2) \sin x_j \cos x_j \frac{\d x_j}{\d \lambda}.
\end{align}
Substituting again the formula in Lemma \ref{angle caustic lemma} into equation \eqref{differentiated angle}, we find that
\begin{align}\label{d omega d lambda}
\frac{\d \omega_j}{\d \lambda} = \frac{(b^2 + (a^2 - b^2) \sin^2 x_j)^{1/2} - \sin \omega_j (a^2 - b^2) \sin x_j \cos x_j \frac{\d x_j}{\d \lambda}}{(b^2 + (a^2 - b^2)\sin^2 x_j) \cos \omega_j}.
\end{align}
Lemma \ref{lambda small} tells us that $\lambda, \omega_j = O(1/j)$ and hence $\cos \omega_j = 1 + O(1/j), \sin \omega_j = O(1/j)$. This implies that the denominator of \eqref{d omega d lambda} is bounded below by a positive constant for $j$ larger than some fixed $j_0$ depending only on $\Omega$. The only unbounded term in the numerator is $\d x_j / \d \lambda = \d B_j/\d \lambda \gtrsim j$, as was shown in Lemma \ref{billiard iterates}. However, an examination of the proof of Lemma \ref{billiard iterates} actually shows that $\d x_j/\d \lambda = O(1/j)$ and hence, $\sin \omega_j \frac{\d x_j}{\d \lambda} = O(1)$, which implies that \eqref{d omega d lambda} is in fact bounded.
\end{proof}
\noindent With Lemmas \ref{lambda small}, \ref{billiard iterates} and \ref{omega derivative bounded}, we can prove half of Lemma \ref{phase function}:
\begin{lemm}\label{big bound}
For $j$ sufficiently large, the angular components in elliptical polar coordinates of both intersection points $\phi_j^1$ and $\phi_j^2$ of the final link with the caustic $C_{\lambda_y}$ are monotonically increasing in $\lambda$ with approximate speed $j$.
\end{lemm}
\begin{proof}
Recall from the proof of Lemma \ref{angle caustic lemma} that the billiard ray emanating from $x_j$ at angle $\omega_j$ is parametrized by
\begin{align}\label{Lalpha}
L(t) = (a \cos x_j + i b \sin x_j) + t e^{i\omega_j} (-a \sin x_j + i b \cos x_j).
\end{align}
Taking real and imaginary parts, we find that
\begin{align*}
\Re L_{\omega_j}(t) &= a \cos x_j + t (-a \sin x_j \cos \omega_j - b \cos x_j \sin \omega_j),\\
\Im L_{\alpha_j}(t) &= b \sin x_j + t(b\cos x_j \cos \omega_j - a \sin x_j \sin \omega_j).
\end{align*}
For simplicity, denote by $A$ the coefficient of $t$ in \eqref{Lalpha}. Converting the line $L$ from parametric form, we find that
\begin{align}\label{nonparametric}
\frac{z - a \cos x_j}{-a \sin x_j \cos \omega_j - b \cos x_j \sin \omega_j} = \frac{w - b \sin x_j}{b \cos x_j \cos \omega_j - a \sin x_j \sin \omega_j}.
\end{align}
We know that $L$ has exactly two intersection points with
$$
C_{\lambda_y} = \left\{ (z,w) \in \R^2 : \frac{z^2}{a^2 - \lambda_y^2} + \frac{w^2}{b^2 - \lambda_y^2} = 1 \right\},
$$
the caustic on which $y$ lies. These correspond to two different values of $t$ in equation \eqref{Lalpha}. At either intersection point, the values of $x, y$ in \eqref{nonparametric} are constrained to be of the form $(c \cosh \mu_j \cos \phi_j, c \sinh \mu_j \sin \phi_j)$, where $c = \sqrt{a^2 - b^2}$, $\mu_j = \cosh^{-1}((a^2 - \lambda_y^2)^{1/2}/c)$ and $\phi_j \in [0,2\pi)$ is the angular parameter in elliptical polar coordinates for the caustic $C_{\lambda_y}$. For simplicity, denote
\begin{align*}
c_j = c \cosh \mu_j = (a^2 - \lambda_y^2)^{1/2}, \quad s_j = c \sinh \mu_j = (b^2 - \lambda_y^2).
\end{align*}
To begin, assume that $\phi_j$ is either the first intersection point or the second intersection point. Only at the end of the proof will we need to distinguish between the two cases. We want to show precisely that $\phi_j' = \frac{\d}{\d \lambda} \phi_j >0$ so that the intersection points wind around $C_{\lambda_y}$. Solving for $w$ in equation \eqref{nonparametric} and substituting $z = c_j \cos \phi_j, w = s_j \sin \phi_j$, we see that
\begin{align}\label{phi j}
s_j \sin \phi_j = (c_j \cos \phi_j - a \cos x_j) \frac{\Im A}{\Re A} + b \sin x_j.
\end{align}
Differentiating \eqref{phi j} and collecting terms, we have
\begin{align}\label{key eqn}
\begin{split}
D_1 \phi_j' =  D_2 x_j' + D_3 \omega_j',
\end{split}
\end{align}
where
\begin{align*}
\begin{split}
D_1 &= (s_j \cos \phi_j \Re A + c_j \sin \phi_j \Im A)\\
D_2 &=  (-s_j \sin \phi_j f + a \sin x_j \Im A + c_j \cos \phi_j g - a \cos x_j g + b \cos x_j \Re A + b \sin x_j f ),\\
D_3 &= (- s_j \sin \phi_j \Im A + c_j \cos \phi_j  \Re A  - a \cos x_j \Re A + b \sin x_j \Im A),
\end{split}
\end{align*}
and
\begin{align*}
\begin{split}
f = (-a \cos \omega_j \cos x_j + b \sin \omega_j \sin x_j),\\
g = (- a \sin \omega_j \cos x_j - b \cos \omega_j \sin x_j).
\end{split}
\end{align*}
As both intersection points are within $O(1/j)$ of $x_j$, we have also that $|x_j - \phi_j| = O(1/j)$ and $D_3 = O(1/j)$. However, setting $\omega_j = 0$, $\phi_j = x_j$, $c_j = a$, and $s_j = b$, we see that $D_1$, $D_2$ and $D_3$ all vanish. Instead we Taylor expand each coefficient to second order in $1/j$. Recall that we are free to choose $y$ as close to the boundary as we want, so $\lambda_y = O(j^{-N})$, $c_j = a + O(j^{-N}),$ and $s_j = b + O(j^{-N})$ for any $N \in \N$. Expanding $D_1$ and simplifying, we see that
\begin{align}\label{D1}
D_1 = ab\cos \omega_j \sin (\phi_j - x_j) - \sin \omega_j (a^2\sin^2 x_j + b^2 \cos^2 x_j) + O(j^{-2})).
\end{align}
Similarly, we obtain
\begin{align}\label{D2}
%\begin{split}
%a^2(&\cos^2 x_j - \cos x_j \cos \phi_j)\sin \alpha_j + b^2 ( \sin^2 x_j - \sin x_j \sin \phi_j)\sin \alpha_j \\
%&+ ab (\cos x_j \sin \phi_j - \cos \phi_j \sin x_j) \cos \alpha_j - a^2 \sin^2 x_j - b^2 \cos^2 x_j\\
D_2 = ab \cos \omega_j \sin(\phi_j - x_j) - \sin \omega_j (a^2 \sin^2 x_j + b^2 \cos^2 x_j) + O(j^{-2}).
%\end{split}
\end{align}
and
\begin{align}\label{D3}
D_3 = a^2 \sin x_j (\cos x_j - \cos \phi_j) + b^2 \cos x_j (\sin x_j - \sin \phi_j).
\end{align}
Hence, $D_1 = D_2 +O(j^{-2})$ and $D_3 = O(1/j)$. As $x_j' \gtrsim j$ by Lemma \ref{billiard iterates}, we are done as long as we can show that $D_1, D_2$ don't vanish to first order in $1/j$ so that we may divide through by them in equation \eqref{key eqn}. If $\phi_j$ is the first intersection point, then $\phi_j = x_j + O(j^{-N})$ for any $N$ so modulo $O(j^{-2})$ terms, $D_1$ becomes $- \omega_j (a^2 \sin^2 x_j \phi + b^2\cos^2 x_j) \leq - c \omega_j \leq - c'/j$ for some constants $c, c' >0$ depending only on the curvature of $\d \Omega$. Hence, we may assume $\phi_j$ is the second intersection point, in which case choosing $y$ sufficiently close to $\d \Omega$ amounts to setting $\phi_j = x_{j+1} + O(j^{-N})$ for any desired $N$. 
\\
\\
Simplifying further, we see that modulo terms of order $O(j^{-2})$,
\begin{align}\label{D1prime}
D_1 = ab (x_{j+1} - x_j) - \omega_j (a^2\sin^2 x_j + b^2 \cos^2 x_j).
\end{align}
We want to show that the positivity of the first term in \eqref{D1prime} outweights the second term.
%using that the coefficient of $\alpha_j$ is bounded below by $b^2$.
Recall from Section \ref{d omega d theta} that $x_{j+1} - x_j = \am(s + \delta_{\lambda}, k_\lambda) - \am(s,k_\lambda)$ for some $s \in \R$ and $\lambda$ (not $\lambda_y$) corresponding to the caustic to which the orbit is tangent. Hence,
\begin{align}
({\am(s+ \delta_\lambda, k_{\lambda}) - \am(s,k_\lambda)}) \frac{1}{\sqrt{1- \lambda^2 k_\lambda^2/ b}} \geq \int_{\am(s,k_\lambda)}^{\am(s+\delta_{\lambda}, k_\lambda)} \frac{d\tau}{\sqrt{1 - k_\lambda^2 \sin^2 \tau}} = \delta_{\lambda},
\end{align}
which implies that
\begin{align}
\begin{split}
(x_{j+1} - x_j) &\geq \delta_{\lambda} \sqrt{1- \lambda^2 k_\lambda^2/ b}  = 2 \sqrt{1- \lambda^2 k_\lambda^2/ b}  \int_0^{\arcsin(\lambda/b)} \frac{d\tau}{\sqrt{1 - k_\lambda^2 \sin^2 \tau}}\\
&\geq {2}{\sqrt{1- \lambda^2 k_\lambda^2/ b}} \arcsin \lambda/b.
\end{split}
\end{align}
As $\lambda = O(1/j)$ for $j$ large, $\arcsin \lambda/ b = \lambda/b + O(j^{-2})$ and $(1- \lambda^2 k_\lambda^2 / b)^{1/2} = 1 + O(\lambda) = 1 + O(1/j)$. Hence, modulo terms of order $O(j^{-2})$, we have
\begin{align}\label{lower bound}
ab(x_{j+1} - x_j) \geq 2a\lambda.
\end{align}
Now recall Lemma \ref{angle caustic lemma}, which gives
$$
\lambda^2 = \sin^2(\omega_j) \left (b^2 + (a^2-b^2) \sin^2(x_j) \right ).
$$
Near $\omega_j = 0, \lambda = 0$, this tells us that
\begin{align*}
\omega_j = \lambda \left(b^2 + (a^2-b^2) \sin^2(x_j) \right )^{-1/2} + O(j^{-2}).
\end{align*}
The term in $D_1$ with a $\sin \omega_j$ can hence be estimated modulo $O(j^{-2})$ by
\begin{align}\label{coeff}
\omega_j (a^2\sin^2 x_j + b^2 \cos^2 x_j) = \lambda \frac{a^2\sin^2 x_j + b^2 \cos^2 x_j}{\sqrt{b^2 + (a^2 - b^2)\sin^2 x_j}} + O(j^{-2}).
\end{align}
The coefficient of $\lambda$ in \eqref{coeff} is positive and maximized at critical points corresponding to $\sin x_j = 0, \cos x_j = 0$ or
\begin{align}\label{critical point}
b^2 \cos^2 x_j = (a^2 - 2b^2) \sin^2 x_j + 2b^2.
\end{align}
If $\sin x_j = 0$, then $\alpha_j = \lambda/b + O(j^{-2})$ and the leading order coefficient of $\omega_j$ in the formula for $D_1$ is $b^2$. Hence, \eqref{lower bound} implies that
$$
ab (x_{j+1} - x_j) \geq 2 a \lambda > b\lambda = b^2 \omega_j
$$
and $D_1 > 0$ is nonvanishing. If $\cos x_j = 0$, then $\omega_j = \lambda/a + O(j^{-2})$ and the leading order coefficient of $\omega_j$ in the formula for $D_1$ is $a^2$. Hence, \eqref{lower bound} again implies that
$$
ab (x_{j+1} - x_j) \geq 2a\lambda > a \lambda = a^2 \omega_j
$$
and $D_1 > 0$ is nonvanishing. In the third case, solving \eqref{critical point} results in
$$
(a^2 - b^2) \sin^2 x_j + b^2 = 0,
$$
which impossible unless $b^2 = 0$.
\end{proof}

\noindent In this way, we obtain two orbits from $x$ to $y$ by letting $\alpha$ increase within the $\alpha \geq \alpha_x$ regime. Dynamically, these orbits can be characterized by having a point of tangency to a confocal ellipse before a moment of reflection at the boundary in the forwards direction.\\
\\
We now consider the regime in which $\alpha \leq \alpha_x$. There is a different one to one correspondence between $\alpha$ and $\lambda$ but many of the equations above remain completely valid. In particular \eqref{big part} and \eqref{product factors} are unchanged. Hence, by essentially the same bounds as in Lemma \ref{billiard iterates} before, if $\phi_j^{(1)}$ and $\phi_j^{(2)}$ denote the angular components of the intersection points in elliptical polar coordinates, we have
\begin{align*}
\frac{\d \phi_j^{(1)}}{\d \lambda} \gtrsim j, \qquad \frac{\d \phi_j^{(2)}}{\d \lambda} \gtrsim j
\end{align*}
in the $\alpha \leq \alpha_x$ regime. This provides two additional orbits, which are dynamically characterized by having a reflection at the boundary before becoming tangent to a confocal ellipse.\\
\\
To obtain the four \textit{clockwise} orbits, note that we can first apply the isometry $\R^2 \ni (z,w) \mapsto (-z, w)$, obtain four counterclockwise orbits as above, and then reflect back. In the last link of any orbit, the first point of intersection with $C_{\lambda_y}$ is reached before a point of tangency with a confocal ellipse while the second intersection point is reached after a point of tangency. These characterizations are important in Section \ref{Computing the singularity in elliptical polar coordinates} for understanding four different types of orbit configurations and determining which types of limiting orbits give rise to periodic orbits of precisely $j$ reflections as $(x,y)$ approach the diagonal of the boundary.

\begin{rema}
For sufficiently large $j$ and $x,y$ both lying \textit{on the boundary} near the diagonal, the existence of a single such geodesic for general smooth, strictly convex domains was proven in \cite{GuMeCohomological}, \cite{MaMe82} and \cite{Popov1994}. The eight orbits in the statement of Lemma \ref{phase function} can be seen to collapse into the orbits described in \cite{MaMe82} as $(x,y)$ approach the diagonal of the boundary from any direction. However, there may be a different number of reflections in the limiting orbit (see the proof of Lemma \ref{Sj lemma}).
\end{rema}

%\begin{align*}
%\sin^2 x_j = \frac{3b^2}{a^2 - b^2}
%\end{align*}
%and substitution gives
%\begin{align}
%\alpha_j (a^2 \sin^2 x_j + b^2\cos^2 x_j) = 2b \lambda + O(j^{-2}).
%\end{align}
%As long as $a > b$, we have
%$$
%ab(x_{j+1} - x_j) \geq 2a\lambda > 2b\lambda,
%$$
%which again implies $D_1$ is nonvanishing. \red{Separate proof if $a = b$ and $\Omega$ is a disk?}

\section{Acknowledgements} The author would like to thank Hamid Hezari and Katya Krupchyk for their support and suggestions during this project. The author would also like to thank Vadim Kaloshin and Alfonso Sorrentino for allowing the use of their images in Figure \ref{caustic pictures}. Additionally, Katya Krupchyk was gracious enough to provide the author with funding from NSF grant DMS 1500703 during the summer of 2016.

\nocite{HZ2}
\nocite{GuMeCohomological}
\nocite{LeviTabachnikov}
\nocite{Arnold}
\nocite{HaHiFo18}
\nocite{Zw12}
\nocite{So14}
\nocite{Ho485}

\bibliographystyle{alpha}
\bibliography{EllipseReferences}

\end{document}